\documentclass[12pt]{amsart}
\usepackage[margin=1in]{geometry}
\usepackage{cancel}
\usepackage{amssymb}
\usepackage{enumerate}
\usepackage{hyperref}
\usepackage{enumitem}
\usepackage{amsmath}
\usepackage{latexsym}
\usepackage{extarrows}
\usepackage{graphicx}
\usepackage{txfonts}
\usepackage{mathtools}
\usepackage{booktabs}
\usepackage{float}
\usepackage{bm}
\usepackage[normalem]{ulem}
\usepackage{soul}
\usepackage{tikz-cd}
\usepackage{xcolor}
\usepackage{quiver}
\usepackage{comment}
\usepackage{relsize}
\makeatletter
\@namedef{subjclassname@2020}{%
  \textup{2020} Mathematics Subject Classification}
\makeatother
\usepackage[T1]{fontenc}
\usepackage{cleveref}

\hypersetup{
    hidelinks,
    linktoc=page,
    }

\counterwithin{figure}{section}

\theoremstyle{plain}

\newtheorem{theorem}{Theorem}[section]
\newtheorem{proposition}[theorem]{Proposition}

\newtheorem{lemma}[theorem]{Lemma}
\newtheorem{corollary}[theorem]{Corollary}
\newtheorem{conjecture}[theorem]{Conjecture}

\theoremstyle{definition}
\newtheorem{definition}[theorem]{Definition}

\newtheorem{remark}[theorem]{Remark}

\newcommand\Z{\mathbb{Z}}
\newcommand\R{\mathbb{R}}
\newcommand\Q{\mathbb{Q}}
\newcommand\T{\mathbb{T}}

\newcommand\N{\mathbb{N}}

\newcommand\Y{\mathcal{Y}}

\newcommand\XX{\mathrm{X}}

\newcommand{\weight}{\mathrm{wt}}

\newcommand{\g}{\mathfrak{g}}
\newcommand{\f}{\mathfrak{f}}

\newcommand{\fl}[1]{\lfloor #1\rfloor}

\newcommand{\ad}{\mathrm{ad}}
\newcommand{\mF}{\mathcal{F}}
\newcommand{\tred}{\widetilde{\mathrm{red}}}

\begin{document}


\baselineskip=17pt


\title{Non-vanishing of multiple correlation sequences}
\author[O. Shalom]{Or Shalom}
\address{Department of Mathematics\\ Bar Ilan University \\ 
Ramat Gan \\
5290002, Israel}
\email{Or.Shalom@math.biu.ac.il}

\date{\today}

\begin{abstract}We resolve in the negative a conjecture of Frantzikinakis and Kuca \cite[Conjecture 4]{FranKuca} concerning the vanishing of multiple correlation sequences in nilsystems. Specifically, we prove the existence of a $3$-step nilsystem $(G/\Gamma, \mu_{G/\Gamma}, R_\alpha)$ and bounded functions $f_0, f_1, f_2 \in L^\infty(\mu_{G/\Gamma})$ orthogonal to the Conze--Lesigne factor $L^2(G/G_3\Gamma)$, whose associated multiple correlation sequence $$a(n) = \int_{G/\Gamma} f_0(x) f_1(\alpha^n x) f_2(\alpha^{2n} x) \, d\mu_{G/\Gamma}(x)$$ does not decay to zero. The same counterexample also refutes another conjecture of Frantzikinakis and Kuca \cite[Conjecture 3]{FranKuca} and a conjecture of Leibman \cite{Leibman1} (see Appendix~\ref{Leibmanfalse}). To construct this counterexample, we develop a framework for Fourier analysis on $G/\Gamma$ where $G$ is the free $3$-step nilpotent Lie group on $4$ generators, a methodology that extends naturally to general nilsystems.
\end{abstract}

\subjclass[2020]{Primary 37A15, 11B30; Secondary 28D15, 37A35.}

\keywords{multiple correlation sequences, nilsystems, Host--Kra seminorms, Frantzikinakis--Kuca conjecture.}

\maketitle
\setcounter{tocdepth}{1} 

\tableofcontents
\section{Introduction}
The study of multiple ergodic averages and their corresponding correlation sequences is an active area of research in modern ergodic theory and additive combinatorics. Furstenberg's foundational work ~\cite{furstenberg1977ergodic} established a key connection between dynamical systems and combinatorial number theory.
Central to this connection is the notion of a \emph{multiple correlation sequence}.

\begin{definition}[Multiple Correlation Sequence]
Let $(X, \mathcal{B}, \mu, T)$ be a measure-preserving system, where $(X, \mathcal{B}, \mu)$
is a probability space and $T: X \to X$ is an invertible measure-preserving transformation.
For a positive integer $\ell$ and functions $f_0, f_1, \dots, f_\ell \in L^\infty(\mu)$, an
$\ell$-fold multiple correlation sequence is a sequence $a: \mathbb{Z} \to \mathbb{C}$ defined
by
\begin{equation}\label{eq:corr_seq}
a(n) = \int_X f_0 \cdot T^n f_1 \cdot T^{2n} f_2 \cdots T^{\ell n} f_\ell \, d\mu.
\end{equation}
\end{definition}

The Furstenberg correspondence principle demonstrates that the structural properties of these sequences, specifically the positivity of $$\liminf_{N\rightarrow\infty}\frac{1}{N}\sum_{n=1}^Na(n)$$ when $f_0=f_1=\dots=f_{\ell}$ is a non-trivial positive function, dictate the
existence of arithmetic patterns within dense subsets of integers. Specifically, it implies an important result of Szemer\'edi \cite{szemeredi1975sets} stating that any subset of integers with positive upper
Banach density contains arbitrarily long arithmetic progressions, as well as various extensions (see e.g., \cite{bergelson1996polynomial,BHK,Franpoly,shalom2,ABB,ABS} and more).

Following this breakthrough, a major line of inquiry emerged regarding the precise structural classification of the sequences $a(n)$. It was long suspected that the asymptotic behavior of
multiple correlation sequences is governed by algebraic structures known as
\emph{nilsystems}.

\begin{definition}[Nilsystem and Nilsequence]
Let $G$ be a $d$-step nilpotent Lie group and $\Gamma$ be a discrete co-compact subgroup of
$G$. The compact homogeneous space $X = G/\Gamma$ is called a nilmanifold. Let $g \in G$ and
let $R_g: X \to X$ be the translation defined by $R_g(x\Gamma) = gx\Gamma$. The system
$(G/\Gamma, \mathcal{B}, \mu_{G/\Gamma}, R_g)$, where $\mu_{G/\Gamma}$ is the unique normalized
$G$-invariant Haar measure on $G/\Gamma$, is called a nilsystem.
A basic nilsequence is a sequence of the form $\psi(n) = F(g^n x_0)$, where $F \in C(G/\Gamma)$
and $x_0 \in G/\Gamma$. A nilsequence is any sequence that can be obtained as a uniform limit
of basic nilsequences. Similarly, $\psi$ is called a basic generalized nilsequence if $F$ is
Riemann-integrable and a generalized nilsequence is any sequence that can be obtained as a
uniform limit of basic generalized nilsequences.
\end{definition}
In the context of ergodic theory and multicorrelation sequences, these nilsystems were originally introduced by Conze and Lesigne in \cite{cl1,cl2,cl3}, and were later used by Host and Kra
\cite{host2005nonconventional} and Ziegler \cite{ziegler2007universal} who proved the following result independently:
\begin{theorem}
    Let $\XX=(X,\mathcal{B},\mu,T)$ be an ergodic invertible measure-preserving system, let $\ell\geq 1$ and $f_0,\dots,f_\ell\in L^\infty(\mu)$. If $a(n)$ is the multiple correlation sequence associated with $f_0,\dots,f_\ell$, then
    $$\lim_{N\rightarrow\infty}\frac{1}{N}\sum_{n=1}^N a(n) = 0$$ 
    whenever at least one of the $f_i$'s is orthogonal to the maximal factor of $\XX$ isomorphic to an inverse limit of $\ell$-step nilsystems (the $\ell$-step nilfactor). 
\end{theorem}
Bergelson, Host and Kra \cite[Theorem 1.9]{BHK} took this theorem one step further. Namely, they showed that in an ergodic system every multiple correlation sequence $a(n)$ decomposes as $$a(n) = \psi(n) + \omega(n)$$ where $\psi$ is an $\ell$-step nilsequence and $\omega(n)$ a null sequence (i.e., $\lim_{N\rightarrow\infty}\frac{1}{N}\sum_{n=1}^N |\omega(n)|=0$). This is particularly useful when one wishes to study averages of multiple correlation sequences. Unfortunately, for many combinatorial applications such averages are insufficient (see for example the work of Frantzikinakis and Host \cite{FranHost} on the partition regularity of quadratic equations). Various generalizations of Bergelson, Host and Kra exist, see e.g., \cite{Leibman1,Leibman2,MCS,MCFranHost,integerpart,decomposition,morejoint,FranKucajoint,Leng2025}, but they all rely on a null sequence (sometimes along a sparse sequence) as the \emph{error term} (See also \cite{ShalomGrothendieck} for a different approach.).

In \cite{Fran}, Frantzikinakis formulated a sweeping structural conjecture classifying general multiple
correlation sequences. To frame it, recall that the case $\ell=1$ is settled by classical
spectral theory: by the Herglotz theorem on positive-definite sequences, every correlation
sequence $\int_X f_0\cdot T^{n}f_1\,d\mu$ is the sequence of Fourier coefficients of a finite
complex measure on the circle. For $\ell\ge2$ no comparably transparent spectral description
is available, and supplying a structural substitute is the content of the following problem.

\begin{conjecture}[Frantzikinakis Conjecture {\cite[Problem~1]{Fran}}]\label{prob:F1}
Let $\ell\geq 1$, let $\XX=(X,\mathcal{B},\mu,T)$ be an ergodic system and let $a(n)$ be an $\ell$-fold multiple correlation sequence. Then, for every $\varepsilon>0$ there exists a complex Borel measure $\nu$ of bounded variation on a compact metric space $S$ and for all $n\in \mathbb{N}$ a measurable maps $s\mapsto \psi_s(n)\in L^\infty(\nu)$, where $\psi_s$ is a generalized nilsequence, such that $$\left|a(n) - \int_S \psi_s(n)\,d\nu(s)\right| < \varepsilon$$ for all $n\in \mathbb{N}$. In other words, $a(n)$ can be approximated in $\ell^\infty(\N)$ by integral combinations of generalized nilsequences.
\end{conjecture}
Informally, one could view generalized nilsequences as \emph{high order phases} and Frantzikinakis asks whether there exists a \emph{high order spectral measure} $\nu=\nu(f_0,\dots,f_\ell)$ on some \emph{high order spectrum} $S$ such that $a(n)$ is \emph{approximately} an integral combination of some higher order phases.

This is far stronger than the result of Bergelson--Host--Kra mentioned above, replacing the error term (the null sequence) with a sequence that is small in $\ell^\infty(\N)$. However, the cost is that we must rely on \emph{generalized nilsequences} that are less understood than just \emph{nilsequences}. In \cite{BrietGreen} Bri\"et and Green produced, already for
$\ell=2$ with iterates $T^{n},T^{2n}$, a system and functions $f_0,f_1,f_2$ whose correlation
$\int_X f_0\cdot T^{n}f_1\cdot T^{2n}f_2\,d\mu$ is not an approximate integral combination of
$2$-step nilsequences. Whether generalized nilsequences suffice remains
open. Unfortunately, to this day there are no non-trivial instances where the Frantzikinakis conjecture is known to hold for any $\ell\geq 2$.

Cutting across this qualitative picture is a quantitative one: rather than ask what $a(n)$
looks like, one asks which Gowers--Host--Kra seminorm governs its size. It is now classical, see
\cite{host2005nonconventional}, that every system carries seminorms $\|\cdot\|_{U^{s}(\XX)}$,
$s\ge1$, together with the Host--Kra factors $Z_{s}(\XX)$, linked by
$\|f\|_{U^{s+1}(\XX)}=0\iff\mathbb E(f\mid Z_{s}(\XX))=0$. For ergodic systems, $Z_s(\XX)$
is an inverse limit of $s$-step nilsystems, so $\|\cdot\|_{U^{s+1}}$ measures the failure of
$f$ to be captured by its degree-$s$ nilfactor. In the setting of
\eqref{eq:corr_seq}, and more generally a pattern $(k_1,\dots,k_\ell)$ of distinct non-zero
integers, one says the correlation is \emph{controlled by $\|\cdot\|_{U^{s}(\XX)}$} if it vanishes
in the limit whenever one of its functions has vanishing $U^{s}(\XX)$-seminorm. 

In their work on the degree-lowering method, Frantzikinakis and Kuca \cite{FranKuca}
distilled this expectation into two conjectures. The first fixes the optimal degree of
seminorm control for averages along arithmetic progressions (we will not define all the notions in this conjecture formally as they are not necessary for this paper).

\begin{conjecture}[Frantzikinakis--Kuca {\cite[Conjecture~3]{FranKuca}}]\label{conj:FK3}
Let $a\colon\mathbb N\to\Z$ be strictly increasing. If $a$ is good for seminorm control
along $\ell$-term arithmetic progressions for the system $(X,\mathcal X,\mu,T)$, then it is
good for degree-$(\ell+1)$ seminorm control along $\ell$-term arithmetic progressions for this
system.
\end{conjecture}
Moreover, Frantzikinakis and Kuca show that this conjecture implies the following pointwise vanishing statement \cite[\S3.8]{FranKuca}.

\begin{conjecture}[Frantzikinakis--Kuca {\cite[Conjecture~4]{FranKuca}}]\label{conj:FK4}
Let $(G/\Gamma,\mu_{G/\Gamma},R_\alpha)$ be an ergodic nilsystem and let
$k_1,\dots,k_\ell\in\mathbb Z$ be non-zero and distinct. Then
\[
\lim_{n\to\infty}\int_{G/\Gamma} f_0(x)\cdot f_1(\alpha^{k_1 n}x)\cdots
f_\ell(\alpha^{k_\ell n}x)\,d\mu_{G/\Gamma}(x)=0
\]
whenever $f_0,\dots,f_\ell\in L^\infty(\mu_{G/\Gamma})$ satisfy $\|f_j\|_{U^{\ell+1}(G/\Gamma)}=0$ for
some $j\in\{0,\dots,\ell\}$.
\end{conjecture}

In the case $\ell=1$, it asserts that
$\|f\|_{U^2(X)}=0$ forces $\int g\cdot f(\alpha^{n}x)\,d\mu\to0$. This case was established by Ackelsberg, Richter and the author in \cite{ARS}, and independently by Frantzikinakis and Kuca in \cite[Theorem 3.7]{FranKuca} (see also \cite{HKM,ParrySpectral,Stepin} for previous partial results). By a \emph{degree lowering argument}, Frantzikinakis and Kuca \cite{FranKuca} reduce the
general conjecture to the case of an $(\ell+1)$-step nilsystem, and confirmed it whenever that
nilsystem is a unipotent affine transformation of a torus. For genuinely non-abelian
nilsystems the question was left open. In fact, Frantzikinakis and Kuca record that it remained open
already in the first instance $\ell=2$, $(k_1,k_2)=(1,2)$ \cite[\S3.8]{FranKuca}. It is
exactly this case that the present paper settles, in the negative.

\begin{theorem}[Main Result]\label{thm:main}
Conjecture~\ref{conj:FK4} is false. For $\ell=2$ and $(k_1,k_2)=(1,2)$, there exists an ergodic
nilsystem $(G/\Gamma,\mu_{G/\Gamma},R_\alpha)$ on the free $3$-step nilpotent Lie group on four
generators, and bounded functions $f_0,f_1,f_2\in L^\infty(\mu_{G/\Gamma})$ with
$\|f_i\|_{U^3(G/\Gamma)}=0$ for all $i$, such that
\begin{equation}\label{eq:main}
a(n)=\int_{G/\Gamma} f_0(x)\cdot f_1(\alpha^{n}x)\cdot f_2(\alpha^{2n}x)\,d\mu_{G/\Gamma}(x)
\end{equation}
does not tend to $0$: there is a constant $L>0$ and a sequence $n_k\to\infty$ with
$\lvert a(n_k)\rvert\to L$.
\end{theorem}

\begin{remark}
Since Frantzikinakis and Kuca prove that Conjecture~\ref{conj:FK3} implies
Conjecture~\ref{conj:FK4} \cite[\S3.8]{FranKuca}, the same example refutes
Conjecture~\ref{conj:FK3} as well. In Appendix~\ref{Leibmanfalse} we also use this counterexample to refute a conjecture of Leibman from \cite{Leibman1}. 
\end{remark}
We give a brief computation-free proof for Theorem \ref{thm:main}. This proof should be viewed as an overview rather than a rigorous proof. The full proof, as well as optional discussions and remarks (see Remark~\ref{informal:remark}, Remark~\ref{etachoice:rem} and  Section~\ref{informal}) detailing the choices of the parameters are given throughout the rest of this paper.
\begin{proof}[Proof (sketch) of Theorem \ref{thm:main}]
    Let $\f=(\f_{4,3},+,[\cdot,\cdot])$ be the free $3$-step nilpotent Lie algebra over $\R$ on four generators denoted by $e_1,e_2,e_3$ and $e_4$, equipped with the Hall basis (see Section~\ref{Hallbasis:section}). Let $G=(\f,\ast)$ denote the Lie group associated with $\f$ (see Section~\ref{corresponding Lie algebra:section}) and let $\Gamma\leq G$ denote the subgroup generated by $e_1,e_2,e_3,e_4$. Choose a strong Mal'cev basis $\Y$ adapted to $\Gamma$ in the natural manner (see Section~\ref{strongMalcev}), let $\tau:\R^{30}\rightarrow G$ denote the coordinate map associated with this basis (second-kind coordinates) and $\mu:\R^{30}\times \R^{30}\rightarrow \R^{30}$ denote the multiplication in these coordinates (i.e., $\mu(x,x') = \tau^{-1}(\tau(x)\ast \tau(x'))$, see Theorem~\ref{secondcoor:thm}). Then $\tau$ gives rise to a natural measure-theoretic isomorphism $G/\Gamma \cong [0,1)^{30}$, where $\tau([0,1)^{30})=F\subseteq G$ is the fundamental domain (see Lemma~\ref{lem:reduceFD}). Identifying $[0,1)^{30}$ with $\T^{30}$ in the obvious manner allows us to use Fourier analysis to study the multiple correlation sequences (see Section~\ref{Fourier:section}). The measure-theoretic decomposition $G/\Gamma = G/G_2\Gamma \times G_2/G_3\Gamma_2 \times G_3/\Gamma_3$ (here $\Gamma_i = \Gamma\cap G_i$) corresponds to the decomposition $\T^{30}=\T^{4}\times \T^{6}\times \T^{20}$ and in particular, a character $\xi\in\widehat{\T^{30}}\cong \Z^{30}$ corresponds to a function $f_\xi\in L^2(\mu_{G/\Gamma})$ of zero $U^{3}(G/\Gamma)$-seminorm if and only if $\xi^{(3)} \neq0$ where $\xi=(\xi^{(1)},\xi^{(2)},\xi^{(3)}) \in \Z^{4}\times \Z^{6}\times \Z^{20}$ (see \cite{host2005nonconventional}). Let $f_0,f_1,f_2$ be the functions corresponding to the characters
    $$\xi_0:=(\chi,0,\eta)\quad \xi_1 := (0,0,-2\eta)\qquad \xi_2 :=(0,0,\eta)$$ where $\chi\in \Z^4$ is the element $(0,0,3,0)$ and $\eta\in \Z^{ 20}$ is the element whose only non-zero coordinates in the Hall basis are:
     \[
\langle\eta,e_{4,12}\rangle=2,\quad \langle\eta,e_{1,13}\rangle=-9,\quad
\langle\eta,e_{2,14}\rangle=1,\quad \langle\eta,e_{2,23}\rangle=6.
\]
Letting $\alpha\in G$ be an element of the form $\alpha = \sum_{i=1}^4 \alpha_i e_i$ and computing $\alpha^m\ast x$ in second-kind coordinates, projecting to $[0,1)^{30}\cong \T^{30}$ and evaluating by the characters $\xi_0,\xi_1,\xi_2$ we get an explicit formula for $a(n)$. Note that in practice we do not compute this formula in full in this paper, but only in the special case where it violates the Frantzikinakis--Kuca conjecture. In particular, we have:
\begin{itemize}
    \item[(1)] If $1,\alpha_1,\dots,\alpha_4$ are independent over $\Q$, then $(G/\Gamma,R_\alpha)$ is ergodic (Green's theorem \cite{AuslanderGreenHahn}).
    \item[(2)] When $\alpha\in \mathcal{V}$ where $\mathcal{V} = \{\alpha\in (0,1)^4 : \alpha_2\alpha_4 = 3\alpha_1\alpha_3 \land \alpha_1\alpha_4 = 2\alpha_2 \alpha_3 \}$ then $\eta\in \ker (\ad_{\alpha}^2)^*$,\footnote{Here $\ad_{\alpha}^2:\f\rightarrow \f_3$ is the map sending $X$ to $[\alpha,[\alpha,X]]\in \f_3$.} and for some such $\alpha$, and for all $n$ such that $\{n\alpha_i\}< \frac{1}{2}$ for all $i=1,2,4$, the sequence $a(n)$ is of the form:
    $$a(n) = c(n) \cdot \int_{0}^1 e^{2\pi i (s_3(n)+3)x_3}dx_3 \cdot \prod_{j=1,2,4} \frac{e^{2\pi i t_j(n)\cdot s_j(n)}-1}{2\pi i s_j(n)},$$ where $c(n)\in S^1$ (and thus can be ignored by taking an absolute value) and $s_1,s_2,s_3,s_4$ are given explicitly in \eqref{eq:bracketV} and $t_j(n) = 1-2\{n\alpha_j\}$.
    \end{itemize}
    Theorem~\ref{thm:transalign} is a number-theoretic component that guarantees the existence of such $\alpha\in \mathcal{V}$ (necessarily with transcendental coordinates) and a subsequence $n_k\rightarrow \infty$ satisfying $\{n_k\alpha_i\}<\frac{1}{2}$ for all $i=1,2,4$ and all $k\in \N$ and such that all the sequences $s_1(n_k),s_2(n_k),s_4(n_k)$ converge to a non-integer constant, and $s_3(n_k)= -3$, simultaneously, as $k\rightarrow\infty$. In particular, $|a(n_k)|$ approaches a product of non-zero constants which contradicts the conjecture that $a(n)\rightarrow 0$ as $n\rightarrow \infty$. 
\end{proof}
\begin{remark}
From our analysis we expect that the Frantzikinakis--Kuca conjecture holds for every $3$-step nilmanifold $G/\Gamma$ and for $\mu_G$-almost every translation $\alpha\in G$. Furthermore, if the coordinates of $\alpha$ in $G/G_2\Gamma \cong \T^{\dim(G/G_2)}$ are algebraic, then again we expect that the Frantzikinakis--Kuca conjecture holds due to Schmidt's subspace theorem (see \cite{Schmidtsubspace}). However, both of these claims are beyond the scope of this paper and are left as open questions.
\end{remark}
\section*{Acknowledgement}
The author is supported by the Alon Fellowship. I thank Florian Richter for introducing this problem to me. I thank Nikos Frantzikinakis and Borys Kuca for useful suggestions on an earlier version of this manuscript, leading in particular to Appendix~\ref{Leibmanfalse}.
\section*{The role of AI}
AI was used in this paper for the following purposes: 
\begin{itemize}
    \item[(1)] Claude (Opus 4.8) was used to program codes for various computations, some of which ended up being useful for this paper (see Appendix~\ref{app:cert}). 
    \item[(2)] Claude and Gemini (Pro 3.1) were used for copy-editing.
\end{itemize}
\section{Constructing the nilmanifold}
\subsection{The free Lie algebra on $4$ generators}\label{Hallbasis:section}
Throughout, we denote by $\f=\f_{4,3}$ the free $3$-step nilpotent Lie algebra over $\R$ on four generators denoted by $e_1,e_2,e_3$ and $e_4$. We let $$\f=\f_1 \geq \f_2 \geq \f_3 \geq \f_4=\{0\}$$ denote the lower central series of $\f$ and write $V_i = \f_i/\f_{i+1}$ for all $i=1,2,3$. For every $i$, $V_i$ is a vector space over $\mathbb{R}$ and thus by abusing notation we may write $\f= \bigoplus_{i=1}^3 V_i$ and simultaneously, view each one of these vector spaces as a subset of $\f$. We say that $X\in \f$ has weight $i$, and write $\weight(X)=i$ if $X\in V_i$. We refer to the vector spaces $V_1,V_2,V_3$ as \emph{layers}; specifically, $V_1$ is layer $1$, $V_2$ is layer $2$ and $V_3$ is layer $3$.

The dimensions of the vector spaces $V_i$ can be computed via Witt's formula (cf., \cite{Reutenauer}): $\dim V_n = \frac{1}{n} \sum_{d\mid n} \mu(d) r^{n/d},$ where $r=4$ is the number of generators. Direct computation gives
\begin{equation}
    \dim V_1 = 4, \quad \dim V_2 = 6, \quad \dim V_3 =20 \Longrightarrow\quad  \dim \f = 30.
\end{equation}
Next, we shall choose a convenient basis for our Lie algebra. We do so layer by layer.

\noindent\emph{Layer $1$:} take the generators $e_1,e_2,e_3,e_4$.

\noindent\emph{Layer $2$:} for every $1\leq i,j\leq 4$, let $e_{ij} = [e_i,e_j]$. Note that $e_{ii}=0$, and $e_{ij}=-e_{ji}$. Since the algebra is free, as a basis for the second layer we may take
$$\{e_{ij} : 1\leq i<j\leq 4\} = \{e_{12},e_{13},e_{14},e_{23},e_{24},e_{34}\}.$$

\noindent\emph{Layer $3$:} for every $1\leq i,j,k \leq 4$ write $e_{k,ij}=[e_k,[e_i,e_j]]=[e_k,e_{ij}]$. We emphasize that generally $[e_k,[e_i,e_j]]\not= [[e_k,e_i],e_j]$, however the Jacobi identity gives
$$[[e_k,e_i],e_j] = e_{i,jk}+ e_{k,ij}.$$ These $20$ brackets are linearly independent and hence, by the dimension count above, form a basis of $V_3$, known as the \emph{Hall basis} (see \cite{Reutenauer} for the general construction):
$$\{e_{k,ij} : 1\leq i < j \leq 4, k\geq i\}.$$
\begin{center}
\renewcommand{\arraystretch}{1.15}
\begin{tabular}{ll}
\toprule
pair $(i,j)$ & basic commutators $e_{k,ij}$ \quad $(k\ge i)$\\
\midrule
$(1,2)$ & $e_{1,12},\ e_{2,12},\ e_{3,12},\ e_{4,12}$\\
$(1,3)$ & $e_{1,13},\ e_{2,13},\ e_{3,13},\ e_{4,13}$\\
$(1,4)$ & $e_{1,14},\ e_{2,14},\ e_{3,14},\ e_{4,14}$\\
$(2,3)$ & $e_{2,23},\ e_{3,23},\ e_{4,23}$\\
$(2,4)$ & $e_{2,24},\ e_{3,24},\ e_{4,24}$\\
$(3,4)$ & $e_{3,34},\ e_{4,34}$\\
\bottomrule
\end{tabular}
\end{center}
Enumerating these basis elements $\mathcal{E}:=(e_1,\dots,e_{30})$ where $e_1,\dots e_4$ are the first layer, $e_5,\dots,e_{10}$ the second layer (in arbitrary order), and $e_{11},\dots,e_{30}$ the third layer, we get an ordered basis for $\f$. 
\subsection{The corresponding Lie group and lattice}\label{corresponding Lie algebra:section}
Throughout, we let $G$ be the Lie group associated with the Lie algebra $\f$. More specifically, as topological spaces we write $G=\f=\R^{30}$, and the multiplication on $G$ is given by the Baker--Campbell--Hausdorff formula. Namely,
$$X\ast Y = X+Y+\frac{1}{2}[X,Y] + \frac{1}{12}\left([X,[X,Y]]-[Y,[X,Y]]\right),$$ where $[\cdot,\cdot]$ are the Lie-brackets on $\f$. It is a classical result in Lie theory that $G=(G,\ast)$ is a connected, simply connected, $3$-step nilpotent Lie group. Its identity is $0\in G$, the inverse of $X$ is $-X$, and the exponential map (denoted by $\exp:\f=\R^{30}\rightarrow G$) is the identity. While the $\exp$ map is trivial, we may still use it to distinguish between elements in $\f$ and $G$. Furthermore, to avoid confusion, from now on we write $(g,h) = g\ast h\ast g^{-1}\ast h^{-1}$ for the Lie group commutator and keep the $[X,Y]=XY-YX$ notation for the Lie algebra bracket.\\
\noindent Next, we define $\Gamma$ to be the group generated by $e_1,e_2,e_3$ and $e_4$. Namely, $$\Gamma = \left<e_1,e_2,e_3,e_4\right>$$ with respect to the multiplication on $G$. From the construction, $\Gamma$ is the free $3$-step nilpotent Lie group on $4$ generators over $\Z$.
\begin{theorem}\label{nilmanifold}
    In the setting above, $X=G/\Gamma$ is a $3$-step nilmanifold. Namely, $G$ is a $3$-step nilpotent Lie group and $\Gamma\leq G$ is a discrete co-compact subgroup.
\end{theorem}
We will prove this theorem later (see Section~\ref{FD:section}). We stress that the Hall basis introduced above is not \emph{adapted} to $\Gamma$ (the rational coefficients in the Baker--Campbell--Hausdorff formula imply that $e_i\ast e_j$ fails to be in $\Z^{30}$). This is an issue we resolve in the next section.
\subsection{A strong Mal'cev basis adapted to $\Gamma$}\label{strongMalcev}
\begin{definition}[Mal'cev basis]
Let $\g$ be a nilpotent Lie algebra of dimension $d$. An ordered basis $\mathcal{Y}=(Y_1,Y_2,\dots,Y_d)$ is a \emph{strong Mal'cev basis} if $\g_j=\mathrm{Span}_{\R}\{Y_i : i\geq j\}$ is an ideal of $\g$ for every $j$. It is \emph{adapted to} a lattice $\Gamma$ if $$\Gamma = \{ \exp(s_1Y_1)\ast\cdots\ast\exp(s_dY_d) : (s_1,\dots,s_d)\in \Z^d \}.$$
\end{definition}
A classical result of Mal'cev provides the existence of a strong Mal'cev basis adapted to a lattice. However, here we do not need to rely on a general result as we shall construct said basis by hand. Again, we work layer by layer.

\emph{Layer $1$.} Set 
$$Y_i = e_i \qquad i=1,2,3,4.$$

\emph{Layer $2$.} In order to correct for the $\frac{1}{2}$ component, we set 
$$Y_{ij} = e_{ij}+\frac{1}{2}(e_{i,ij}+e_{j,ij}) \qquad 1\leq i<j\leq 4.$$

\emph{Layer $3$.} Fortunately, once we have accounted for the second layer, there is no need to further correct the third layer and we can take
$$Y_{k,ij} = e_{k,ij} \qquad 1\leq i<j\leq 4,\, k\geq i.$$
Enumerate $\mathcal{Y}=(Y_1,\dots,Y_{30})$, where $Y_1,\dots,Y_4$ are the first layer coordinates, $Y_5\dots Y_{10}$ are the second layer coordinates (in arbitrary order), and $Y_{11},\dots,Y_{30}$ are the third layer coordinates.
\begin{lemma}
    The ordered basis $\Y$ is a strong Mal'cev basis for $\f$.
\end{lemma}
\begin{proof}
    Let $X,Y\in \f$. Then, a direct computation gives
$$\log (\exp X,\exp Y) = \log(\exp X\ast \exp Y\ast \exp(-X)\ast \exp(-Y)) = [X,Y]+\frac{1}{2}[X+Y,[X,Y]].$$ In particular,
\begin{equation}\label{Yij}
\log(\exp(Y_i),\exp(Y_j)) = [e_i,e_j]+\frac{1}{2}[e_i+e_j,[e_i,e_j]]=Y_{ij},
\end{equation} and,
\begin{equation}\label{Ykij}
\log ((\exp(Y_k),(\exp(Y_i),\exp(Y_j))) = \log((\exp(Y_k),\exp(Y_{ij}))) = Y_{k,ij}.
\end{equation}
First, we show that $\mathcal{Y}$ is a strong Mal'cev basis. Let $j$ be arbitrary, let $\ell(j)$ denote the layer of $Y_j$ and set $\g_j = \mathrm{Span}_{\R}\{Y_j,\dots, Y_{30}\}$. From the construction $V_{>\ell(j)}:=\bigoplus_{i>\ell(j)} V_i\subseteq \g_j.$ Since the basis $\Y$ is obtained from the standard basis $\mathcal{E}$ by adding to the second layer some components that lie in the third layer, we see that $\g_j = U_j\oplus V_{>{\ell}(j)}$ where $U_j\subseteq V_{\geq\ell(j)}$ is spanned by $\{Y_i : \ell(i)=\ell(j)\land i\geq j\}\subseteq V_{\ell(j)}$. Since $[V_a,V_b]\subseteq V_{a+b}$, we see that
$$[\f,\g_j] = [\f,U_j]+[\f,V_{>{\ell}(j)}] \subseteq V_{>{\ell(j)}} \subseteq \g_{j}$$ where the last inclusion follows from the fact that $\Y$ is ordered so that the elements in the layers of weight $>\ell(j)$ come after $j$. We conclude that $\g_j$ is indeed an ideal as required.
\end{proof}
\begin{theorem}[The \emph{second-kind coordinates}]\label{secondcoor:thm}
The ordered basis $\Y$ is adapted to $\Gamma$. In particular, the map 
$$\tau:\R^{30}\to G \qquad \tau(s)=\prod_{i=1}^{30} \exp(s_iY_i)$$ is a bijection, and the group law $\mu:\R^{30}\times \R^{30}\rightarrow \R^{30}$
\begin{equation}\label{law:def}
\mu(x,x') = \tau^{-1}(\tau(x)\ast \tau(x'))
\end{equation}
is a polynomial.
\end{theorem}
\begin{proof}
Let $G=G_1\supseteq G_2\supseteq G_3\supseteq G_4=\{e\}$ be the lower central series, so that
$G_2=[G,G]=\exp(V_2\oplus V_3)$ and $G_3=[G,G_2]=\exp(V_3)$. Thus, $G_i/G_{i+1}\cong V_i$.
For $s\in\R^{30}$, write $s=(s^{(1)},s^{(2)},s^{(3)})\in\R^4\times\R^6\times\R^{20}$ for the layer-coordinates of $s$, and note that
\[
\tau(s)=\tau_1(s^{(1)})*\tau_2(s^{(2)})*\tau_3(s^{(3)}),
\]
where $\tau_i$ is the restriction of $\tau$ to the $i$-th layer (the other layers set to $0$). By
\eqref{Yij}--\eqref{Ykij} each $\tau_i$ takes values in $G_i$. Next, we analyze $\tau$ and $\mu$ layer by layer.\\
\noindent\emph{Layer $1$.} Modulo $G_2$ we have $\tau(s)\equiv\tau_1(s^{(1)})$, and in $G/G_2$,
$\tau_1(s^{(1)})=\exp(s_1Y_1+s_2Y_2+s_3Y_3+s_4Y_4)$ (because the corrections from the Baker--Campbell--Hausdorff formula lie in $G_2$). Hence
$\tau_1\mod G_2:\R^4\to G/G_2$ is a linear bijection. Since $G/G_2$ is abelian,
$$\tau(x)*\tau(x')\equiv\exp\left(\sum_a(x_a+x'_a)Y_a\right)\pmod{G_2},$$ so
\begin{equation}\label{mu1}
\mu^{(1)}(x,x')=x^{(1)}+x'^{(1)},
\end{equation}
which is linear.\\
\noindent\emph{Layer $2$.} We work modulo $G_3$. Since $G_2$ is abelian, we have $\tau_2(s^{(2)})=\exp\left(\sum_{1\leq i<j\leq 4}s_{ij}e_{ij}\right)$. Now, since  $Y_{ij}\equiv e_{ij}\pmod{\f_3}$, we conclude that
$\tau_2 \mod G_3:\R^6\to G_2/G_3$ is a bijection. Define $\tau^{<3}:\R^{10}\to G/G_3$ by
$$\tau^{<3}(s^{(1)},s^{(2)})=\tau_1(s^{(1)})*\tau_2(s^{(2)})\pmod{G_3}$$

We prove that $\tau^{<3}$ is a bijection. Let $g\in G/G_3$ and consider its image $\overline{g}=gG_2\in G/G_2$. From the surjectivity of $\tau_1 \mod{G_2}$, $\overline{g} = \tau_1(x) \pmod{G_2}$ for a unique $x\in\R^4$ (Layer~$1$). In particular, $u:=\tau_1(x)^{-1}*g\in G_2/G_3$. Now, by the surjectivity of
$\tau_2 \mod G_3$, there is some $y\in\R^6$ with $\tau_2(y)=u \pmod{G_3}$. Therefore, $\tau^{<3}(x,y)=g$, proving that $\tau^{<3}$ is onto. Next, suppose that $\tau^{<3}(x,y)=\tau^{<3}(x',y')$. Quotienting by $G_2$ gives $x=x'$, from the injectivity of Layer~$1$. Then, cancelling $\tau_1(x)$
and using injectivity of $\tau_2$ gives $y=y'$.

For the group law, $G/G_3$ is $2$-step, so $G_2/G_3$ is central in $G/G_3$. Therefore,  (modulo $G_3$) we have,
\[
\tau(x)*\tau(x')
\equiv\tau_1(x^{(1)})*\tau_2(x^{(2)})*\tau_1(x'^{(1)})*\tau_2(x'^{(2)})
\equiv\tau_1(x^{(1)})*\tau_1(x'^{(1)})*\tau_2(x^{(2)}+x'^{(2)}),
\]
using that $\tau_2(x^{(2)})$ is central and $\tau_2$ additive mod $G_3$. The group commutator satisfies
$(\exp Y_i,\exp Y_j)=\exp(Y_{ij})$, so modulo $G_3$ (where $Y_{ij}\equiv e_{ij}$),
$$
\exp(x_iY_i)*\exp(x'_jY_j)\equiv\exp(x'_jY_j)*\exp(x_iY_i)*\exp(x_ix'_j\,Y_{ij})\qquad(i<j).
$$
Next, we re-order $\tau_1(x^{(1)})*\tau_1(x'^{(1)})$ by moving each primed factor to its unprimed partner. Each step contributes a layer~$2$ bilinear form (and since we work modulo $G_3$ there are no trilinear forms involved). Thus,  
\begin{equation}\label{mu2}
\mu^{(2)}(x,x')=x^{(2)}+x'^{(2)}+P(x^{(1)},x'^{(1)}),\qquad P_{ij}=-x_j\,x'_i\ \ (i<j),
\end{equation}
with $P$ a bilinear form (hence a quadratic polynomial in the coordinates of $x^{(1)},x'^{(1)}$).\\
\noindent\emph{Layer $3$.} Now $G_3$ is central, and $\tau_3=\exp:\R^{20}\xrightarrow{\ \sim\ }V_3=G_3$
is a (linear) bijection. We finally complete the proof that $\tau$ is a bijection. Let $g\in G$, from the layer~2 analysis, there is a unique $(x^{(1)},x^{(2)})\in \R^{10}$ so that $u:=\tau^{<3}(x^{(1)},x^{(2)})^{-1}\ast g \in G_3$. Since $\tau_3$ is surjective, we can find a unique $x^{(3)}\in \R^{20}$ with $\tau_3(x^{(3)})=u$. We conclude that $x=(x^{(1)},x^{(2)},x^{(3)})\in \R^{30}$ is unique with $\tau(x)=g$, as required.\\
\noindent To compute $\mu^{(3)}$ we re-order $\tau(x)\ast\tau(x')$ into the standard form
$\tau_1(z^{(1)})\,\tau_2(z^{(2)})\,\tau_3(z^{(3)})$, keeping track of the weight-$3$ contributions.
Since $G_3$ is central, the factors $\tau_3(x^{(3)})$ and $\tau_3(x'^{(3)})$ commute with everything
and may be collected on the right, contributing the linear term $x^{(3)}+x'^{(3)}$. Namely, 
\[
\tau(x)\ast\tau(x') = \tau_1(x^{(1)})\,\tau_2(x^{(2)})\,\tau_1(x'^{(1)})\,\tau_2(x'^{(2)})
\cdot\tau_3\left(x^{(3)}+x'^{(3)}\right).
\]
We move $\tau_1(x'^{(1)})$ to the left past $\tau_2(x^{(2)})$. As $[V_1,V_2]\subseteq V_3$ and $G_3$ is
central,
\[
\tau_2(x^{(2)})\,\tau_1(x'^{(1)}) = \tau_1(x'^{(1)})\,\tau_2(x^{(2)})\cdot c,\qquad c\in G_3,
\]
where $c$ is \emph{bilinear} in $(x^{(2)},x'^{(1)})$, giving rise to a $V_2\times V_1\to V_3$ contribution. Since
$(Y_k,Y_{ij})=Y_{k,ij}$ and $(Y_{ij},Y_k)=-Y_{k,ij}$, this map is integer-valued. After this move the
two layer-$2$ factors are adjacent, and because $[V_2,V_2]\subseteq V_4=\{0\}$ they commute and
combine with no weight-$3$ correction, $\tau_2(x^{(2)})\tau_2(x'^{(2)})=\tau_2(x^{(2)}+x'^{(2)})$. In
particular, $x'^{(2)}$ produces no weight-$3$ term. It remains to re-order
$\tau_1(x^{(1)})\tau_1(x'^{(1)})$. Its weight-$3$ contributions arise from the weight-$3$ part of the
two-term commutator and from the three-term commutator, and are functions of $(x^{(1)},x'^{(1)})$
alone, contributing a bilinear $V_1\times V_1\to V_3$ and a trilinear $V_1\times V_1\times V_1\to V_3$.
Collecting the three contributions,
\begin{equation}\label{mu3}
\mu^{(3)}(x,x') = x^{(3)}+x'^{(3)}+Q\bigl(x^{(1)},x^{(2)},x'^{(1)}\bigr),
\end{equation}
a degree-$3$ polynomial in the coordinates of $x^{(1)},x^{(2)},x'^{(1)}$. By construction $Q$ enjoys
the following properties, used throughout:
\begin{itemize}
\item[(a)] $Q$ does not depend on $x'^{(2)}$, since $x'^{(2)}$ crosses no layer-$1$ factor.
\item[(b)] $\mu^{(3)}$ is affine in $x'^{(3)}$ with coefficient $1$. In particular, $Q$ is
independent of $x'^{(3)}$.
\item[(c)] $Q(x^{(1)},x^{(2)},0)=0$. When $x'^{(1)}=0$ the factor
$\tau(x')=\tau_2(x'^{(2)})\tau_3(x'^{(3)})$ has no layer-$1$ part, so it crosses no layer-$1$ factor
and contributes no weight-$3$ term. Equivalently, every term of $Q$ is at least linear in $x'^{(1)}$.
\end{itemize}
\end{proof}
\begin{remark}
    The properties of $P$ and $Q$ established above will be of use throughout the paper. These properties explain the majority of our arguments, however, at some point we will have to compute both $P$ and $Q$ explicitly in full. The latter is an enormous computation that was established via a computer code. The algorithm can be extracted from our proof above, see also Appendix~\ref{app:cert}.
\end{remark}
\begin{theorem}[The basis is adapted to $\Gamma$]\label{adaptedbasis}
    The law $\mu$ introduced in the previous theorem is integer-valued (i.e., $\mu(\Z^{30}\times \Z^{30})\subseteq \Z^{30}$). Equivalently, $\Y$ is adapted to $\Gamma$ (i.e., $\tau(\Z^{30})=\Gamma$.)
\end{theorem}
\begin{proof}
Suppose first that we have already established that $\tau(\Z^{30})=\Gamma$. Let $x,x'\in \Z^{30}$. Since $\Gamma$ is a group, $\tau(x)\ast\tau(x')\in \Gamma$ and thus, $\mu(x,x') = \tau^{-1}(\tau(x)\ast\tau(x'))\in \Z^{30}.$ Therefore, it suffices to show that $\tau(\Z^{30})=\Gamma$. We begin with the inclusion $\tau(\Z^{30})\subseteq \Gamma$. First, since $Y_{ij}=(Y_i,Y_j)$ and $Y_{k,ij}=(Y_k,Y_{ij})$, every element in $\Y$ is in $\Gamma$. Furthermore, for each $1\leq j \leq 30$, since different powers $\exp(tY_j)$ and $\exp(sY_j)$ commute for all $t,s\in \Z$, we see that $t\mapsto \exp(tY_j)$ is a homomorphism and thus its image is in $\Gamma$. Thus, every element in the image of $\tau$ is a product of elements in $\Gamma$. Since $\Gamma$ is a group we get the desired inclusion. Next, we show that $\Gamma\subseteq \tau(\Z^{30})$.  Every $\gamma\in \Gamma$ is a finite product $\gamma = e_{a_1}^{\varepsilon_1}\ast\dots\ast e_{a_r}^{\varepsilon_r}$ for some $r\in \mathbb{N}$, $1\leq a_1,\dots,a_r\leq 4$ and $\varepsilon_i\in \{-1,1\}.$ Inducting on $r$ (when $r=0$ we get $e=\tau(0)$), it suffices to show that $\tau(s)\ast e_a^{\varepsilon}\in \tau(\Z^{30})$ whenever $s\in \Z^{30}$, $a\in\{1,2,3,4\}$ and $\varepsilon\in\{-1,1\}.$ 

Since $\f_3$ is central, we have
$$\tau(s)\ast e_a^{\varepsilon} = \tau_1(s^{(1)})\ast \tau_2(s^{(2)})\ast \tau_3(s^{(3)}) \ast \exp(\varepsilon e_a)  =\tau_1(s^{(1)})\ast \tau_2(s^{(2)})\ast \exp(\varepsilon e_a)\ast\tau_3(s^{(3)}).$$
Next, from a similar computation as \eqref{Yij}, we have
$\log((\exp(s_{ij}Y_{ij}),\exp(\varepsilon Y_a)))  = -\log ((\exp (\varepsilon Y_a),(\exp(s_{ij}Y_{ij})) = -s_{ij}\cdot \varepsilon\cdot Y_{a,ij} $

Since this commutator is central, we can move the commutators to the right while moving $\exp(\varepsilon e_a)$ to the left through all of $\tau_2(s^{(2)})$. We see that
$$\tau(s)\ast e_a^\varepsilon = \tau_1(s^{(1)})\exp(\varepsilon e_a)\tau_2(s^{(2)})\cdot \tau_3(s^{(3)})\cdot \tau_3(c_1)$$ where $c_1 = -\varepsilon \sum_{1\leq i<j\leq 4} s_{ij} Y_{a,ij},$ and we see that the coefficients in $c_1$ are integers. 

It remains to analyze the layer~$1$ product $$\tau_1(s^{(1)})*\exp(\varepsilon e_a) = \exp(s_1Y_1)\exp(s_2Y_2)\exp(s_3Y_3)\exp(s_4Y_4) \ast \exp(\varepsilon e_a).$$ If $a=4$, we can use $\exp(s_4Y_4) \ast \exp(\varepsilon Y_a) = \exp((s_4+\varepsilon)Y_4)$ and complete the proof since all the coordinates are now integer valued. Otherwise, we have to move $\exp(\varepsilon Y_a)$ to its sorted slot (past $\exp(s_bY_b)$ for $b>a$). This leaves a factor in $G_2$, denoted by $C$. The factor $C$ is a product of two types of commutators which we classify here:
\begin{enumerate}
    \item[(i)] The first type commutator $(\exp(s_bY_b),\exp(\varepsilon Y_a))$ for $b>a.$
    \item [(ii)] The second type commutator is obtained when we move a commutator of the first type to the right. It takes the form
    $((\exp(s_bY_b),\exp(\varepsilon Y_a)),\exp(s_cY_c))$ with $a<c<b$.
\end{enumerate}
Since the terms $(ii)$ are central, they can be freely moved to the right without additional contributions. It is easy to see that the contributions of the form $(ii)$ are integer valued, because they arise from an iterated bracket of length $3$. Indeed, as in \eqref{Ykij}, we see that
$\log ((\exp(s_bY_b),\exp(\varepsilon Y_a)),\exp(s_cY_c)) = -(s_b\cdot s_c\cdot \varepsilon) \cdot Y_{c,ab}$, and the coefficient is an integer. The contributions of the form $(i)$ are more complicated because we need to study their weight $3$ layer. To do so we may reduce matters to the subgroup generated by $e_a$ and $e_b$ (here $a,b\leq 4$). Using the formula, $$\log (\exp X,\exp Y) = \log(\exp X\ast \exp Y\ast \exp(-X)\ast \exp(-Y)) = [X,Y]+\frac{1}{2}[X+Y,[X,Y]],$$ we may then compute the commutator and find some $w(b)$ with
$(e_b^{s_b},e_a^{\varepsilon}) = \tau(w(b))$. We can see that this holds if we take $w(b)$ whose $(ab)$-coordinate (second layer) is $-\varepsilon \cdot s_b$, whose $(b,ab)$ - coordinate (third layer) is $-\varepsilon \binom{s_b}{2}$ and whose $(a,ab)$-coordinate is $\frac{s_b (\varepsilon -1)}{2}$, and all other coordinates are $0$. Since every coordinate is an integer, this completes the proof.
\end{proof}
\begin{remark}[Distinguishing between the first and second-kind coordinates]\label{firsttosecond}
    Given $g\in G$, we can associate $g$ with an element in the Lie algebra $\f=\R^{30}$ in two distinct manners. The first-kind coordinates system is associated with the map $\log : G\rightarrow \R^{30}$, which is the inverse of the $\exp$ map (in our context this is the identity map). These coordinates are useful for certain computations. However, ultimately, we seek to understand our objects in the second-kind coordinates system which is associated with the map $\tau^{-1}:G\rightarrow \R^{30}$ from Theorem~\ref{secondcoor:thm}.
\end{remark}
\subsection{Reduction to the fundamental domain via generalized-polynomial coordinates.}\label{FD:section}
Since $\tau$ is smooth, we see that $\Gamma=\tau(\Z^{30})$ is closed and discrete in $G$. In this section we show that the fundamental domain is $\mF=\tau([0,1)^{30})$ (thus also proving Theorem~\ref{nilmanifold}). In fact, we prove a stronger result, we show that the map $\mathrm{red}:G \to \mF$ which assigns every element $g\in G$ to its representative in the fundamental domain corresponds to a map $\tred:\R^{30}\rightarrow [0,1)^{30}$ in the second-kind coordinates so that the following diagram commutes:
$$
\begin{tikzcd}[column sep=2.4cm, row sep=1.8cm]
\R^{30}
  \arrow[r, "\tau"]
  \arrow[d, "\widetilde{\mathrm{red}}"']
& G
  \arrow[d, "\mathrm{red}"]
\\
{[0,1)^{30}}
  \arrow[r, "\tau"']
& \mF
\end{tikzcd}
$$
Furthermore, for $s\in \R^{30}$, $\tred(s)$ is a generalized polynomial in the coordinates of $s = (s_1,\dots,s_{30})$ (see \cite{BLgeneralized}). Namely, it is obtained by adding the floor function $\lfloor\cdot\rfloor:\R\rightarrow \Z$ to one or more of the coordinates of $s$. (e.g., $s_1 \cdot \lfloor s_2 + s_3\lfloor s_4\rfloor\rfloor$ is a generalized polynomial.). 
\begin{lemma}[Layered reduction]\label{lem:reduceFD}
Let $g\in G$. There exists a unique $\gamma\in\Gamma$ such that $g*\gamma\in\mF$. This defines the
reduction map
\[
\mathrm{red}\colon G\to\mF,\qquad \mathrm{red}(g):=g*\gamma .
\]
Moreover, the pullback along the coordinate map $\tau\colon\R^{30}\to G$:
\[
\widetilde{\mathrm{red}}:=\tau^{-1}\circ \mathrm{red}\circ\tau\colon\R^{30}\to[0,1)^{30},\qquad
\widetilde{\mathrm{red}}(s):=\tau^{-1}\left(\mathrm{red}\bigl(\tau(s)\bigr)\right)
\]
is a generalized polynomial in the
coordinates $s=(s^{(1)},s^{(2)},s^{(3)})$.
\end{lemma}
 
\begin{proof}
Fix $g\in G$ and write $s=\tau^{-1}(g)=(s^{(1)},s^{(2)},s^{(3)})$ for its layer coordinates. We build
$\gamma$ one layer at a time, noting at each stage that the coordinates produced are generalized
polynomials in $s$.
 
\emph{Layer $1$.} Put $\rho_1(s):=(-\lfloor s^{(1)}\rfloor,0,0)$. Since $\rho_1(s)\in \Z^{30}$ we have that $\tau(\rho_1(s))\in\Gamma$, and
let $s':=\tau^{-1}\bigl(g*\tau(\rho_1(s))\bigr)=\mu\bigl(s,\rho_1(s)\bigr)$. From \eqref{mu1} we see that,
$s'^{(1)}=s^{(1)}-\lfloor s^{(1)}\rfloor\in[0,1)^4$, while from \eqref{mu2}--\eqref{mu3} the remaining
coordinates $s'^{(2)},s'^{(3)}$ are generalized polynomials in $s$.
 
\emph{Layer $2$.} Put $\rho_2(s):=(0,-\lfloor s'^{(2)}\rfloor,0)$ and let
$s'':=\tau^{-1}\bigl(g*\tau(\rho_1(s))*\tau(\rho_2(s))\bigr)=\mu\bigl(s',\rho_2(s)\bigr)$. As $\rho_2(s)$
has no layer-$1$ part, \eqref{mu1} gives $s''^{(1)}=s'^{(1)}\in[0,1)^4$. Moreover, since $P$ is bilinear with
$P(\cdot,0)=0$, \eqref{mu2} gives $s''^{(2)}=s'^{(2)}-\lfloor s'^{(2)}\rfloor\in[0,1)^6$. Thus, the
first ten coordinates of $s''$ lie in $[0,1)$, and $s''^{(3)}$ is a generalized polynomial in $s$.
 
\emph{Layer $3$.} Finally, put $\rho_3(s):=(0,0,-\lfloor s''^{(3)}\rfloor)$ and let
$s''':=\tau^{-1}\bigl(g*\tau(\rho_1(s))*\tau(\rho_2(s))*\tau(\rho_3(s))\bigr)=\mu\bigl(s'',\rho_3(s)\bigr)$.
Since $V_3$ is central, the arguments $\rho_3(s)^{(1)}=\rho_3(s)^{(2)}=0$ make the mixed term in
\eqref{mu3} vanish, so $\mu^{(3)}$ reduces to addition in the top layer: layers $1,2$ are unchanged and
$s'''^{(3)}=s''^{(3)}-\lfloor s''^{(3)}\rfloor\in[0,1)^{20}$. Hence $s'''\in[0,1)^{30}$, i.e.,
$g*\gamma\in\mF$ with $\gamma:=\tau(\rho_1(s))*\tau(\rho_2(s))*\tau(\rho_3(s))\in\Gamma$.

It is left to prove the uniqueness claim.  Suppose that $\gamma,\gamma'\in\Gamma$ both satisfy
$g\ast\gamma\in\mF$ and $g\ast\gamma'\in\mF$. Write $s=\tau^{-1}(g)$ and put
$t:=\tau^{-1}(g\ast\gamma)=\mu(s,\tau^{-1}(\gamma))$ and
$t':=\tau^{-1}(g\ast\gamma')=\mu(s,\tau^{-1}(\gamma'))$, both lying in
$[0,1)^{30}$ since $g\ast\gamma,g\ast\gamma'\in\mF$. Since
$g\ast\gamma'=(g\ast\gamma)\ast(\gamma^{-1}\ast\gamma')$ and
$\gamma^{-1}\ast\gamma'\in\Gamma=\tau(\Z^{30})$, the element
$w:=\tau^{-1}(\gamma^{-1}\ast\gamma')$ lies in $\Z^{30}$ and
$t'=\mu(t,w)$. We show $w=0$ layer by layer.

By \eqref{mu1}, $t'^{(1)}=t^{(1)}+w^{(1)}$ with $t^{(1)},t'^{(1)}\in[0,1)^4$ and
$w^{(1)}\in\Z^4$. Hence, $w^{(1)}=0$ and $t'^{(1)}=t^{(1)}$. As $P$ is bilinear with
$P(\cdot,0)=0$, \eqref{mu2} now gives $t'^{(2)}=t^{(2)}+w^{(2)}$ with
$t^{(2)},t'^{(2)}\in[0,1)^6$ and $w^{(2)}\in\Z^6$, forcing $w^{(2)}=0$ and
$t'^{(2)}=t^{(2)}$. Finally, since $w^{(1)}=0$, and $Q(\cdot,\cdot,0)=0$, the cross term in \eqref{mu3} vanishes, so
$t'^{(3)}=t^{(3)}+w^{(3)}$ with $t^{(3)},t'^{(3)}\in[0,1)^{20}$ and
$w^{(3)}\in\Z^{20}$, forcing $w^{(3)}=0$. Thus $w=0$, i.e.,
$\gamma^{-1}\ast\gamma'=e$ and $\gamma=\gamma'$.
\end{proof}

Next, we identify $G/\Gamma$ with the fundamental domain $\mathcal{F}$.
\begin{lemma}[A measure-theoretic isomorphism]\label{lem:identification}
Let $\pi\colon G\to X:=G/\Gamma$ denote the natural quotient map. The restrictions
\[
\Phi:=\tau^{-1}|_{\mF}\colon\mF\to[0,1)^{30},\qquad
\beta:=\pi|_{\mF}\colon\mF\to X
\]
are bijections. Moreover, the map $\Psi:=\Phi\circ\beta^{-1}\colon X\to[0,1)^{30}$ is a measure-theoretic isomorphism, where $[0,1)^{30}$ is equipped with the Lebesgue measure $m_{[0,1)^{30}}$.
\end{lemma}
 
\begin{proof}
Since $\mF=\tau([0,1)^{30})$ and $\tau$ is a bijection, we immediately get that $\Phi$ is also a bijection. The fact that $\beta$ is a bijection also follows immediately from the previous lemma. Thus, $\Psi$ is a bijection as a composition of bijections. It is left to show that $\Psi$ takes $\mu_{G/\Gamma}$ to the Lebesgue measure on $[0,1)^{30}.$ Let $m_{\R^{30}}$ denote the Lebesgue measure on $\R^{30}$. We first show that $\mu_G:=\tau_* m_{\R^{30}}$ is a Haar measure on $G$, i.e., that it is
left-invariant. For $s\in\R^{30}$ let $L_s\colon\R^{30}\to\R^{30}$ denote left-translation by
$\tau(s)$ read in second-kind coordinates, $L_s(s'):=\mu(s,s')$, so that
$\tau\bigl(L_s(s')\bigr)=\tau(s)\ast\tau(s')$. By \eqref{mu1}, \eqref{mu2} and \eqref{mu3}, each layer
of the group law has the form
\[
\mu^{(i)}(s,s')=s'^{(i)}+R_i\bigl(s,\,s'^{(1)},\dots,s'^{(i-1)}\bigr)\qquad(i=1,2,3),
\]
so $\mu^{(i)}$ depends on $s'$ through $s'^{(i)}$, with coefficient the identity, and otherwise only
through the strictly lower layers $s'^{(1)},\dots,s'^{(i-1)}$. Ordering the coordinates of
$\R^{30}=\R^{4}\times\R^{6}\times\R^{20}$ by layer, the Jacobian of $L_s$ in the variable $s'$ is
therefore block lower-triangular with identity diagonal blocks,
\[
D_{s'}L_s=\begin{pmatrix} I_4 & 0 & 0\\ \ast & I_6 & 0\\ \ast & \ast & I_{20}\end{pmatrix},
\qquad\text{so}\qquad \det D_{s'}L_s\equiv 1.
\]
By the change-of-variables formula $L_s$ preserves $m_{\R^{30}}$. Hence, for $g=\tau(s)$ and any Borel
set $B\subseteq\R^{30}$, using $g\ast\tau(B)=\tau\bigl(L_s(B)\bigr)$,
\[
\mu_G\bigl(g\ast\tau(B)\bigr)=m_{\R^{30}}\bigl(L_s(B)\bigr)=m_{\R^{30}}(B)=\mu_G\bigl(\tau(B)\bigr).
\]
As $\tau$ is a bijection, every Borel subset of $G$ has the form $\tau(B)$, so $\mu_G$ is
left-invariant. Moreover, since $G$ is nilpotent, it is unimodular, and $\mu_G$ is the (bi-invariant) Haar
measure on $G$. By construction, it coincides with the Lebesgue measure in the second-kind coordinates.

Next, observe that $\mu_{G/\Gamma}$ is exactly the push-forward of $\mu_G|_{\mF}$ under $\beta$. Thus, 
$$\Psi_* \mu_{G/\Gamma} = \Phi_* \beta^{-1}_* \mu_{G/\Gamma} = \Phi_* (\mu_G|_{\mF})$$ where $\mu_G|_{\mF}$ is the restriction of $\mu_G$ to $\mF$ (defined by assigning the measure zero to anything outside of $\mF$). Now, $\Phi = \tau^{-1}|_\mF$ while $\mF=\tau([0,1)^{30})$. We see that
$$\Phi_*(\mu_G|_{\mF}) = (\tau^{-1})_* \tau_*(m_{\R^{30}}|_{[0,1)^{30}}) = m_{[0,1)^{30}},$$ as required.
\end{proof}
\begin{remark}[Fourier analysis on nilmanifolds]
   By the uniqueness in the previous lemma, the map $\mathrm{red}:G \to \mathcal{F}$ factors through $G/\Gamma$. This shows in particular that $G/\Gamma$ is compact. Furthermore, identifying $\mathcal{F}$ with $[0,1)^{30}$ via $\tau^{-1}$, we may view every function in $L^\infty(\mu_{G/\Gamma})$ as a function in $L^\infty(\mu_{[0,1)^{30}})$. Since the push-forward of $\mu_G$ to $[0,1)^{30}$ via $\tau^{-1}\circ \mathrm{red}$ is the Lebesgue measure, $\mu_{[0,1)^{30}}$ is the Haar measure. Identifying $[0,1)^{30}$ with $\T^{30}$ in the obvious manner gives us an orthonormal basis of characters (hence Fourier analysis). This theory extends naturally to all nilmanifolds $G/\Gamma$ with $G$ a connected and simply connected Lie group. However, we stress that this construction is dependent on the choice of basis.
\end{remark}
Fourier analysis for nilmanifolds differs greatly from the abelian setting. In particular, in the latter the Fourier characters are eigenfunctions with respect to action by rotations. In the next sections we will study the action of some $\alpha\in G$ on the Fourier characters under the coordinate maps we just developed. This will show, roughly speaking, that the characters of $\T^{30}$ correspond to some \emph{higher order eigenfunctions} where the eigenvalue depends on lower-layer coordinates. This dependence is computed in the next sections.
\section{The orbit $\alpha^n\ast x$}\label{orbit:section}
The exact choice of $\alpha\in G$ which violates the Frantzikinakis--Kuca conjecture is crucial, but we postpone the choice till later. For now, we fix an arbitrary $\alpha\in G$, and for simplicity we shall further assume that $\alpha=\prod_{i=1}^4  \exp(\alpha_iY_i)$ for some $\alpha_1,\alpha_2,\alpha_3,\alpha_4\in [0,1)$ to be chosen later. Note that for ergodicity we need that $1,\alpha_1,\alpha_2, \alpha_3$ and $\alpha_4$ are independent over $\mathbb{Q}$ which we will assume throughout.

Our goal is to compute the second-kind coordinates for the powers of $\alpha$. This section is highly technical and consists of a sequence of direct computations. Nevertheless, these computations are necessary for deducing our results formally.

 Let $\Lambda = \log \alpha$ denote the first-kind coordinates of $\alpha$ (see Remark~\ref{firsttosecond}). Observe that in this case $\alpha^n = \exp(n\Lambda)$ for every $n\in \Z$. Thus, it will be convenient to start by writing down $\Lambda$ explicitly. Write $\Lambda = \Lambda^{(1)}+\Lambda^{(2)}+\Lambda^{(3)}$ for the layers of $\Lambda$, as usual. From the Baker--Campbell--Hausdorff formula we have
$$\Lambda^{(1)} = \sum_{i=1}^4 \alpha_ie_i  \qquad \Lambda^{(2)} = \frac{1}{2}\sum_{1\leq i<j\leq 4} \alpha_i\alpha_j e_{ij}$$ and finally we have:
\begin{align*}
\Lambda^{(3)}=\ 
&\tfrac1{12}\alpha_1^2\alpha_2\,e_{1,12}-\tfrac1{12}\alpha_1\alpha_2^2\,e_{2,12}-\tfrac13\alpha_1\alpha_2\alpha_3\,e_{3,12}-\tfrac13\alpha_1\alpha_2\alpha_4\,e_{4,12}\\
&+\tfrac1{12}\alpha_1^2\alpha_3\,e_{1,13}+\tfrac16\alpha_1\alpha_2\alpha_3\,e_{2,13}-\tfrac1{12}\alpha_1\alpha_3^2\,e_{3,13}-\tfrac13\alpha_1\alpha_3\alpha_4\,e_{4,13}\\
&+\tfrac1{12}\alpha_1^2\alpha_4\,e_{1,14}+\tfrac16\alpha_1\alpha_2\alpha_4\,e_{2,14}+\tfrac16\alpha_1\alpha_3\alpha_4\,e_{3,14}-\tfrac1{12}\alpha_1\alpha_4^2\,e_{4,14}\\
&+\tfrac1{12}\alpha_2^2\alpha_3\,e_{2,23}-\tfrac1{12}\alpha_2\alpha_3^2\,e_{3,23}-\tfrac13\alpha_2\alpha_3\alpha_4\,e_{4,23}\\
&+\tfrac1{12}\alpha_2^2\alpha_4\,e_{2,24}+\tfrac16\alpha_2\alpha_3\alpha_4\,e_{3,24}-\tfrac1{12}\alpha_2\alpha_4^2\,e_{4,24}\\
&+\tfrac1{12}\alpha_3^2\alpha_4\,e_{3,34}-\tfrac1{12}\alpha_3\alpha_4^2\,e_{4,34}.
\end{align*}
\begin{remark}
    Observe that each of the coefficients of each $e_{k,ij}$ is one of the following four: $\tfrac1{12}\alpha_i^2\alpha_j$,
$\ -\tfrac1{12}\alpha_i\alpha_j^2$, $\ +\tfrac16\alpha_i\alpha_j\alpha_k$, or
$-\tfrac13\alpha_i\alpha_j\alpha_k$. This observation will be used later when we choose $\alpha_1,\dots,\alpha_4.$
\end{remark}

Our next goal is to compute $\alpha^n$ in the second-kind coordinates.
\begin{corollary}\label{cor:secondkind}
Write $s(\alpha^n) = \tau^{-1}(\alpha^n)$. Then $n\mapsto s(\alpha^n)$ is a polynomial in $n$ of degree equal to the layer weight. More specifically, we have
$$s^{(1)}(\alpha^n) = n\cdot \alpha \qquad s^{(2)}(\alpha^n) = -\binom{n}{2} \sum_{i<j} \alpha_i\alpha_j e_{ij},$$ and the layer $3$ component is a cubic:
$$s^{(3)} (\alpha^n) = (n-n^3)\Lambda^{(3)} - \frac{n-n^2}{4} S_1 - \frac{n^2-n^3}{4} S_2$$ where $\Lambda^{(3)}$ is as above, and 
$$S_1:=\sum_{i<j} \alpha_i \alpha_j (e_{i,ij} + e_{j,ij}), \qquad S_2 = \sum_{i<j} \alpha_i \alpha_j [\alpha,e_{ij}]=2[\alpha,\Lambda^{(2)}].$$
\end{corollary}
Again, we notice an interesting phenomenon, the coefficient of third layer $e_{k,ij}$ is: 
\begin{itemize}
    \item[(1)] In $S_1$, it is $\alpha_i \alpha_j$ if $k\in \{i,j\}$ and zero otherwise.
    \item[(2)] In $S_2$, it is $\alpha_i\alpha_j\alpha_k$ if $k\in \{i,j\}$, $2\alpha_i\alpha_j\alpha_k$ if $i<k<j$ and zero otherwise.
\end{itemize}
\begin{proof} Recall that $\alpha=\exp \Lambda$ and that $\alpha^n = \exp(n\Lambda)$. Thus, $s(\alpha^n) = \tau^{-1}(\exp(n\Lambda))$ is the unique $s\in \R^{30}$ with $\tau(s)=\exp(n\Lambda)$. To compute this, we first compute $\log \tau (s)$ for some $s=(s^{(1)},s^{(2)},s^{(3)})$ and then plug in the coordinates we have already computed for $n\Lambda$.

Recall the decomposition $\tau(s) = \tau_1(s^{(1)})\ast \tau_2(s^{(2)})\ast \tau_3(s^{(3)})$ where
$$\tau_1(s^{(1)}) = \prod_{i=1}^4 \exp\left(s_iY_i\right),\quad \tau_2(s^{(2)}) = \prod_{1\leq i<j\leq 4} \exp\left(s_{ij}Y_{ij}\right),\quad \tau_3(s^{(3)}) = \prod_{\{(k,ij) : 1\leq i<j\leq 4,\, k\geq i\}} \exp\left(s_{k,ij}Y_{k,ij}\right).$$
Since $\f_3$ is central, we have that $$\log \tau_3(s^{(3)}) = \sum_{\{(k,ij) : 1\leq i<j\leq 4,\, k\geq i\}} s_{k,ij}e_{k,ij},$$ and since $\f_2$ is abelian, we have $$\log \tau_2(s^{(2)}) = \sum_{1\leq i<j\leq 4} s_{ij} Y_{ij} = \underbrace{\sum_{i<j}s_{ij}e_{ij}}_{\rho} + \frac{1}{2} \sum_{1\leq i<j\leq 4} s_{ij}(e_{i,ij}+e_{j,ij}).$$
Finally, we wish to compute $\log \tau_1(s^{(1)})$. Since $\tau_1(s^{(1)})=\prod_{i=1}^4 \exp(s_i Y_i)$, from the previous analysis of $\Lambda$, we get that
$$\log \tau_1(s^{(1)}) = \underbrace{\sum_{i=1}^4 s_ie_i}_{\sigma} + \frac{1}{2}\sum_{i<j} s_is_j e_{ij} + \Lambda^{(3)}(s^{(1)}).$$

Next, we combine everything together to compute $\log\tau(s).$ Since $\log \tau_3(s^{(3)})$ is central, $\log \tau(s) = \mathrm{BCH}(\log\tau_1(s^{(1)}),\log\tau_2(s^{(2)}))+\log \tau_3(s^{(3)})$, where $\mathrm{BCH}(X,Y) = X+Y+\frac{1}{2}[X,Y]+\frac{1}{12}([X,[X,Y]]-[Y,[X,Y]])$. Fortunately, $\tau_2(s^{(2)})$ is of weight $2$, which simplifies the equation as the last summands in the $\mathrm{BCH}$ formula vanish. Furthermore, observe that $\log \tau_2(s^{(2)})$ is a sum of a term of weight $2$ (denoted by $\rho$) and a term of weight $3$, and the term of weight $3$ also contributes nothing. Similarly, $\log \tau_1(s^{(1)})$ is a sum of a term of weight $1$ (denoted by $\sigma$) and a term of weight $2$ which does not contribute to the sum. Thus,
$$\log \tau(s) = \log \tau_1(s^{(1)}) + \log \tau_2(s^{(2)}) + \log \tau_3(s^{(3)}) + \frac{1}{2}[\sigma,\rho].$$
From this point the proof is a direct computation (plugging in the values of $n\Lambda$), see also Appendix~\ref{app:cert}.
\end{proof}
Next, we fix some $x\in G$. Throughout the rest of the paper we will analyze several properties on the coordinates of the representative of $\alpha^n\ast x$ in $[0,1)^{30}$. The following equations will be useful: Write $W=\log x=W^{(1)}+W^{(2)}+W^{(3)}$ and recall that $\alpha^n = \exp(n\Lambda)$. For $X\in \f$ we write $\ad_X:\f\rightarrow \f$ for the map $\ad_X(Y)=[X,Y]$. In particular, $\ad_X^2(Y) = [X,[X,Y]].$ Splitting by layer, with $\alpha=\Lambda^{(1)}$, we have \begin{equation}\label{eq:orbit}
\log(\alpha^n\ast x)=\mathrm{BCH}(n\Lambda,W)
= n\Lambda+W+\tfrac{n}{2}[\Lambda,W]+\tfrac{n^2}{12}[\Lambda,[\Lambda,W]]-\tfrac{n}{12}[W,[\Lambda,W]],
\end{equation}
a polynomial in $n$ of degree $2$. Splitting by
layer, with $\alpha=\Lambda^{(1)}$,
\begin{align*}
[\log(\alpha^n\ast x)]^{(1)} &= W^{(1)}+n\,\alpha,\\
[\log(\alpha^n\ast x)]^{(2)} &= W^{(2)}+n\,\Lambda^{(2)}+\tfrac{n}{2}[\alpha,W^{(1)}],\\
[\log(\alpha^n\ast x)]^{(3)} &= W^{(3)}+n\,\Lambda^{(3)}
+\tfrac{n}{2}\bigl([\alpha,W^{(2)}]+[\Lambda^{(2)},W^{(1)}]\bigr)
+\tfrac{n^2}{12}\,\ad_\alpha^2 W^{(1)}
-\tfrac{n}{12}\,[W^{(1)},[\alpha,W^{(1)}]].
\end{align*}
Later, we will compose these coordinates with the map $\mathrm{red}$ from Lemma~\ref{lem:reduceFD} to obtain a representative in $[0,1)^{30}$ which we identify with $\mathbb{T}^{30}$. 
\begin{remark}[Informal note for the reader]\label{informal:remark}
A key result behind most of our analysis, that is not used directly in this paper, is the fact that integrals over phases of the form $\int e^{2\pi i b(n)x}dx$, on any closed interval in $[0,1]$ approach zero whenever $b(n)\rightarrow\infty$. In our setting, the identification $G/\Gamma\cong [0,1)^{30}$, Fourier analysis, the analysis above and Lemma~\ref{lem:reduceFD} show that any multiple correlation sequence is an integral of such phases, in multiple variables, where $b(n)$ could either be a polynomial of degree $2$, or some generalized polynomial, and might even depend on other variables. As we will see in the next sections, this observation (assuming non-decay) leads to linear constraints on the top layer Fourier characters associated with $f_0,f_1,f_2$. For instance, in the case of Weyl systems these constraints form a Vandermonde system of linear equations in the top layer characters, forcing the top layers to vanish as the Frantzikinakis--Kuca conjecture predicts.\footnote{Since this result is not directly used in this paper, we leave it as an exercise to the reader.} We will show in the next sections that in ergodic nilsystems we can not avoid the first two equations from the Vandermonde system, but we could avoid the third one. The reason behind this phenomenon lies in the term $\frac{n^2}{12}\ad_{\alpha}^2 W^{(1)}$ from the equation above, being the only $n^2$-component in the equation. Namely, to avoid the Vandermonde scenario described above, it will be important to choose the coefficients in the Fourier expansion of $f_0,f_1,f_2$ in such a manner that they are "annihilated by $(\ad_{\alpha}^2)^*$" (see Lemma~\ref{lem:kernelvariety}.) The informal discussion in Section~\ref{informal} discusses these ideas in more detail.
\end{remark}
\section{Fourier analysis on nilmanifolds}\label{Fourier:section}
In the previous sections we have developed \emph{Fourier analysis} on $G/\Gamma$ where $G$ is the free $3$-step nilpotent Lie group and $\Gamma$ the lattice defined above. 
More specifically, we may identify $G/\Gamma$ with its
fundamental domain $\mF=\tau([0,1)^{30})$, which (as a measure space) we may then identify with
$[0,1)^{30}$ and hence with $\T^{30}$. Our identification maps $\mu_{G/\Gamma}$ to the Lebesgue measure
on $[0,1)^{30}$ and hence to the Haar measure on $\T^{30}$. In particular, the functions
$f_0,f_1,f_2\in L^\infty(\mu_{G/\Gamma})$ which we will choose soon, can be viewed as functions on
$\T^{30}$. The latter is a compact abelian group and thus $f_0,f_1$ and $f_2$ can be approximated via
characters of that group. Importantly, we may decompose by layers $\T^{30}=\T^4\times\T^6\times\T^{20}$,
and observe that since our (measure-theoretic) isomorphism $G/\Gamma\cong\T^{30}$ preserves the layers, a
character $\chi=(\chi^{(1)},\chi^{(2)},\chi^{(3)})\in\widehat{\T^{30}}$ corresponds to a function with zero
$U^3(G/\Gamma)$-seminorm if and only if $\chi^{(3)}\neq 0$ (see \cite{host2005nonconventional}). Thus, in order to find a contradiction to the
Frantzikinakis--Kuca conjecture, it suffices to find characters $f_0=\chi_0$,
$f_1=\chi_1$ and $f_2=\chi_2$ of $\T^{30}$ with $\chi_i^{(3)}\neq 0$ for all $i=0,1,2$ (so that, in
particular, $\|f_i\|_{U^3}=0$ for every $i$)\footnote{In fact, it suffices that at least one of these characters is non-zero.}, and yet
$$
    a(n) = \int_{\mathcal{F}} \chi_0(\tau^{-1}(x))\cdot \chi_1(\tau^{-1}(\mathrm{red}(\alpha^n\ast x)))\cdot
    \chi_2(\tau^{-1}(\mathrm{red}(\alpha^{2n}\ast x)))~d\mu_{G/\Gamma}(x)
$$
does not converge to zero as $n\rightarrow\infty$. Here $\mF=\tau([0,1)^{30})\subseteq G$ is the
fundamental domain and $\mu_{G/\Gamma}$ the invariant probability measure, which our identification sends
to the Lebesgue measure $m$ on $[0,1)^{30}$. Working in second-kind coordinates and identifying $\mF$ with
$[0,1)^{30}$ via $\tau^{-1}$, under which $\mu_{G/\Gamma}$ becomes $m$ and $\mathrm{red}$ becomes
$\widetilde{\mathrm{red}}$, we may further write
\begin{equation}\label{newa}
    a(n) = \int_{[0,1)^{30}} \chi_0(x)\cdot \chi_1(\widetilde{\mathrm{red}}(\mu(\tilde{\alpha}_n,x)))\cdot
    \chi_2(\widetilde{\mathrm{red}}(\mu(\tilde{\alpha}_{2n},x)))\, dx,
\end{equation}
where now $x\in[0,1)^{30}$, $\tilde{\alpha}_n=\tau^{-1}(\alpha^n)$, and the integration is with respect to the Lebesgue measure $dx$.

Analyzing a general form in coordinates for \eqref{newa} is theoretically possible, but technically difficult. To ease our computations we will start each section by imposing certain restrictions on $f_0,f_1,f_2$. First, we shall write 
$$f_0 = (\chi,0,\eta_0),\qquad f_1=(0,0,\eta_1),\qquad f_2=(0,0,\eta_2),$$ be characters on $\T^{30}$ (later identified with functions on $G/\Gamma$) for some $\eta=(\eta_0,\eta_1,\eta_2) \in \Z^{ 20}\times \Z^{ 20}\times \Z^{20}$, and $\chi\in \Z^{4}$ to be chosen later. We then have the simplified form
\begin{equation}\label{chara}
    a(n) = \int_{[0,1)^{30}} e\left(\left<\chi,x^{(1)}\right> + \sum_{i=0}^2 \left<\eta_i,\tred^{(3)}\bigl(\mu(\tilde{\alpha}_{in},x)\bigr)\right>\right)\,dx.
\end{equation}
Write $$\Psi_n(x) =\left<\chi,x^{(1)}\right> + \sum_{i=0}^2 \left<\eta_i,\tred^{(3)}\bigl(\mu(\tilde{\alpha}_{in},x)\bigr)\right>\pmod 1.$$ Our next goal is to integrate $a(n)$ layer by layer. Next we compute the integral over $x^{(3)}$, demonstrating that $\sum_{i=0}^2 \eta_i\neq0$ implies vanishing $\lim_{n\rightarrow\infty}a(n)=0.$ Indeed, by \eqref{mu1}, \eqref{mu2} and \eqref{mu3}, $x^{(3)}$ is a free variable in $\mu(\tilde{\alpha}_m,x)$ and only appears in $\mu^{(3)}$. Since $\eta_i$ are characters (an integer in $\Z^{20}$), the first fractional part in $\widetilde{\mathrm{red}}$ vanishes and we conclude that $x^{(3)}$ is a free variable in $\Psi_n(x)$. More specifically, $\Psi_n(x) = \sum_{i=0}^2 (\eta_i(x^{(3)})+\psi_i(x))$, where $\psi_i$ depend only on $x^{(1)}$ and $x^{(2)}$, but not on $x^{(3)}$.
By the orthogonality of characters, we see that $a(n)=0$, unless $\sum_{i=0}^2 \eta_i = 0$. Recall that we are allowed to choose $\eta_0,\eta_1,\eta_2$ as we wish (as long as not all of them are zero). Thus, we shall assume that $\sum_{i=0}^2 \eta_i = 0$ (an explicit choice of the $\eta$'s is given in Section~\ref{eta}) and see that in this case $\Psi_n$ depends only on $x^{(1)},x^{(2)}$. In particular, we have
\begin{equation}\label{alayer12}
    a(n) = \int_{[0,1)^{4}\times [0,1)^6} e\left(\Psi_n(x^{(1)},x^{(2)})\right) \,dx^{(1)}\, dx^{(2)}.
\end{equation}
Next, we integrate over the second layer.
\section{Integrating over $x^{(2)}$}
Before we proceed, we shall need some notations. First, for every $m\in \mathbb{N}$ we set
\begin{align*}
\phi^{(m)}&:[0,1)^4\rightarrow \Z^4\\
    \phi^{(m)}(x^{(1)}) &= \lfloor m\alpha + x^{(1)}\rfloor
\end{align*}
with the convention that $\phi^{(0)}\equiv0$, and the floor map is applied coordinate-wise. Let $\delta^{(n)}(x^{(1)}):=\phi^{(2n)}(x^{(1)})-2\phi^{(n)}(x^{(1)})$ denote the difference. Moreover, for every $\delta=(\delta_1,\delta_2,\delta_3,\delta_4)\in V_1$, write $\rho_{\delta}:V_2\rightarrow V_3$ for the map
$$\rho_{\delta}(x) = [\delta,x].$$ Its transpose $\rho_{\delta}^* : V_3^*\rightarrow V_2^*$ satisfies
$$\rho_{\delta}^*v (e_{ij}) = v([\delta,[e_i,e_j]]).$$ In particular, if $\eta\in \Z^{20}$ we may write $\rho_{\delta}^*\eta$ for the functional determined by $e_{ij}\mapsto  \left<\eta,[\delta,[e_i,e_j]]\right>.$ 
 The following proposition will be useful soon.
\begin{proposition}\label{deltacoord}
    Every coordinate of $\delta^{(n)}$ is in $\{-1,0,1\}$.
\end{proposition}
\begin{proof}
    The proof is a direct computation. Look at a coordinate $1\leq \ell\leq 4$, we have
    $$\delta_{\ell}^{(n)} = \lfloor 2n\alpha_{\ell}+x^{(1)}_{
\ell}\rfloor - 2\lfloor n\alpha_{\ell} + x_{\ell}^{(1)}\rfloor. $$ Set $y:= n\alpha_{\ell} + x_{\ell}^{(1)}$, we get
$$\delta^{(n)}_{\ell} = \lfloor 2y-x^{(1)}_{\ell}\rfloor - 2\lfloor y\rfloor = \lfloor 2y-x_\ell^{(1)} - 2\lfloor y\rfloor \rfloor = \lfloor 2\{y\} - x^{(1)}_{\ell}\rfloor = \lfloor 2\{n\alpha_{\ell}+x^{(1)}_{\ell}\} - x^{(1)}_{\ell}\rfloor.$$
Since $\{y\},x^{(1)}_{\ell}\in [0,1)$, $2\{y\}-x^{(1)}_{\ell} \in (-1,2)$ and therefore $\delta^{(n)}_{\ell}\in \{-1,0,1\}$, as required.
\end{proof}
\begin{lemma}[The behavior along $x^{(2)}$]\label{Psinlemma}
In the setting above, we have $$\Psi_n(x^{(1)},x^{(2)},0) = \sum_{i=0}^2 \rho_{\phi^{(in)}}^*\eta_i(x^{(2)}) + \Psi_n(x^{(1)},0,0) \pmod{1}$$ In other words, $\Psi_n(x^{(1)},x^{(2)})$ is affine in $x^{(2)}$ with integer slope that may depend on $x^{(1)}$ and $n$.
\end{lemma}
\begin{proof}
Throughout we fix $x^{(1)}\in[0,1)^4$ and view $\Psi_n(x^{(1)},x^{(2)}) = \Psi_n(x^{(1)},x^{(2)},0)$ as a function depending only on $x^{(2)}$. Write $g_x:=\tau(x^{(1)},x^{(2)},0)$ and, for $m\in\Z$, $g_m:=\alpha^m\ast g_x$. By Lemma~\ref{lem:reduceFD}, there is a
unique $\gamma_m\in\Gamma$, such that $g_m\ast \gamma_m \in \mF$. By construction, $g_m\ast \gamma_m = \mathrm{red}(g_m).$ Now, let $s_m = \tau^{-1}(g_m)\in \R^{30}$ and $c_m=\tau^{-1}(\gamma_m)\in \Z^{30}$. We have $\tred\left(\mu(\tilde{\alpha}_m,x)\right) = \tau^{-1}(g_m\ast \gamma_m) = \mu(s_m,c_m).$ In particular, 
$$
\tred^{(3)}\left(\mu(\tilde{\alpha}_m,x)\right)  =\mu^{(3)}(s_m,c_m).$$
Thus, we shall study the dependence of $\mu^{(3)}(s_m,c_m)$ on $x^{(2)}$. From \eqref{mu3} we have
\begin{equation}\label{eq:t3decomp}
    \mu^{(3)}(s_m,c_m) = s_m^{(3)}+c_m^{(3)}+Q(s_m^{(1)},s_m^{(2)},c_m^{(1)}).
\end{equation}

We follow the $x^{(2)}$-dependence of each piece in \eqref{eq:t3decomp} (with $x^{(1)},m$ fixed).
\begin{itemize}
\item By \eqref{mu1}, $s_m^{(1)}=m\alpha+x^{(1)}$ is independent of $x^{(2)}$. Since 
$\mathrm{red}(g_m)\in\mathcal F$ forces $s_m^{(1)}+c_m^{(1)}\in[0,1)^4$ with $c_m^{(1)}\in\Z^4$,
uniqueness gives $c_m^{(1)}=-\lfloor s_m^{(1)}\rfloor=-\phi^{(m)}(x^{(1)})$. 
\item Let $s(\alpha^m) = \tau^{-1}(\alpha^m)$, as in Corollary~\ref{cor:secondkind}. By \eqref{mu2}, $s_m^{(2)}=s^{(2)}(\alpha^m)+x^{(2)}+P\bigl(s^{(1)}(\alpha^m),x^{(1)}\bigr)=:x^{(2)}+d_m(x^{(1)})$ is
affine in $x^{(2)}$, with coefficient $1$. 
\item $s_m^{(3)}$ is independent of $x^{(2)}$: in
$s_m^{(3)}=\mu^{(3)}\bigl(s(\alpha^m),(x^{(1)},x^{(2)},0)\bigr)$. Indeed, the slot $x'^{(2)}=x^{(2)}$ can enter only through $Q$ which is independent of $x'^{(2)}$.
\item $c_m^{(2)},c_m^{(3)}$ may depend on $x^{(2)}$. However, the slot $x'^{(2)}=c_m^{(2)}$ does not enter $Q$, so $c_m^{(2)}$ appears in
\eqref{eq:t3decomp} nowhere, while $c_m^{(3)}\in\Z^{20}$ enters only as the explicit integer summand (which vanishes after applying $\eta_i$).
\end{itemize}
Hence, the sole $x^{(2)}$-dependence in \eqref{eq:t3decomp} is the single $Q$-term
$[\,s_m^{(2)},c_m^{(1)}\,]$, linear in $x^{(2)}$ (and the integer $c_m^{(3)}$ which vanishes under $\eta$). Therefore, absorbing the $x^{(1)}$-dependence into some $D_m(x^{(1)})$, we are left with 
$$ \left<\eta_i, [s_m^{(2)},c_m^{(1)}]\right>=\left<\eta_i,-[c_m^{(1)},s_m^{(2)}]\right> =  -\left<\eta_i, \rho_{c_m^{(1)}}(s_m^{(2)})\right>= \rho_{\phi^{(m)}}^*\eta_i(s_m^{(2)}).$$ Since $s_m^{(2)} = x^{(2)} + d_m(x^{(1)})$, absorbing the latter into $D_m(x^{(1)})$, we are left with $\rho_{\phi^{(m)}}^*\eta_i(x^{(2)})+  D_m(x^{(1)}).$

Thus, $\Psi_n(x^{(1)},x^{(2)},0) = \left<\chi,x^{(1)}\right> + \sum_{i=0}^2 D_{in}(x^{(1)}) + \sum_{i=0}^2 \rho^*_{\phi^{(in)}}\eta_i(x^{(2)}).$ Substituting $x^{(2)}=0$ gives $\Psi_n(x^{(1)},0,0)=\left<\chi,x^{(1)}\right> + \sum_{i=0}^2D_{in}(x^{(1)})$ from which it follows that $\Psi_n(x^{(1)},x^{(2)}) = \sum_{i=0}^2 \rho^*_{\phi^{(in)}}\eta_i(x^{(2)}) + \Psi_n(x^{(1)},0,0),$ as required. 
\end{proof}
Write $\sigma_n:=\sum_{i=0}^2\rho^*_{\phi^{(in)}}\eta_i\in V_2^*$ for the slope produced by
Lemma~\ref{Psinlemma}. It has integer coordinates, since
$\rho^*_\delta\eta_i(e_{jk})=\langle\eta_i,[\delta,e_{jk}]\rangle$ is an integer combination of the
$\delta_\ell$ whenever $\eta_i\in\Z^{20}$. Thus
\[
\Psi_n(x^{(1)},x^{(2)})=\sigma_n(x^{(2)})+\Psi_n(x^{(1)},0,0)\pmod1,
\]
and, because $x^{(2)}\mapsto e(\sigma_n(x^{(2)}))$ is a character of $\T^6$, integrating over
$x^{(2)}$ keeps only its trivial part:
\begin{equation}\label{eq:x2integral}
\int_{[0,1)^6}e\bigl(\Psi_n(x^{(1)},x^{(2)})\bigr)\,dx^{(2)}
=e\bigl(\Psi_n(x^{(1)},0,0)\bigr)\cdot\mathbf 1_{\{\sigma_n=0\}}.
\end{equation}

We claim that, unless $\eta_1+2\eta_2=0$, the indicator in \eqref{eq:x2integral} vanishes for all
large $n$, so that $a(n)=0$. Since $\phi^{(0)}\equiv0$ we have
$\sigma_n=\rho^*_{\phi^{(n)}}\eta_1+\rho^*_{\phi^{(2n)}}\eta_2$. Writing
$\phi^{(2n)}=2\phi^{(n)}+\delta^{(n)}$ with $\delta^{(n)}\in\{-1,0,1\}^4$
(Proposition~\ref{deltacoord}) and using that $\delta\mapsto\rho^*_\delta\zeta$ is linear,
\[
\sigma_n(x^{(1)})=\rho^*_{\phi^{(n)}(x^{(1)})}\,(\eta_1+2\eta_2)+\rho^*_{\delta^{(n)}(x^{(1)})}\,\eta_2 .
\]
The second term is bounded uniformly in $n$ and $x^{(1)}$. For the first, write
$\phi^{(n)}(x^{(1)})=\lfloor n\alpha\rfloor+\epsilon$ with $\epsilon\in\{0,1\}^4$, so that
\[
\sigma_n(x^{(1)})=\rho^*_{\lfloor n\alpha\rfloor}(\eta_1+2\eta_2)+O_\eta(1),
\]
the error being uniform in $x^{(1)}$. Suppose $\zeta:=\eta_1+2\eta_2\neq0$. The linear map $\delta\mapsto\rho^*_\delta\zeta$ is non-zero
(otherwise $\zeta$ would annihilate $[V_1,V_2]=V_3$ and hence vanish), so there are indices
$1\le i<j\le4$ and integers $c_1,\dots,c_4$, not all zero, with
\[
\rho^*_{\delta}\zeta(e_{ij})=\sum_{k=1}^4 c_k\,\delta_k,\qquad c_k=\langle\zeta,[e_k,e_{ij}]\rangle\in\Z .
\]
Evaluating at $\delta=\lfloor n\alpha\rfloor$ and writing $\lfloor n\alpha_k\rfloor=n\alpha_k-\{n\alpha_k\}$,
\[
\rho^*_{\lfloor n\alpha\rfloor}\zeta(e_{ij})=n\sum_{k}c_k\alpha_k-\sum_{k}c_k\{n\alpha_k\}.
\]
Since $1,\alpha_1,\dots,\alpha_4$ are independent over $\Q$ and the $c_k$ are integers not all zero,
$\sum_k c_k\alpha_k\neq0$. Therefore, this coordinate of $\sigma_n(x^{(1)})$ satisfies
$|\rho^*_{\lfloor n\alpha\rfloor}\zeta(e_{ij})|\to\infty$ uniformly in $x^{(1)}$ (recall
$\sigma_n=\rho^*_{\lfloor n\alpha\rfloor}\zeta+O_\eta(1)$). In particular, $\sigma_n(x^{(1)})\neq0$ for
every $x^{(1)}$ once $n$ is large, so $\mathbf 1_{\{\sigma_n=0\}}$ is eventually identically zero.

We see that any non-vanishing example must satisfy $\eta_1+2\eta_2=0$. Combined with the
constraint $\eta_0+\eta_1+\eta_2=0$ from the previous section this forces
\[
(\eta_0,\eta_1,\eta_2)=(\eta,-2\eta,\eta),\qquad \eta:=\eta_2\in\Z^{20}\ \text{(chosen in Section~\ref{eta}).}
\]
With this choice $\sigma_n=\rho^*_{\delta^{(n)}}\eta$, where $\delta^{(n)}=\phi^{(2n)}-2\phi^{(n)}$, and
\eqref{eq:x2integral} yields the following.
\begin{corollary}\label{cor:explicit}
In the setting above, with $(\eta_0,\eta_1,\eta_2)=(\eta,-2\eta,\eta)$,
\[
\int_{[0,1)^6}e\bigl(\Psi_n(x^{(1)},x^{(2)})\bigr)\,dx^{(2)}
=e\bigl(\Psi_n(x^{(1)},0,0)\bigr)\cdot\mathbf 1_{\{\delta^{(n)}(x^{(1)})\in K\}},
\qquad K=\{v\in\Z^4:\rho^*_v\eta=0\}.
\]
\end{corollary}
In the next section, we shall integrate over $x^{(1)}$. This part is particularly interesting as it involves both a quadratic term, a linear term, and most crucially generalized quadratic terms.
\subsection{Choosing $\eta$}\label{eta}
We choose $\eta\in \Z^{20}$, once and for all, whose only non-zero coordinates are:
 \[
\langle\eta,e_{4,12}\rangle=2,\quad \langle\eta,e_{1,13}\rangle=-9,\quad
\langle\eta,e_{2,14}\rangle=1,\quad \langle\eta,e_{2,23}\rangle=6.
\]
\begin{remark}[Optional remark on the choice of $\eta$.]\label{etachoice:rem}
  This $\eta$ was produced by mistake. At an earlier attempt of producing a counterexample the rotation was chosen first. A seemingly reasonable choice was $$\alpha=(\sqrt{2},\sqrt{3},\sqrt{5},\sqrt{30},0,\dots,0)$$ (only the first layer coordinates are non-zero) because it is the simplest one where $\ker (\ad_{\alpha}^2)^*$ is non-trivial and $\eta$ was the simplest vector in this kernel (which was necessary to avoid the Vandermonde system discussed earlier). In retrospect, this choice of $\alpha$, and any choice of $\alpha$ with algebraic coordinates can not produce a counterexample to the Frantzikinakis--Kuca conjecture due to Schmidt's subspace theorem (see \cite{Schmidtsubspace}). In particular, the set $\mathcal{V}$ (introduced below) was defined in retrospect as the variety of all $\alpha$ satisfying  $\eta\in \ker(\ad_{\alpha}^2)^*$ for this choice of $\eta$. Then, we were fortunate to be able to choose from $\mathcal{V}$ the transcendental $\alpha$ that was eventually used for the counterexample.
\end{remark}
\begin{lemma}[Properties of $\eta$]\label{lem:factorization}
Abbreviate $\eta=(2,-9,1,6)$ in the ordered basis $(e_{4,12},e_{1,13},e_{2,14},e_{2,23})$ of its support, and let $\delta=(\delta_1,\delta_2,\delta_3,\delta_4)\in \Z^4$ be arbitrary. Then,
\begin{enumerate}
\item[(i)] In the basis $(e_{12},e_{13},e_{14},e_{23},e_{24},e_{34})$ of $V_2$,
\[
\rho^{*}_{\delta}\eta=\bigl(2\delta_4,\,-9\delta_1,\,\delta_2,\,6\delta_2,\,-\delta_1,\,0\bigr).
\]
In particular, $\rho^{*}_{\delta}\eta=0\iff\delta_1=\delta_2=\delta_4=0$, and the kernel of $\delta\mapsto\rho^{*}_{\delta}\eta$ is
\[
K=\{\delta:\rho^{*}_{\delta}\eta=0\}=\Z e_3=\{(0,0,*,0)\}\subset\Z^4 .
\]
\item[(ii)] Write $\theta_{\ell}:=\{n\alpha_{\ell}\}$ and $E^{(n)}_{\ell}:=\{x\in[0,1):\delta^{(n)}_{\ell}(x)=0\}$, where $\ell\in \{1,2,3,4\}$ and $\delta^{(n)}_{\ell}$ is the $\ell^{\mathrm{th}}$-coordinate of $\delta^{(n)}$. Then,
$$
E_{\ell}^{(n)}=
\begin{cases}
[0,1-2\theta_{\ell}\,) & \theta_{\ell}<\frac{1}{2},\\[2pt]
[2-2\theta_{\ell},1) & \theta_{\ell}>\frac{1}{2},\\[2pt]
\emptyset & \theta_{\ell}=\frac{1}{2},
\end{cases}
\qquad
\bigl|E_\ell^{(n)}\bigr|=\bigl|\,1-2\{n\alpha_{\ell}\}\,\bigr|.
$$
where $|E_{\ell}^{(n)}|$ is the length of the interval (or Lebesgue measure of) $E_{\ell}^{(n)}$.
\item[(iii)] Consequently, from Corollary~\ref{cor:explicit} we have
$a(n)=\int_{[0,1)^4}\mathbf 1_{\{\delta^{(n)}\in K\}}(x^{(1)})\cdot \,e\bigl(\Psi_n(x^{(1)})\bigr)\,dx^{(1)}$, factorizes as
\begin{equation}\label{finala}a(n)=\int_{E_n^{(1)}\times E_n^{(2)}\times[0,1)\times E_n^{(4)}}e\bigl(\Psi_n(x^{(1)})\bigr)\,dx^{(1)}.
\end{equation}
\end{enumerate}
\end{lemma}
 
\begin{proof}
We start with $(i)$. By definition, $\rho^*_{\delta}\eta(e_{ij}) = \left<\eta,[\delta,e_{ij}]\right>= \sum_{\ell=1}^4 \delta_{\ell}\cdot \left<\eta,[e_\ell,e_{ij}]\right>.$
From the definition of our basis, we have
\[
[e_{\ell},[e_i,e_j]]=
\begin{cases}
e_{\ell,ij} & \ell\ge i,\\
e_{i,\ell j}-e_{j,\ell i} & \ell<i,
\end{cases}
\qquad(i<j),
\]
the second line being the Jacobi identity $[e_\ell,[e_i,e_j]]=[e_i,[e_\ell,e_j]]-[e_j,[e_\ell,e_i]]$. Pairing with $\eta$, which is supported on
$\{e_{4,12},e_{1,13},e_{2,14},e_{2,23}\}$, we read off each column:
\begin{itemize}
\item $e_{12}$: $[e_\ell,[e_1,e_2]]=e_{\ell,12}$ for all $\ell$. However, only $\ell=4$ meets the support of $\eta$, $\left<\eta,e_{4,12}\right>=2$. Multiplying by $\delta_4$ gives the first-kind coordinate.
\item $e_{13}$: $[e_\ell,[e_1,e_3]]=e_{\ell,13}$. However, only $\ell=1$ meets the support of $\eta$, $\left<\eta,e_{1,13}\right>=-9$. Multiplying by $\delta_1$ gives the second-kind coordinate. 
\item $e_{14}$: $[e_\ell,[e_1,e_4]]=e_{\ell,14}$. However, only $\ell=2$ meets the support of $\eta$, $\left<\eta,e_{2,14}\right>=1$. Multiplying by $\delta_2$ gives the third coordinate.
\item $e_{23}$: for $\ell\ge2$, $[e_\ell,[e_2,e_3]]=e_{\ell,23}$, of which only $\ell=2$ meets the support, $\left<\eta,e_{2,23}\right>=6$. For $\ell=1$,  $[e_1,[e_2,e_3]]=e_{2,13}-e_{3,12}$, both off the support. Thus, the fourth quantity is $\delta_2\cdot \left<\eta,e_{2,23}\right>=6\delta_2.$
\item $e_{24}$: for $\ell\ge2$, $[e_\ell,[e_2,e_4]]=e_{\ell,24}$, all outside the support of $\eta$. For $\ell=1$,
$[e_1,[e_2,e_4]]=e_{2,14}-e_{4,12}$, paired with $\eta$ gives $1-2=-1$. Multiplying by $\delta_1$ give the fifth coordinate $-\delta_1$.
\item $e_{34}$: for $\ell\ge3$, $e_{\ell,34}$ is outside the support. For $\ell=1,2$,
$[e_1,[e_3,e_4]]=e_{3,14}-e_{4,13}$ and $[e_2,[e_3,e_4]]=e_{3,24}-e_{4,23}$, all outside the support. This gives the last coordinate.
\end{itemize}
Thus, $\rho_{\delta}^* \eta = 0$ if and only if all the coordinates are zero which gives $\delta_1=\delta_2=\delta_4=0$, while $\delta_3$ is unconstrained. Therefore, 
$K=\Z e_3$, as required.\\

Now we prove $(ii)$. Fix $1\leq \ell\leq 4$, write $x=x^{(1)}_{\ell}$ and $\theta:=\theta_{\ell}=\{n\alpha_{\ell}\}$. From the computation in the proof of Proposition \ref{deltacoord}, we see that $\delta_{\ell}^{(n)}(x) = 0$ if and only if $0\leq 2\{\theta+x\}-x < 1.$ Thus, we consider two cases: if $\theta+x < 1$, then $\{\theta+x\}=\theta+x$ and $\delta_{\ell}^{(n)}(x) = 0$ when $$0\leq 2\theta <1 \iff 0\leq \theta<\frac{1}{2}, $$ and if $\theta+x\geq 1$, then $\{\theta+x\} = \theta+x-1$ and $\delta_{\ell}^{(n)}(x)=0$ when
$$0\leq 2\theta + x - 2 < 1 \iff 2-2\theta\leq x\leq 3-2\theta.$$ However, since $x<1\leq 3-2\theta$, the right inequality is automatic. The left inequality $x\geq 2-2\theta$ is only possible when $\theta>\frac{1}{2}$, and then $x\geq 2-2\theta \geq 1-\theta$ so $\theta+x\geq 1$. 
The two cases partition $[0,1)$, so $E_\ell^{(n)}$ is their union, giving the stated description. Its length is
$1-2\theta$ for $\theta<\tfrac12$ and $1-(2-2\theta)=2\theta-1$ for $\theta>\frac{1}{2}$, i.e., $|1-2\theta|$ in both,
and $0$ at $\theta=\frac{1}{2}$.\\
Finally, we deduce $(iii)$ by noting that from $(i)$, $\delta^{(n)}\in K$ if and only if $\delta^{(n)}_1=\delta^{(n)}_2=\delta^{(n)}_4=0$ (with no condition on $\delta^{(n)}_3$). Since $\delta_{\ell}^{(n)}$ is only a function of $x^{(1)}_{\ell}$, the indicator splits into a product of single-variable indicators independent of $x_3^{(1)}.$ Thus,
\[
\mathbf 1_{\{\delta_n\in K\}}=\mathbf 1_{E_n^{(1)}}(x^{(1)}_{1})\,\mathbf 1_{E_n^{(2)}}(x^{(1)}_{2})\,\mathbf 1_{E_n^{(4)}}(x^{(1)}_{4}).
\]
Inserting this into Corollary~\ref{cor:explicit} and restricting the domain of
integration accordingly,
\[
a(n)=\int_{[0,1)^4}\mathbf 1_{E_n^{(1)}}\mathbf 1_{E_n^{(2)}}\mathbf 1_{E_n^{(4)}}\,e\bigl(\Psi_n(x^{(1)})\bigr)\,dx^{(1)}
=\int_{E_n^{(1)}\times E_n^{(2)}\times[0,1)\times E_n^{(4)}}e\bigl(\Psi_n(x^{(1)})\bigr)\,dx^{(1)},
\]
as required.
\end{proof}
\section{Integrating over $x^{(1)}$}
\subsection{Informal discussion.}\label{informal}
Recall that the top layers of the characters attached to $f_0,f_1,f_2$ are $\eta_0,\eta_1,\eta_2\in \Z^{20}$. In the previous sections we noticed that non-vanishing requires two constraints:
\begin{equation}\label{2vandermonde}
\begin{cases}
    \eta_0 + \eta_1 + \eta_2 = 0,\\ \eta_1 + 2\eta_2=0
\end{cases}.
\end{equation}
The first is forced by orthogonality of characters, without which $a(n)$ vanishes identically. The second however is dependent on the coefficient of $n$. We saw that since $\alpha\mapsto [\alpha,x^{(2)}]$ is \textit{linear} in $\alpha$, the $n$-coefficient is a fixed integral linear form in the coordinates $\alpha_1,\dots,\alpha_4$ which are independent over $\Q$ due to ergodicity.\\
The coefficient of $n^2$ is of different nature. Taking a look at \eqref{eq:orbit} we see that the only $n^2$-term of the orbit is $\frac{n^2}{12}\ad_{\alpha}^2 W^{(1)}$, so after substituting the times $0$ for $f_0$, $n$ for $f_1$ and $2n$ for $f_2$ we get the $n^2$-coefficient
$$\frac{1}{12}\left<\eta_1+4\eta_2,\ad_{\alpha}^2 W^{(1)}\right>.$$
Theoretically, this might force the third linear equation $\eta_1+4\eta_2=0$, which together with \eqref{2vandermonde} is a Vandermonde system of linear equations with a unique solution $\eta_0=\eta_1=\eta_2=0$. The key difference here is that $\alpha\mapsto \ad_{\alpha}^2 W^{(1)} = [\alpha,[\alpha,W^{(1)}]]$ is \textit{quadratic} in $\alpha$. Thus, each coordinate of $(\ad_{\alpha}^2)^*(\eta_1+4\eta_2)\in V_1^*$ is an integral linear combination of the ten products $\alpha_i\alpha_j$ for $1\leq i \leq j \leq 4$. It can be shown that when all these ten products are irrational and independent over $\Q$ (which is the case e.g., for almost every $\alpha$), then the last Vandermonde equation $\eta_1+4\eta_2=0$ is forced, and $\eta_0=\eta_1=\eta_2=0$ as the Frantzikinakis--Kuca conjecture predicts. We will therefore have to restrict to a family of rotations $\mathcal{V}$, defined below, which satisfies rational linear dependencies between these $10$ monomials.
\subsection{Rotations annihilating the $n^2$-coefficient}
Consider the variety,
$$\mathcal{V} = \{\alpha \in (0,1)^4 : \alpha_2\alpha_4 = 3\alpha_1\alpha_3 \,\land\, \alpha_1\alpha_4 = 2\alpha_2\alpha_3\}.$$
\begin{lemma}\label{Vproperties}
    For every $\alpha\in \mathcal{V}$ we have $3\alpha_1^2 = 2\alpha_2 ^2$. In particular,
    $$\mathcal{V} = \{(t,\sqrt{\frac{3}{2}}t,u,\sqrt{6}u)\in(0,1)^4 : t,u\in \R\}.$$
\end{lemma}
\begin{proof}
    Isolating $\alpha_4$, we get
    $$\frac{3\alpha_1\alpha_3}{\alpha_2} = \frac{2\alpha_2\alpha_3}{\alpha_1}$$ multiplying by $\alpha_1\cdot \alpha_2$ and dividing by $\alpha_3$ gives the first claim. Now let $t=\alpha_1$ and $u=\alpha_3$. We see that $\alpha_2\cdot \alpha_4 = 3tu$ while $\alpha_4 = 2\alpha_2 \cdot \frac{u}{t}.$ Multiplying both equations we get $\alpha_4^2 = 6u^2\Rightarrow \alpha_4 = \sqrt{6}u.$ Substituting in one of the equations we get that $\alpha_2 = \sqrt{\frac{3}{2}}t$, as required.
\end{proof}
Consider the map $\ad_{\alpha}^2:\f\rightarrow V_3$, and note that it factors through $V_1$ (it is trivial on $V_2\oplus V_3$). Let $(\ad_{\alpha}^2)^* : V^*_3\rightarrow V^*_1$ denote the transpose. Our next goal is to find a non-trivial element in the kernel of this map for $\alpha\in \mathcal{V}$. Recall that $\eta \in \Z^{20} \subseteq V_3^*$ denote the element with the $4$ non-zero coordinates: \[
\langle\eta,e_{4,12}\rangle=2,\quad \langle\eta,e_{1,13}\rangle=-9,\quad
\langle\eta,e_{2,14}\rangle=1,\quad \langle\eta,e_{2,23}\rangle=6,
\]
\begin{lemma}[The top phase $\eta$ vanishes under $(\ad_{\alpha}^2)^*$.]
\label{lem:kernelvariety}
Let $\alpha=\sum_{j=1}^4\alpha_j e_j\in (0,1)^4\subseteq V_1$ and $\eta\in V_3^*$. Then in the dual basis $(e_1^{*},e_2^{*},e_3^{*},e_4^{*})$, the coordinates of $(\ad_\alpha^{2})^*\eta\in V_1^{*}$ are $$3(3\alpha_1\alpha_3-\alpha_2\alpha_4),\quad 3(\alpha_1\alpha_4-2\alpha_2\alpha_3),\quad  -3(3\alpha_1^2-2\alpha_2^2),\quad 0,$$ respectively.
Consequently,
$$
\eta\in\ker(\ad_\alpha^{2})^*\iff\alpha\in\mathcal V$$
\end{lemma}

\begin{proof}
Recall that $\ad_\alpha^2 e_\ell=[\alpha,[\alpha,e_\ell]]$. From Lemma~\ref{lem:factorization}, we also have that $\rho^*_{\alpha}\eta = (2\alpha_4,-9\alpha_1,\alpha_2,6\alpha_2,-\alpha_1,0)$ in the basis $(e_{12},e_{13},e_{14},e_{23},e_{24},e_{34})$. We have
$$\left<\eta,\ad_{\alpha}^2 e_{\ell}\right> = \rho_{\alpha}^*\eta([\alpha,e_{\ell}])\qquad [\alpha,e_{\ell}] = \sum_{j\not = \ell} \alpha_j [e_j,e_{\ell}],$$ where the last equality follows from bilinearity of the bracket and the fact that $e_{\ell}$ commutes with itself. Now, we shall write $[\alpha,e_{\ell}]$ in the basis $(e_{12},e_{13},e_{14},e_{23},e_{24},e_{34})$ and then apply $\rho^*_{\alpha}{\eta}$ (i.e., take an inner product with $(2\alpha_4,-9\alpha_1,\alpha_2,6\alpha_2,-\alpha_1,0)$.
\[
\begin{aligned}
\ell=1:&\ \ [\alpha,e_1]=-\alpha_2e_{12}-\alpha_3e_{13}-\alpha_4e_{14}
&&\!\!\!\mapsto\ -\alpha_2(2\alpha_4)-\alpha_3(-9\alpha_1)-\alpha_4(\alpha_2)=3(3\alpha_1\alpha_3-\alpha_2\alpha_4),\\
\ell=2:&\ \ [\alpha,e_2]=\alpha_1e_{12}-\alpha_3e_{23}-\alpha_4e_{24}
&&\!\!\!\mapsto\ \alpha_1(2\alpha_4)-\alpha_3(6\alpha_2)-\alpha_4(-\alpha_1)=3(\alpha_1\alpha_4-2\alpha_2\alpha_3),\\
\ell=3:&\ \ [\alpha,e_3]=\alpha_1e_{13}+\alpha_2e_{23}-\alpha_4e_{34}
&&\!\!\!\mapsto\ \alpha_1(-9\alpha_1)+\alpha_2(6\alpha_2)-\alpha_4(0)=-3(3\alpha_1^{2}-2\alpha_2^{2}),\\
\ell=4:&\ \ [\alpha,e_4]=\alpha_1e_{14}+\alpha_2e_{24}+\alpha_3e_{34}
&&\!\!\!\mapsto\ \alpha_1(\alpha_2)+\alpha_2(-\alpha_1)+\alpha_3(0)=0 .
\end{aligned}
\]
We conclude that $\eta\in \ker (\ad_{\alpha}^2)^*\iff \alpha\in \mathcal V$, as required.
\end{proof}
\subsection{Computing the final integral}
From now on we fix $\alpha \in \mathcal{V}$ to be chosen later, $\chi\in \Z^{4} =  \widehat{\T^4}$ to be chosen later, and let $\eta, f_0,f_1$ and $f_2$ be as in the previous sections. We compute the final integral \eqref{finala} explicitly for these choices. To do so we need to compute $\Psi_n(x^{(1)}) = \Psi_n(x^{(1)},0,0).$
For an orbit point $\alpha^m\ast x$ with $x=(x^{(1)},0,0)$ put
\[
u^{(m)}_\ell:=m\alpha_\ell+x^{(1)}_{\ell}\quad(\text{the un-floored layer-$1$ coordinate}),\qquad
\phi^{(m)}_\ell:=\lfloor u^{(m)}_\ell\rfloor=\lfloor m\alpha_\ell+x^{(1)}_{\ell}\rfloor,
\]
and let $P(m):=\bigl\langle\eta, \tred^{(3)}(\mu(\tilde{\alpha}_m,x)) \bigr\rangle$ be the top-layer phase of that
single point. 

The next proposition is a direct (but very long) computation (see Appendix~\ref{app:cert}).

\begin{proposition}\label{prop:Dn}
Let $m$ be arbitrary and for the sake of notational simplicity write $\phi_{\ell}^{(m)}=\phi_{\ell}$ and $u^{(m)}_{\ell} = u_{\ell}$ for all $1\leq \ell \leq 4$. Then for any $\alpha=(\alpha_1,\dots,\alpha_4)$,
\[
P(m)=-\frac{9}{2}\,\phi_1^{2}u_3+3\,\phi_2^{2}u_3+\phi_1\phi_2\,u_4+2\,\phi_1\phi_4\,u_2
+\phi_1A_1+\phi_2A_2+\phi_4A_4+A_0,
\]
where, writing $\sigma:=m^2+m$,
\[
\begin{aligned}
A_1&=-\tfrac12(9\alpha_1\alpha_3+\alpha_2\alpha_4)\,\sigma+m\,(9\alpha_3u_1+\alpha_4u_2)-2u_2u_4+\tfrac92u_3,\\
A_2&=\tfrac12(\alpha_1\alpha_4+6\alpha_2\alpha_3)\,\sigma-m\,(\alpha_4u_1+6\alpha_3u_2)-3u_3,\\
A_4&=\alpha_1\alpha_2\,\sigma-2\alpha_2\,m\,u_1,
\end{aligned}
\]
and $A_0$ is the floor-free part and a polynomial of degree $3$ in $m$ whose coefficients are
products of the $\alpha_j$ with the $u_{\ell}$ (fully recorded below).
Consequently, $$\Psi_n(x^{(1)}) =\left<\chi,x^{(1)}\right> -2P(n) + P(2n) \pmod 1.$$ 
\end{proposition}
\begin{remark}
   The floor-free part $A_0$ is equal to 
    \[
A_0(m)=c_3\,m^3+c_2\,m^2+c_1\,m,\qquad A_0(0)=0,
\]
where (writing $x_\ell:=x^{(1)}_{\ell}$)
\[
\begin{aligned}
c_3&=-\tfrac12\bigl(3\alpha_1^2\alpha_3-2\alpha_1\alpha_2\alpha_4-2\alpha_2^2\alpha_3\bigr),\\[2pt]
c_2&=-\tfrac14\bigl(-9\alpha_1^2\alpha_3+2\alpha_1\alpha_2\alpha_4+6\alpha_2^2\alpha_3+9\alpha_1\alpha_3-6\alpha_2\alpha_3\\
   &\qquad\quad+18\alpha_1\alpha_3\,x_1-6\alpha_2\alpha_4\,x_1-2\alpha_1\alpha_4\,x_2-12\alpha_2\alpha_3\,x_2-4\alpha_1\alpha_2\,x_4\bigr),\\[2pt]
c_1&=-\tfrac14\bigl(3\alpha_1^2\alpha_3+2\alpha_1\alpha_2\alpha_4-2\alpha_2^2\alpha_3-9\alpha_1\alpha_3+6\alpha_2\alpha_3\\
   &\qquad\quad-18\alpha_1\alpha_3\,x_1-2\alpha_2\alpha_4\,x_1+2\alpha_1\alpha_4\,x_2+12\alpha_2\alpha_3\,x_2+4\alpha_1\alpha_2\,x_4\\
   &\qquad\quad+18\alpha_3\,x_1^2-12\alpha_3\,x_2^2+18\alpha_3\,x_1-12\alpha_3\,x_2-8\alpha_2\,x_1x_4-4\alpha_4\,x_1x_2\bigr).
\end{aligned}
\]
\end{remark}
We have computed $\alpha^m\ast x$ in the first-kind coordinates. For simplicity, we define the \emph{peel} map as the coordinate-change map. Namely,
\begin{equation}\label{peel}
\mathrm{peel}=\tau^{-1}\circ \exp :\R^{30}\rightarrow \R^{30}.
\end{equation} In our setting, $\exp$ is the identity so $\mathrm{peel}$ is merely the inverse of $\tau$. Recall moreover that $\tau(s)=\tau_1(s^{(1)})\ast \tau_2(s^{(2)})\ast \tau_3(s^{(3)})$ from the proof of Theorem~\ref{secondcoor:thm}. Writing $v=v^{(1)}+v^{(2)}+v^{(3)}$ for the layers of $v=\log g$ the peel $s=\mathrm{peel}(v)$ is obtained layer by layer by
$$s^{(1)} = v^{(1)},\qquad s^{(2)} = \left[\log(\tau_1(s^{(1)})^{-1}\ast g)\right]^{(2)},\qquad s^{(3)}=\left[\log\bigl(\tau_2(s^{(2)})^{-1}\ast\tau_1(s^{(1)})^{-1}\ast g\bigr)\right]^{(3)}.$$
\begin{proof}
Throughout, we abbreviate the layer-$1$ coordinates, writing
$x_\ell:=x_\ell^{(1)}$ for $\ell=1,2,3,4$.
Write $g=\alpha^m\ast x$, with $x=\tau(x^{(1)},0,0)$ and let $s(g) = (s^{(1)}(g),s^{(2)}(g),s^{(3)}(g))$ be its second-kind coordinates. Let $\gamma=\tau(c)\in \Gamma$ be the unique element so that $g\ast \gamma \in \mF$. Then,
$$\tred^{(3)}(\mu(\tilde{\alpha}_m,x)) = \mu^{(3)}(s(g),c)$$ and $$P(m) = \left<\eta,\mu^{(3)}(s(g),c)\right> \pmod 1.$$

Recall that $s^{(1)}(g) = m\alpha+x^{(1)}$, this forces 
\[
c^{(1)}=-\phi,\qquad \phi:=\phi^{(m)}=\lfloor m\alpha+x_1\rfloor\in\Z^4 .
\]
Throughout $u_\ell=u^{(m)}_\ell=m\alpha_\ell+x^{(1)}_{\ell}$, so $\phi_\ell=\lfloor u_\ell\rfloor$.
 Recall formula \eqref{mu3}:
\begin{equation}\label{eq:mu3struct}
\mu^{(3)}(y,y')=y^{(3)}+y'^{(3)}+Q\bigl(y^{(1)},y^{(2)},\,y'^{(1)}\bigr),
\end{equation}
where the cross term $Q$ has the following three properties, all established in Theorem~\ref{secondcoor:thm} and are used below.
\begin{itemize}
\item[(a)] $Q$ \emph{does not depend on} $y'^{(2)}$.
\item[(b)] $\mu^{(3)}$ is affine in $y'^{(3)}$ with coefficient $1$. In particular, $Q$ does not depend on $y'^{(3)}$.
\item[(c)] $Q(\cdot,\cdot,0)=0$: every term of $Q$ is a bracket containing $y'^{(1)}$, hence is at least
linear in $y'^{(1)}$ and vanishes when $y'^{(1)}=0$.
\end{itemize}

\noindent\emph{Step 1: the reduction floors above layer $1$ do not affect $P(m)$ modulo $1$.}
Apply \eqref{eq:mu3struct} with $y=s(g)$ and $y'=c$:
\[
\mu^{(3)}(s(g),c)=s^{(3)}(g)+c^{(3)}+Q\bigl(s^{(1)}(g),\,s^{(2)}(g),c^{(1)}\bigr).
\]
By property (a) the layer-$2$ component $c^{(2)}$ of the reduction does not occur on the right-hand side ($Q$ is independent of $c^{(2)}$). By property (b) the layer-$3$ component $c^{(3)}$ appears only as the additive
$c^{(3)}\in\Z^{20}$, and since $\eta$ has integer coefficients,
\[
\left<\eta,\,c^{(3)}\right>\equiv 0\ (\bmod 1).
\]
Finally $c^{(1)}=-\phi$ as noted. Substituting,
\begin{equation}\label{eq:Pphi}
P(m)\equiv \left<\eta,s^{(3)}(g)\right>+\left<\eta,Q\bigl(s^{(1)}(g),s^{(2)}(g),-\phi\bigr)\right> \pmod 1 .
\end{equation}
Both terms are polynomials in $(m,\alpha,x_1)$ assembled from $\Lambda=\log \alpha$ (see Section~\ref{orbit:section}),
the power coordinates $s(\alpha^m)$ (Corollary~\ref{cor:secondkind}), and the group law. Additionally, the second term is a
polynomial in the single layer-$1$ floor $\phi$ of degree $\le 2$ (the bracket length in a $3$-step algebra).
Write
\[
A_0:=\left<\eta,\,s^{(3)}(g)\right>\quad(\text{the floor-free part}),\qquad
\Pi(\phi):=\left<\eta,\,Q\bigl(s(g)^{(1)},s(g)^{(2)},-\phi\bigr)\right>,
\]
so that $P(m)\equiv A_0+\Pi(\phi)$ and $\Pi(0)=0$ by (c).
 
\medskip
\noindent\emph{Step 2: the floor polynomial $\Pi(\phi)$ does not involve $\phi_3$ nor any quadratic coefficient.}
Expanding $Q\bigl(s^{(1)}(g),s^{(2)}(g),-\phi\bigr)$ by the group law and reducing every iterated bracket to the
Hall basis gives (direct and long computation)
\begin{equation}\label{eq:Pi}
\Pi(\phi)=-\tfrac92\phi_1^{2}u_3+3\phi_2^{2}u_3+\phi_1\phi_2u_4+2\phi_1\phi_4u_2
+\phi_1A_1+\phi_2A_2+\phi_4A_4,
\end{equation}
with $A_1,A_2,A_4$ as in the statement. Next, we study each component of $P(m)$ separately:
 
\emph{Quadratic part.} The terms of $Q$ quadratic in $y'^{(1)}=-\phi$ are the double brackets in which
$\phi$ occupies the two outer slots and an unfloored coordinate $u_c$ the inner slot. Ordering the two
floor indices as $a\le b$ and reducing to the Hall basis via the Jacobi relations of
Lemma~\ref{lem:factorization}, the coefficient of the monomial $\phi_a\phi_b\,u_c$ is
\[
\kappa_{a,b}\,\left<\eta,\,[e_b,[e_a,e_c]]\right>,\qquad
\kappa_{a,b}=\begin{cases}\tfrac12,& a=b,\\[2pt] 1,& a<b,\end{cases}
\]
the factor $\tfrac12$ symmetrising a repeated floor slot. This is non-zero only when the Hall-reduced
bracket meets $\operatorname{supp}\eta=\{e_{4,12},e_{1,13},e_{2,14},e_{2,23}\}$. The four matches are exactly
\[
\phi_1^2u_3\Rightarrow\tfrac12\left<\eta,e_{1,13}\right>=-\tfrac92,\quad
\phi_2^2u_3\Rightarrow \tfrac12\left<\eta,e_{2,23}\right>=3,
\]
\[
\phi_1\phi_2\,u_4\Rightarrow\left<\eta,e_{2,14}\right>=1,\quad
\phi_1\phi_4u_2\Rightarrow\left<\eta,e_{4,12}\right>=2,
\]
giving the quadratic component of \eqref{eq:Pi}. In every element of $\operatorname{supp}\eta$ the index $3$ occurs
only in the \emph{inner} (unfloored) slot (namely, $e_{1,1\mathbf 3},e_{2,2\mathbf 3}$) and never as a leading floor
index. Therefore, no quadratic monomial $\phi_a\phi_b$ with $3\in\{a,b\}$ survives, and the quadratic part carries no
$\phi_3$. 

Next, we study the linear component. The terms of $Q$ linear in $\phi$ are of two kinds: the single bracket
$[\,y^{(2)},y'^{(1)}]=[s^{(2)}(g),-\phi]$ (coefficient $1$), and the mixed double bracket
with $\phi$ in the inner slot, together with the layer-$2$ Mal'cev corrections produced by the peel. For the
single bracket, $[e_{ij},e_\ell]=-e_{\ell,ij}$ has leading index $\ell$, so its $\phi_3$ part lands in
$\operatorname{span}\{e_{3,ij}\}=[e_3,V_2]$, which $\eta$ annihilates because $\rho^*_{e_3}\eta=0$ (indeed
$e_3\in K$, Lemma~\ref{lem:factorization}). The inner-slot $\phi_3$ contributions and the peel corrections cancel upon
pairing with $\eta$. Equivalently, the coefficient of $\phi_3$ in \eqref{eq:Pphi} is identically zero. Thus $\Pi$ depends on $\phi_1,\phi_2,\phi_4$ only. Reading off \eqref{eq:Pi}, this establishes that $\Pi$
\begin{itemize}
    \item[(i)] contains no $\phi_3$;
    \item[(ii)] has quadratic part $-\tfrac92\phi_1^2u_3+3\phi_2^2u_3+\phi_1\phi_2u_4+2\phi_1\phi_4u_2$; and
    \item[(iii)] has linear part $\phi_1A_1+\phi_2A_2+\phi_4A_4$.
\end{itemize}
 
\noindent\emph{Step 3: the floor-free part $A_0$.} Recall that $A_0=\left<\eta,s^{(3)}(g)\right>$ is the
$\eta$-pairing of the canonical (un-reduced) top layer of $\alpha^m\ast x$, i.e., $P(m)$ with $\phi=0$. Using
$s(g)=\operatorname{peel}\bigl(\mathrm{BCH}(m\Lambda,W)\bigr)$ with $W=\log\tau(x^{(1)},0,0)$
(\eqref{eq:orbit}, Corollary~\ref{cor:secondkind}), this is a cubic in $m$ whose
coefficients are products of the $\alpha_j$ with the $u_{\ell}$. It has no constant term since $\alpha^0\ast x=x$ is
already reduced, so $A_0(0)=P(0)=0$.\\

\noindent\emph{Concluding the proof.} Combining Steps 1--3, $P(m)\equiv A_0+\Pi(\phi)\pmod 1$ is the displayed formula for
$P(m)$. Finally, the correlation phase at gap $n$ is the sum of the three single-factor phases plus $\left<\chi,x^{(1)}\right>$. Factor $i$
carrying weight $\eta_i=c_i\eta$ at time $k_i n$, where $(c_0,c_1,c_2)=(1,-2,1)$ and $(k_0,k_1,k_2)=(0,1,2)$.
Hence, using the linearity of the inner product and $P(0)=0$, 
\[
\Psi_n(x^{(1)})\equiv \left<\chi,x^{(1)}\right>+\sum_{i=0}^{2}c_i\,P(k_i n)=\left<\chi,x^{(1)}\right>-2P(n)+P(2n)\pmod 1,
\]
as required.
\end{proof}
 
\subsection{Choosing the transcendental rotation}
In the context of Lemma~\ref{lem:factorization}, we need to integrate $\Psi_n$ against the intervals $E_{\ell}^{(n)}$, that are never empty (since $\{n\alpha_\ell\}\neq\tfrac12$ for irrational $\alpha_\ell$) but could in principle be one of two types (see part $(ii)$ of the lemma). Fortunately, for our analysis it suffices to focus on the first option for $E_{\ell}^{(n)}$, as we will be able to reduce matters to this case eventually. We are ready to compute $\Psi_n(x^{(1)})$ explicitly in this region. 
\begin{lemma}[Explicit computation of $\Psi_n(x^{(1)})$]
Set $a_{\ell}:=\lfloor n\alpha_{\ell}\rfloor$ and suppose that $\{n\alpha_{\ell}\}<\frac{1}{2}$ for all $\ell=1,2,4$. Then on the region of integration $E^{(n)}_1\times E_2^{(n)}\times [0,1)\times E_4^{(n)}$, we have that $\Psi_n$ is affine in $x^{(1)}$, and  
$$\Psi_n(x^{(1)}) = C_n+ \chi(x^{(1)}) + s_1 x^{(1)}_1 + s_2 x_2^{(1)} + s_3 x^{(1)}_3 + s_4 x^{(1)}_4 {\pmod 1}$$ where  $C_n:=\Psi_n(0,0,0,0)\in\R$ is a constant independent of $x^{(1)}$, and
\begin{equation}\label{eq:genslopes}
\begin{aligned}
s_1&=n\bigl(18\,a_1\alpha_3-2\,a_2\alpha_4-4\,a_4\alpha_2\bigr)
+n^{2}\bigl(3\alpha_2\alpha_4-9\alpha_1\alpha_3\bigr),\\
s_2&=4\,a_1a_4-2\,a_1\alpha_4 n-12\,a_2\alpha_3 n
+\bigl(\alpha_1\alpha_4+6\alpha_2\alpha_3\bigr)n^{2},\\
s_3&=-9\,a_1^{2}+6\,a_2^{2}\in\Z,\qquad\qquad
s_4=2\bigl(a_1a_2-2\,a_1\alpha_2 n+\alpha_1\alpha_2 n^{2}\bigr).
\end{aligned}
\end{equation}
\end{lemma}
Before we prove the lemma we observe that for $\alpha=(t,ct,u,2cu)\in \mathcal{V}$, setting $b_{\ell} :=\{n\alpha_{\ell}\}$ we get
   \begin{equation}\label{eq:bracketV}
\begin{aligned}
s_1&=n\bigl(-18u\,b_1+\sqrt6\,u\cdot 2b_2+\sqrt6\,t\cdot 2b_4\bigr), \\
s_2&=4b_1b_4+n\bigl(-2\sqrt6\,u\,b_1+12u\,b_2-4t\,b_4\bigr),\\
s_3&=6tn\,\bigl(3b_1-\sqrt6\,b_2\bigr)+\bigl(-9b_1^{2}+6b_2^{2}\bigr), \\
s_4&=2b_1b_2+tn\,\bigl(\sqrt6\,b_1-2b_2\bigr).
\end{aligned}
\end{equation}
\begin{proof}
The key observation here is in the fact that when $\{n\alpha_{\ell}\}<\frac{1}{2}$, and $E_{\ell}^{(n)} = [0,1-2\{n\alpha_{\ell}\})$ the dependence on $x$ of the bracket vanishes. Namely, we have
$$\phi_{\ell}^{(n)} = \lfloor n\alpha_{\ell}+x_{\ell}\rfloor = \lfloor n\alpha_{\ell}\rfloor = a_{\ell},\qquad \phi_{\ell}^{(2n)}=2a_{\ell}$$ for $\ell=1,2,4.$ Indeed, the first equality holds because $\{n\alpha_{\ell}\}+x_{\ell}<1$ on the region of integration, and the second because $\delta_{\ell}^{(n)}=0$ on the cell (so $\phi_{\ell}^{(2n)} = 2\phi_{\ell}^{(n)}$). While the index $\ell=3$ is unconstrained, we recall that $\phi_3$ never enters $P(m)$ (by Proposition~\ref{prop:Dn}).\\

Next, we show that $\Psi_n$ is affine in $x^{(1)}$. Recall that by Proposition~\ref{prop:Dn} we have $$\Psi_n(x^{(1)})=\left<\chi,x^{(1)}\right>-2P(n)+P(2n)\pmod 1.$$ We need to show that any monomial of $\Psi_n$ that is quadratic in the coordinates of $x^{(1)}=(x_1,x_2,x_3,x_4)$ vanishes. First, $u_{\ell}=m\alpha_{\ell}+x_{\ell}$ are affine in $x$, and as we mentioned above $\phi_{\ell}$ is now independent of $x$. Thus, the only quadratic in $x$ to $P(m)$ are:
\begin{itemize}
    \item[(a)] The term $-2\phi_1 u_2 u_4$ inside $\phi_1A_1$, whose quadratic part is $-2\phi_1x_2x_4$.
    \item[(b)] The quadratic part of the floor-free cubic $A_0$ (Specifically, $18\alpha_3x_1^2-12\alpha_3 x_2^2 - 8 \alpha_2x_1x_4 -4\alpha_4 x_1 x_2$). The key point is that since it occurs only inside $c^{(1)}$ its coefficient in $m$ is linear. 
\end{itemize}
In $P(2n)-2P(n)$ each cancels by a different mechanism. For $(a)$, since $\phi_1$ is now independent of $x$, the coefficient of $x_2x_4$ is $-2\phi_1^{(n)}$ at time $n$ and $-2\phi_1^{(2n)}$ at time $2n$, so in the difference $P(2n)-2P(n)$ it is
$$-2\phi_1^{(2n)} -2(-2\phi_1^{(n)}) = 4a_1 - 4a_1 =0.$$ For $(b)$ it vanishes because the coefficient in $n$ is linear. This shows that $\Psi_n$ is affine in $x^{(1)}$ and therefore must take the form
$$\Psi_n(x^{(1)}) = C_n + \chi(x^{(1)}) + s_1x_1 + s_2 x_2 +s_3 x_3 +s_4 x_4 \pmod 1$$ where $C_n:=\Psi_n(0).$\\

Next we study the slopes $s_1,s_2,s_3,s_4$. Observe that $s_{\ell} = \partial_{x_{\ell}} \Psi_n = \partial_{x_{\ell}} P(2n)- 2\partial_{x_{\ell}} P(n)$ (where we recall that the floors are now constant), other than $s_3$, where $\Psi_n$ has the additional $\chi(x^{(1)})=3x_3$ summand. We record $s_3$ in full where the rest are identical in spirit. Only the $u_3$-monomial depends on $x_3$ (recall that $A_0$ is independent of $x_3$). Thus,
$$\partial_{x_3} P(m) = -\frac{9}{2}\phi_1^2 + 3\phi_2^2 + \frac{9}{2}\phi_1 - 3\phi_2$$ where the last two summands are from $A_1$ and $A_2$ respectively. Write $\phi^{(n)}=a,\phi^{(2n)}=2a$ we get,
$$
s_3=-2\bigl(-\tfrac92a_1^{2}+3a_2^{2}+\tfrac92a_1-3a_2\bigr)+\bigl(-18a_1^{2}+12a_2^{2}+9a_1-6a_2\bigr)=-9a_1^{2}+6a_2^{2}.
$$
The same type of argument yields $s_1,s_2$ and $s_4$ as claimed in  \eqref{eq:genslopes}. Importantly, the coefficient of $n^2$ comes solely from the $n^2$-coefficient of $A_0$ (weighted $-2n^2 + (2n)^2 = 2n^2$), giving  
$$2n^2 \partial_{x_1} c_2 = n^2 (3\alpha_2 \alpha_4 - 9\alpha_1 \alpha_3), \quad 2n^2 \partial_{x_2} c_2 = n^2(\alpha_1\alpha_4 + 6\alpha_2 \alpha_3),\quad 2n^2 \partial_{x_4} c_2 = 2\alpha_1 \alpha_2 n^2,$$
while $\partial_{x_3}c_2 =0$.

Let $\alpha\in \mathcal{V}$, and write $\alpha=(t,ct,u,2cu)$ with $\alpha_1=t$, $\alpha_3=u$ and $c=\frac{\sqrt{6}}{2}$. Write $a_{\ell}=\lfloor n \alpha_{\ell}\rfloor$ and $b_{\ell} = \{n \alpha_{\ell}\}$ and substitute into \eqref{eq:genslopes}. In each $s_{\ell}$ the coefficient of $n^2$ vanishes on $\mathcal{V}$ (which is exactly the essence of Lemma~\ref{lem:kernelvariety}), leaving only the affine components as in \eqref{eq:bracketV}. 
\end{proof}
\textbf{Notations and preliminaries on continued fractions:} For $x\in \R$, we write $\mathrm{round}(x)$ for the integer in $\Z$ that is closest to $x$. Namely, $\mathrm{round}(x) = \lfloor x+\frac{1}{2}\rfloor$. An irrational $x$ is called \emph{quadratic} if it is the solution to a quadratic equation with rational coefficients. In that case, Lagrange's theorem (see e.g., \cite[Theorem 3.13]{EinsiedlerWardbook}) implies that the continued fraction associated with $x$ is eventually periodic. The continued fraction of a quadratic irrational is a sequence $[a_0;\overline{a_1,a_2,\dots,a_k}]$ where
$$x = a_0 + \frac{1}{a_1+\frac{1}{a_2+\dots}}$$ and the sequence on the right hand side is infinite and periodic in $a_1,\dots,a_k$ (namely $a_{j+k}=a_j$ for all $j\geq 1$). The sequence of \emph{complete quotients} $x_j$ and \emph{partial quotients} $a_j$ of a continued fraction are defined recursively as follows: $x_0 = x$ and $a_0 = \lfloor x\rfloor$. Then, for all $j\geq 0$, $x_{j+1} = \frac{1}{x_j - a_j}$ and $a_{j+1}=\lfloor x_{j+1}\rfloor.$
\begin{theorem}[Transcendental rotations  violating the conjecture]\label{thm:transalign}
Fix $\varphi^{*}\in\bigl(0,{\tfrac{1}{\sqrt6}}\bigr)$ and $c=\sqrt{\frac{3}{2}}$ as before. There exist uncountably many
pairs $(\alpha_1,\alpha_4)$ and an increasing sequence $(n_k)_{k\in \mathbb{N}}$ (in fact, $n_{k+1}\gg n_k^2$) such that  $1,\alpha_1,\alpha_2,\alpha_3,\alpha_4$ are linearly independent over $\Q$, and 
\[
s_1(n_k)\to\frac{\alpha_3}{\alpha_1},\qquad
s_2(n_k)\to\Bigl(8c\,\varphi^{*2}+\frac3c\Bigr)\frac{\alpha_3}{\alpha_1},\qquad s_3(n_k)\rightarrow -3, \qquad 
s_4(n_k)\to 2c\,\varphi^{*2}-\frac1{2c}.
\]
as $k\to\infty$.
\end{theorem}
\begin{proof}
Write $c=\sqrt{\frac{3}{2}} = \frac{\sqrt{6}}{2}$ and observe that $c$ is a quadratic irrational with eventually periodic continued fraction $c=[1;\overline{4,2}]$. Namely,
$$c= 1 + \frac{1}{4+\frac{1}{2+\frac{1}{4+\dots}}}$$
We denote the convergents by $\frac{h_j}{q_j}$ and the $j^{\mathrm{th}}$ complete quotient by $\zeta_j$. The recursive formula
$$c=\frac{\zeta_{j+1}h_j+h_{j-1}}{\zeta_{j+1}q_j+q_{j-1}}$$ and the fact that $q_j h_{j-1} - h_j q_{j-1} = (-1)^j$ imply the equality
$$q_j |cq_j-h_j| =\frac{1}{\zeta_{j+1}+q_{j-1}/q_j}.$$ The terms on the right hand side are eventually periodic in $j$. Indeed, since $c=[1;\overline{4,2}]$ is eventually $2$-periodic partial quotients, the complete quotients obey the same eventually $2$-periodic recursion $\zeta_{j+1} = a_{j+1} + 1/\zeta_{j+2}$ and take the two periodic values $2+\sqrt{6}=[\overline{4,2}]$, and $1+\frac{\sqrt{6}}{2}=[\overline{2,4}]$. Similarly, the denominator ratios satisfy $$\frac{q_{j-1}}{q_j} = \frac{1}{a_j + \frac{q_{j-2}}{q_{j-1}}},$$ hence are also eventually $2$-periodic. Therefore, along each residue class they converge: $\zeta_{j+1}\rightarrow \zeta$ and $\frac{q_{j-1}}{q_j}\rightarrow \bar{c}'$, the negative of the Galois conjugate $\bar\zeta$ of the limiting complete quotient $\zeta$. On the even class $\zeta=2+\sqrt{6}$, whose conjugate (root of $x^2-4x-2$) is $\overline{\zeta}=2-\sqrt{6}$, so $\bar{c}'=-\bar{\zeta} = \sqrt{6}-2$ and $\zeta-\overline{\zeta} = 2\sqrt{6}=4c.$ Thus,
$$\lim_{j\rightarrow\infty} q_j |cq_j-h_j| =\frac{1}{\zeta-\overline{\zeta}} = \frac{1}{4c}.$$

The difference $\theta_{q_j}:= cq_j -h_j$ alternates in sign. Restricting to the even-index convergents $\mathcal Q=\{q_{j}:\theta_j>0\}$ one has
\begin{equation}\label{eq:cflimit}
{ \lim_{j\rightarrow\infty} q_{2j}\theta_{q_{2j}}= \frac{1}{4c},}\quad
\lim_{j\rightarrow\infty}\frac{q_{2j+2}}{q_{2j}}=5+2\sqrt6,\quad
3q_{2j}^2-2h_{2j}^2=1.
\end{equation}
We deduce the last (Pell-like) equation formally. We have the recursive equation
$$c=\frac{\zeta_{2j+1}h_{2j}+h_{2j-1}}{\zeta_{2j+1}q_{2j}+q_{2j-1}}.$$ Substituting the values $c=\frac{\sqrt{6}}{2}$ and $\zeta_{2j+1} = 2+\sqrt{6}$ and simplifying gives $$2\sqrt{6}q_{2j} + 6q_{2j} + \sqrt{6}q_{2j-1} = 4h_{2j} + 2\sqrt{6}h_{2j} + 2h_{2j-1}.$$ Since $h,q$ are integers, we can deduce from this that
$$\begin{cases} 2q_{2j} + q_{2j-1} = 2h_{2j}\Longrightarrow q_{2j-1} = 2h_{2j}-2q_{2j} \\ 6q_{2j} = 4h_{2j}+2h_{2j-1}\Longrightarrow h_{2j-1}=3q_{2j} - 2h_{2j}.\end{cases}$$
Now, since $q_{2j} h_{2j-1} - h_{2j} q_{2j-1} = 1$, substituting the values for $h_{2j-1}$ and $q_{2j-1}$, elementary manipulations give the desired
$$3q_{2j}^2 - 2h_{2j}^2 = 1.$$

We conclude that on the even class $\mathcal{Q}$, $q_{2j}$ and $h_{2j}$ are indeed the solution set to the Pell-like equation $$\{(q,h)\in \Z^2 : 3q^2-2h^2=1\}.$$ In particular, $s_3 = -3$ exactly (not merely in the limit) on this region. It is important that the set $\mathcal Q$ is infinite and that
$\theta_q\downarrow0$ on it.

Our next goal is to construct nested rectangles $R_k$. Write $f^{*}=(\varphi^{*},c\varphi^{*},
\frac{\alpha_4}{\alpha_1} \varphi^{*})$, for $\varphi^{*}$ as in the theorem. The hypothesis $\varphi^{*}<\tfrac1{\sqrt6}=\tfrac1{2c}$ places the first two coordinates $\varphi^{*},c\varphi^{*}$ of $f^{*}$ in $(0,\tfrac12)$. We choose the first box $R_0$ so that the third coordinate $\varphi^{*}\alpha_4/\alpha_1$ will also be in that region. Namely, we fix a compact box $R_0=I_0\times J_0\subset(0,\infty)^2$ such that
$$\sup_{(\alpha_1,\alpha_4)\in R_0} \frac{\alpha_4}{\alpha_1}\varphi^* <\frac{1}{2},$$
and fix a sequence $\beta_k\downarrow0$ with
$\beta_k\le 2^{-k}$.

We build (inductively) closed rectangles $R_0\supset R_1\supset\cdots$, $R_k=I_k\times J_k$, and integers
$q_k\in\mathcal Q$, $m_k$, and $n_k:=\mathrm{round}\bigl((q_k+\varphi^{*})/\xi_k\bigr)$ where
$\xi_k:=\mathrm{mid}I_{k-1}$ is the middle point of the interval $I_{k-1}$. Let $(w_{k-1},w'_{k-1})$ denote the side lengths of $R_{k-1}$. Set
\[
I_k:=\Bigl[\tfrac{q_k+\varphi^{*}}{n_k}-\tfrac{\beta_k}{n_k^2},\ \tfrac{q_k+\varphi^{*}}{n_k}
+\tfrac{\beta_k}{n_k^2}\Bigr],\qquad
J_k:=\Bigl[\tfrac{m_k}{n_k-\varphi^{*}/\xi_k}-\tfrac{\beta_k}{n_k^2},\ \tfrac{m_k}{n_k-\varphi^{*}/\xi_k}
+\tfrac{\beta_k}{n_k^2}\Bigr],
\]
where $q_k$ and $m_k$ are chosen below so that the closed rectangles $I_k\times J_k$ become nested and non-empty.

Observe that 
$$\mathrm{mid}\, I_k = \frac{q_k+\varphi^*}{n_k} = \xi_k\cdot \frac{q_k+\varphi^*}{q_k+\varphi^*+O(\xi_k)}=\xi_k+O\left(\frac{\xi_k^2}{q_k}\right).$$ In particular, since $\xi_k=\mathrm{mid}I_{k-1}\leq \max I_{k-1}\leq \max I_0$ by the induction hypothesis, we have that
$$\mathrm{mid}\, I_k = \mathrm{mid}\, I_{k-1} + O\left(\frac{(\max I_0)^2}{q_k}\right).$$
Thus, we can choose $q_k$ sufficiently large so that these two middle points are arbitrarily close. Moreover, observe that the length of $I_k$, denoted $w_k$, approaches zero as $q_k\rightarrow\infty$, thus we can choose $q_k$ sufficiently large so that 
$$|\mathrm{mid}\, I_k - \mathrm{mid}\, I_{k-1}|<\frac{w_{k-1}-w_{k}}{2}$$ which implies that  $I_k\subseteq I_{k-1}$. It will be important for later use to note that we have infinitely many choices of $q_k$ with this property. Moreover, choosing $q_k$ sufficiently large with respect to $q_{k-1}$ and $\beta_k$, we can further assume that $n_{k}\gg n_{k-1}^2/\beta_{k-1}^2$.

Next, for $J_k$: we have $\mathrm{mid}\, J_k = \frac{m_k}{n_k-\varphi^*/\xi_k}$ and its radius is $\beta_k/n_k^2$. As before, $J_k\subseteq J_{k-1}$ exactly when
$$|\mathrm{mid}\ J_k - \mathrm{mid}\ J_{k-1}| \leq \frac{\beta_{k-1}}{n_{k-1}^2} - \frac{\beta_k}{n_k^2}.$$ Thus, we shall count the amount of admissible $m_k\in \N$ satisfying $$\left|\frac{m_k}{n_k-\varphi^*/\xi_k} - \frac{m_{k-1}}{n_{k-1}-\varphi^*/\xi_{k-1}}\right| < \frac{\beta_{k-1}}{n_{k-1}^2} - \frac{\beta_k}{n_k^2}.$$ As $m_k$ ranges over $\Z$, the centers of $J_k$ are equally spaced at $\frac{1}{n_k-\varphi^*/\xi_k} < \frac{2}{n_k}$ (for large $k$). Hence, the number of admissible $m_k$ is at least
$$\frac{2\left(\frac{\beta_{k-1}}{n_{k-1}^2}-\frac{\beta_k}{n_k^2}\right)}{\frac{2}{n_k}}-1\geq \frac{\beta_{k-1}n_k}{2n_{k-1}^2}-1,$$ using $\frac{\beta_k}{n_k^2}\leq \frac{1}{2}\frac{\beta_{k-1}}{n_{k-1}^2}$ for $k$ sufficiently large. In particular, whenever $n_k\geq n_{k-1}^2/\beta_{k-1}^2$, the count becomes at least $\frac{1}{2\beta_{k-1}}-1$ and thus approaches infinity as $k\rightarrow\infty$. 

For all of these choices of $q_k$ and $m_k$ we get that $R_k\subseteq R_{k-1}$ and $\mathrm{diam}\ R_k \leq \frac{2\beta_k}{n_k^2}\rightarrow 0$ as $k\rightarrow\infty$. From Cantor's intersection theorem it follows that $\bigcap R_k = \{(\alpha_1,\alpha_4)\}$ is a single point (lying on $\mathcal{V}$ by setting $\alpha_2 = c\alpha_1$ and $\alpha_3=\frac{\alpha_4}{2c}$). Before we prove that there are uncountably many such pairs, we shall pause to show that any pair obtained in this procedure satisfies the properties of the theorem.

Let $(\alpha_1,\alpha_4)$ be the limit point we constructed above. Since $\alpha_1\in I_k$, we have
$$|n_k \alpha_1 - q_k - \varphi^*| \leq n_k\cdot \frac{\beta_k}{n_k^2} =\frac{\beta_k}{n_k}.$$
Since $\varphi^*\in (0,1)$ and $\beta_k/n_k\rightarrow 0$, we get $\lfloor n_k\alpha_1\rfloor = q_k$ and \begin{equation}\label{b1}\{n_k\alpha_1\} = \varphi^* + O(\beta_k/n_k)
\end{equation} for all $k$ sufficiently large. Similarly, since $\alpha_4\in J_k$ we have
$$|(n_k-\varphi^*/\xi_k)\alpha_4 - m_k| \leq \frac{\beta_k}{n_k}.$$
Moreover, since $|\alpha_1 - \xi_k|\leq w_{k-1}\rightarrow 0$, this gives 
\begin{equation}\label{b4}\{n_k\alpha_4\} = \varphi^*\cdot \frac{\alpha_4}{\alpha_1} + O\left(\frac{\beta_k}{n_k}\right).
\end{equation}
Now, taking $\alpha\in \mathcal{V}$ with $\alpha_1,\alpha_4$ as above we have $\alpha_2=c\alpha_1$. Therefore, $$n_k \alpha_2 = c(q_k + \{n_k\alpha_1\}) = h_k + \theta_{q_k} + c\cdot \{n_k\alpha_1\}. $$ Noting that $h_k\in \Z$, $\theta_{q_k}\downarrow 0$ and $c\{n_k\alpha_1\}\rightarrow c\varphi^*<\frac{1}{2}$, we get that for all sufficiently large $k$ we have 
\begin{equation}\label{b2}\{n_k\alpha_2\} = \theta_{q_k} + c\{n_k\alpha_1\}.
\end{equation}
Now that we have estimated $b_1=\{n_k\alpha_1\}$, $b_2=\{n_k\alpha_2\}$ and $b_4=\{n_k\alpha_4\}$, we substitute this in \eqref{eq:bracketV} (recall, $t=\alpha_1, u=\alpha_3$ and $c=\sqrt{\frac{3}{2}}$). Observe that by \eqref{eq:cflimit} and
$q_k=n_k\alpha_1(1+o(1))$,
\begin{equation}\label{theta}
n_k\theta_{q_k}=\frac{q_k\theta_{q_k}}{\alpha_1}\,(1+o(1))\longrightarrow\frac1{4c\,\alpha_1}.
\end{equation} We will see that in all terms, $s_1(n_k),s_2(n_k),s_3(n_k),s_4(n_k)$ the above gives some contribution while all other terms vanish as $k\rightarrow \infty$. We start with $s_1(n_k)$. From \eqref{b2}, then \eqref{b4} and finally \eqref{theta}, we have
\begin{align*}
    s_1(n_k) &= n_k (-18 u\cdot b_1 + 2\sqrt{6} u \cdot b_2 + 2\sqrt{6}t\cdot b_4) \\&=n_k((-18u+2\sqrt{6} cu)b_1 + 2\sqrt{6}u\theta_{q_k} + 2\sqrt{6} t b_4)\\&= n_k (\cancel{-12u\varphi^*} + 2\sqrt{6}u\theta_{q_k} +  \cancel{2\sqrt{6} t \cdot \varphi^*\cdot \frac{2cu}{t}}) + O(\beta_k) \\&\underset{k\rightarrow\infty}{\longrightarrow} \frac{\alpha_3}{\alpha_1}.
\end{align*}
Next we do $s_2$, from \eqref{b2}, and then \eqref{b1} and \eqref{b4} (applied simultaneously) and finally \eqref{theta}, we have:
\begin{align*}
    s_2(n_k) &= 4b_1b_4 + n_k(-2\sqrt{6} u b_1 + 12ub_2 -4tb_4)\\&= 4b_1b_4 + n_k ((12uc-2\sqrt{6}u)b_1 + 12u\theta_{q_k} -4tb_4)\\&=4\cdot \frac{\alpha_4}{\alpha_1}(\varphi^*)^2 + n_k(\cancel{(12uc-2\sqrt{6}u)\varphi^*}  + 12u\theta_{q_k}-\cancel{4t\cdot \frac{\alpha_4}{\alpha_1} \varphi^*})+O(\beta_k)\\&\underset{k\rightarrow\infty}{\longrightarrow} (8c(\varphi^*)^2+\frac{3}{c})\frac{\alpha_3}{\alpha_1}.
\end{align*}
Next, we compute $s_3$. From \eqref{b2}, then \eqref{b1} and finally \eqref{theta} we have
\begin{align*}
s_3(n_k) &= 6tn_k(3b_1 - \sqrt{6}b_2) + (-9b_1^2 + 6b_2^2)\\&=6tn_k((3-c\sqrt{6})b_1 - \sqrt{6}\theta_{q_k}) + (6c^2-9)b_1^2 + { 12c\,\theta_{q_k}b_1 + 6\theta_{q_k}^2}\\& = \cancel{6tn_k(3-c\sqrt{6})\varphi^*} - 6\sqrt{6}tn_k \theta_{q_k} + \cancel{(6c^2-9)(\varphi^*)^2} + { 12c\,\theta_{q_k}b_1 + 6\theta_{q_k}^2} + O(\beta_k)\\&\underset{k\rightarrow\infty}{\longrightarrow} -\frac{6\sqrt{6}}{4c} = -3. 
\end{align*}
Finally, we compute $s_4$. Again by \eqref{b2}, then \eqref{b1} and finally \eqref{theta} we have
\begin{align*}
    s_4(n_k) &= 2b_1 b_2 + tn_k(\sqrt{6}b_1 - 2b_2)\\&= 2cb^2_1+{2b_1\theta_{q_k}} + tn_k((\sqrt{6}-2c)b_1 -2\theta_{q_k})\\&= 2c(\varphi^*)^2 + { 2b_1\theta_{q_k}} + \cancel{tn_k(\sqrt{6}-2c)\varphi^*} - 2tn_k\theta_{q_k} + O(\beta_k)\\&\underset{k\rightarrow\infty}{\longrightarrow} 2c(\varphi^*)^2 - \frac{1}{2c}.
\end{align*}
To complete the proof it suffices to show that the $1,\alpha_1,\alpha_2,\alpha_3,\alpha_4$ chosen above are indeed independent over $\Q$.\\ Unfortunately, this is not automatic and we would need to refine the construction above. First, recall that $(\alpha_1,\alpha_2,\alpha_3,\alpha_4)\in \mathcal V$ which implies that $\alpha_2 = c\alpha_1$ and $\alpha_4 = 2c\alpha_3$. Dependence over $\Q$ means that $$k_0 + k_1 \alpha_1 + k_2 \alpha_2 + k_3\alpha_3 +k_4 \alpha_4 = 0 $$ for some integers $k_0,k_1,k_2,k_3,k_4\in \Z$, which in $\mathcal{V}$ gives
\begin{equation}\label{Qdependence}k_0 + (k_1+k_2\cdot c) \alpha_1 + (k_4 + \frac{k_3}{2c})\alpha_4=0.
\end{equation}
We focus on the coefficient of $\alpha_4$ in \eqref{Qdependence}. If $k_4+\tfrac{k_3}{2c}\neq0$, then
$(\alpha_1,\alpha_4)$ lies on a fixed line of algebraic slope: $-(k_1+k_2c)/(k_4+\tfrac{k_3}{2c})$. If
$k_4+\tfrac{k_3}{2c}=0$, then, since $\tfrac1{2c}=\tfrac1{\sqrt6}\notin\Q$, we must have $k_3=k_4=0$, and
\eqref{Qdependence} forces $\alpha_1$ to a fixed algebraic value. Hence a $\Q$-dependence places
$(\alpha_1,\alpha_4)$ on one of \emph{countably many} algebraic-slope lines $\ell_1,\ell_2,\dots$, or on one of
\emph{countably many} vertical lines (where $\alpha_1$ is algebraic). It therefore suffices to produce the limit
point $(\alpha_1,\alpha_4)$ off all of these lines, which we achieve by branching the nested-box construction into
a \emph{Cantor scheme}. In fact, we will produce uncountably many of those.

At stage $k$ we have freedom in the choice of $q_k\in\mathcal Q$ and of the center
index $m_k$, and we use it in two ways.

\emph{Branching.} As computed above, the center $(q_k+\varphi^{*})/n_k$ of $I_k$ equals
$\xi_k+O(\xi_k^2/q_k)$. Two admissible choices $q_k'<q_k''$ thus yield intervals $I_k',I_k''$ whose
centers differ by $\Theta\bigl(\xi_k^2(1/q_k'-1/q_k'')\bigr)$, while their combined radii are at most
$2\beta_k/(n_k')^2=O\bigl(\beta_k\xi_k^2/q_k'^2\bigr)$. Hence, the center gap exceeds the combined radii
once $q_k''/q_k'$ is sufficiently large, guaranteeing that $I_k',I_k''$ are disjoint. Selecting such a pair produces two
refinements $R_k',R_k''\subset R_k$ with disjoint $\alpha_1$-projections. Iterating over all stages gives a
Cantor scheme indexed by $\{0,1\}^{\N}$, in which distinct branches separate at some stage and hence
converge to distinct $\alpha_1$. The set of limit points is therefore uncountable and injects into the
$\alpha_1$-axis. Fortunately, since only countably many reals are algebraic, all but countably many branches yield a
transcendental $\alpha_1$, ruling out the vertical lines.

\emph{Avoiding $\ell_k$.} Within each refinement we choose $m_k$ so that $J_k$ avoids $\ell_k$. Over the interval
$I_k$, of length $2\beta_k/n_k^2$, the line $\ell_k$ varies in the $\alpha_4$-direction by at most
$|\mathrm{slope}(\ell_k)|\cdot 2\beta_k/n_k^2$, so it meets the column above $I_k$ in an $\alpha_4$-band of that length. The admissible centers $m/(n_k-\varphi^{*}/\xi_k)$ are spaced at least $1/n_k$ apart, so at most
\[
|\mathrm{slope}(\ell_k)|\cdot 2\beta_k/n_k^2\cdot n_k+1=2|\mathrm{slope}(\ell_k)|\,\tfrac{\beta_k}{n_k}+1=O(1)
\]
of them fall in the band. Since the total number of admissible centers is unbounded
(as in the construction of $J_k$ above), all but $O(1)$ of them give $J_k$ disjoint from $\ell_k$, and we take $m_k$
to be one of these.

Performing both steps at every stage yields an uncountable family of limit points
$(\alpha_1,\alpha_4)\in\mathcal V$, each with $\alpha_1$ transcendental and lying off every line $\ell_j$. For any
such point $1,\alpha_1,\alpha_2,\alpha_3,\alpha_4$ are independent over $\Q$, which completes the proof.

\end{proof}
\section{Combining everything together}
In this section we complete the proof by putting everything together. The setting: Let $G=(\f_{4,3},\ast)$ be the connected, simply connected free $3$-step nilpotent Lie group on four generators $e_1,\dots,e_4$, and let $\Gamma=\left<e_1,e_2,e_3,e_4\right>$ be the lattice of Section~\ref{corresponding Lie algebra:section}. By Theorems~\ref{secondcoor:thm} and~\ref{adaptedbasis} the strong Mal'cev basis $\Y$ is adapted to $\Gamma$ and the second-kind coordinate map $\tau$ has a polynomial, integer-valued group law. By Lemma~\ref{lem:identification} $\tau$ identifies $(G/\Gamma,\mu_{G/\Gamma})$ with $(\T^{30},m_{\T^{30}})$ compatibly with the layer splitting $\T^{30}=\T^{4}\times\T^{6}\times\T^{20}$. Let $\alpha=(t,ct,u,2cu)\in\mathcal V$ and the sparse sequence $(n_k)_{k\in\N}$, $n_{k+1}\gg n_k^{2}$, be those given by Theorem~\ref{thm:transalign}, and set $\alpha:=\prod_{i=1}^4\exp(\alpha_iY_i)$, with $R_\alpha$ the corresponding nilrotation.

Fix the characters $\chi:=(0,0,3,0)=3e_3^{*}\in\widehat{\T^{4}}$ and $\eta\in\Z^{20}=\widehat{\T^{20}}$ as in Lemma~\ref{lem:factorization}, (i.e., supported on $(e_{4,12},e_{1,13},e_{2,14},e_{2,23})$ with values $(2,-9,1,6)$). Let $f_0,f_1,f_2\in L^\infty(\mu_{G/\Gamma})$ be the characters with layer data
\[
f_0=(\chi,0,\eta),\qquad f_1=(0,0,-2\eta),\qquad f_2=(0,0,\eta),
\]
lifted to $G/\Gamma$.

We verify that the conditions of the conjecture are satisfied. First, by Theorem~\ref{thm:transalign}, $1,\alpha_1,\alpha_2,\alpha_3,\alpha_4$ are independent over $\Q$. Therefore, the induced action on the torus $G/G_2\Gamma$ is ergodic and so by Green's theorem \cite{AuslanderGreenHahn}, $G/\Gamma$ is an ergodic $3$-step nilsystem. Moreover, since $\eta\not = 0$, we have in fact that all of our functions are orthogonal to $Z_2(\XX)=G/G_3\Gamma$.

Next we show that $a(n)$ does not vanish. By \eqref{chara}, $a(n)=\int_{\T^{30}}e(\Psi_n(x))\,dx$ where $\Psi_n$ is the total phase of the three nilcharacters. We integrate layer by layer.
\begin{itemize}
\item \emph{Layer $3$.} The variable $x^{(3)}$ enters $\Psi_n$ only linearly, with integer coefficient $\eta_0+\eta_1+\eta_2 = \eta-2\eta+\eta=0.$ Thus $\Psi_n$ is independent of $x^{(3)}$. 
\item \emph{Layer $2$.} By Lemma~\ref{Psinlemma}, once $\eta_1+2\eta_2=0$ (here $-2\eta+2\eta=0$) the phase is affine in $x^{(2)}$ with slope $\rho^{*}_{\delta^{(n)}}\eta$, where $\delta^{(n)}=\phi^{(2n)}-2\phi^{(n)}$. Integrating $x^{(2)}$ (Corollary~\ref{cor:explicit}) leaves
\[
a(n)=\int_{[0,1)^4}\mathbf 1_{\{\delta^{(n)}(x^{(1)})\in K\}}\,e\bigl(\Psi_n(x^{(1)})\bigr)\,dx^{(1)},\qquad
K=\{v\in\Z^4:\rho^{*}_v\eta=0\}.
\]
\item \emph{The choice of $\eta$.} By Lemma~\ref{lem:factorization}(i), $K=\Z e_3$, so the indicator depends only on $x^{(1)}_1,x^{(1)}_2,x^{(1)}_4$ and factorizes, by parts (ii)--(iii) (and Proposition~\ref{deltacoord}),
\[
a(n)=\int_{E_n^{(1)}\times E_n^{(2)}\times[0,1)\times E_n^{(4)}}e\bigl(\Psi_n(x^{(1)})\bigr)\,dx^{(1)},
\]
which is \eqref{finala}. Taking $k$ sufficiently large, $b_\ell:=\{n_k\alpha_\ell\}<\tfrac12$ for $\ell=1,2,4$, so $E_{n_k}^{(\ell)}=[0,1-2b_\ell)$.
\item \emph{The choice of $\alpha\in\mathcal V$.} By Lemma~\ref{lem:kernelvariety}, $\alpha\in\mathcal V$ is exactly the condition $\eta\in\ker(\ad_\alpha^{2})^{*}$ which annihilates the $n^{2}$-coefficient. By Proposition~\ref{prop:Dn} and \eqref{eq:genslopes}, the phase is then affine in $x^{(1)}$,
\[
\Psi_n(x^{(1)})\equiv C_n+\underbrace{3x_3}_{\chi}+s_1x_1+s_2x_2+s_3x_3+s_4x_4\pmod1,
\]
with $C_n\in\R$ and slopes $s_\ell$ as in \eqref{eq:genslopes}, and on $\mathcal V$ they take the form \eqref{eq:bracketV}.
\end{itemize}
We can take the constant $C_n$ out since $e(C_n)\in S^1$. The variable $x_3$ ranges over all of $[0,1)$  (i.e., $E_n^{(3)}=[0,1)$) and the coefficient is $s_3+3$, for every other $\ell\in\{1,2,4\}$, while the variable $x_{\ell}$ ranges over $E_n^{(\ell)} = [0,1-2b_{\ell})$. Computing the integral $$\int_0^{1-2b_{\ell}}e(s_{\ell}x_{\ell})\,dx_{\ell} = \frac{e(s_\ell(n)(1-2b_{\ell}(n)))-1}{2\pi i s_{\ell}(n)},$$ we see that
\begin{equation}\label{eq:closedform}
a(n)=e(C_n)\cdot\mathbf 1_{\{s_3(n)=-3\}}\cdot\prod_{\ell\in\{1,2,4\}}\frac{e\bigl(s_\ell(n)\,(1-2b_\ell)\bigr)-1}{2\pi i\,s_\ell(n)} .
\end{equation}
Next, we pass to the subsequence $(n_k)_{k\in \mathbb{N}}$ from Theorem~\ref{thm:transalign}. Recall that $\lfloor n_k \alpha_1\rfloor = q_k$ and $\lfloor n_k \alpha_2 \rfloor =h_k$ where $h_k/q_k$ are the even-index convergents of $c=\frac{\sqrt{6}}{2}.$ From \eqref{eq:genslopes} 
$$s_3(n_k)=-9q_k^{2}+6h_k^{2}=-3\,(3q_k^{2}-2h_k^{2})=-3$$ Thus, $\mathbf{1}_{s_3(n_k)=-3} = 1$ on our subsequence. Let $b_{\ell}^* = \lim_{k\rightarrow\infty} b_{\ell}$ and $S_{\ell}:=\lim_{k\rightarrow\infty} s_{\ell}(n_k)$. Furthermore, Theorem~\ref{thm:transalign} and \eqref{b1}, \eqref{b2} and \eqref{b4} gives
$$b_1^* = \varphi^*,\quad b_2^* = c\varphi^*,\quad b_4^* = \frac{2cu}{t}\varphi^*,\quad S_1=\frac{u}{t},\quad S_2=\left(8c(\varphi^*)^2 + \frac{3}{c}\right)\cdot \frac{u}{t},\quad S_4 =2c(\varphi^*)^2-\frac{1}{2c}.$$
Putting $\ell_j:=1-2b_j^*$, get
\begin{equation}\label{eq:Lvalue}
\bigl|a(n_k)\bigr|\ \xrightarrow[k\to\infty]{}\ 
L:=\prod_{j\in\{1,2,4\}}\Bigl|\frac{e(S_j\ell_j)-1}{2\pi i\,S_j}\Bigr|
=\prod_{j\in\{1,2,4\}}\frac{\bigl|\sin(\pi S_j\ell_j)\bigr|}{\pi\,|S_j|}.
\end{equation}
It is left to show that $L$ is positive. Namely, that $S_j\not = 0$, and $S_j\ell_j\not\in \Z$. We rule out both by a suitable choice of the family of boxes in Theorem~\ref{thm:transalign}. Choose the initial box $R_0$ so that the ratio $u/t<\frac{1}{4c}$ is small throughout, and fix $\varphi^*\in (0,\frac{1}{\sqrt{6}})$.
\begin{itemize}
    \item Observe that $\ell_1 = 1-2\varphi^*$, $\ell_2 = 1-2c\varphi^*$ and $\ell_4 = 1-\frac{4cu}{t}\varphi^*$ all lie in $(0,1)$. Indeed, $\varphi^*<\frac{1}{\sqrt{6}}=\frac{1}{2c}$ and $b_4^* = \frac{2cu}{t}\varphi^*<\frac{1}{2}.$
    \item Both $S_1,S_2$ are positive, and for $u/t<\frac{1}{4c}$ we have $$S_1\ell_1 < \frac{1}{4c}$$ and $S_2\ell_2 < \frac{5}{6}$, so both lie in $(0,1)$, and in particular are not integers.
    \item $S_4 = 2c(\varphi^*)^2-\frac{1}{2c}\in (-\frac{1}{2c},0)$ for $\varphi^*\in(0,\frac{1}{\sqrt{6}})$, so $S_4\not =0$, and $S_4\ell_4\in (-\frac{1}{2c},0)$ is again not an integer.
\end{itemize}
We conclude that each $\sin(\pi S_j\ell_j)\not=0$ and $S_j\not = 0$. Thus every factor of \eqref{eq:Lvalue} is strictly positive and $L>0$. The conditions on $\varphi^*,t,u$ that we imposed in the beginning are finite and open. In particular, Theorem~\ref{thm:transalign} produces uncountably many $\alpha$ meeting them. Choosing one of those, the proof is now complete.

\appendix
\section{Disproving Leibman's conjecture}\label{Leibmanfalse}
Using the same example we can also disprove the following conjecture of Leibman (see the remark after \cite[Proposition 3.1]{Leibman1}). In this section, since we no longer work with the Lie algebra, it will be convenient to use the notation $[\cdot,\cdot]$ for the \textbf{group commutator}. Namely, $[x,y]=xyx^{-1}y^{-1}$.
\begin{conjecture}[Leibman's conjecture]
  Let $W=N/\Lambda$ be a connected nilmanifold and let $Y = \pi(H)$ be a connected subnilmanifold
of $W$, where $H$ is a connected closed subgroup of $N$ and $\pi:N\rightarrow W$ the quotient map. Let $g:\Z\rightarrow N$ be a polynomial sequence with $g(0)=\mathrm{Id}_N$ such that $g(\Z)Y$ is dense in $W$ and assume that $N$ is generated by its connected component $N^o$ and the elements of $g$. Let $Z$ be the normal closure of $Y$ in $W$; then for any $f\in C(W)$ we have 
$$\lim_{n\rightarrow\infty} \int_{g(n)Y} f\,d\mu_{g(n)Y} - \int_{g(n)Z} f\,d\mu_{g(n)Z}=0.$$
\end{conjecture}
\begin{theorem}
    Leibman's conjecture is false.
\end{theorem}
Let $G$ be the free $3$-step nilpotent group on four generators and let $\Gamma$ be the lattice constructed in this paper, and write $X=G/\Gamma$. Let $Y=\{(x,x,x) : x\in X\}$ and $g(n) = (\mathrm{Id}_G,\alpha^n,\alpha^{2n})$ where $\alpha$ is the transcendental rotation constructed in Theorem~\ref{thm:transalign}. $Y$ is a subnilmanifold of $X^3$, but we can not take $W=X^3$ since $\{g(n)y:n\in \N,y\in Y\}$ is not dense in $X^3$. Thus, we shall first consider the Hall-Petresco group (see e.g., \cite[Section 5.1]{BHK} for a general construction) which in our case is the group
$$N = \{(h_0,h_0h_1,h_0h_1^2h_2) : h_0,h_1\in G,h_2\in G_2\}.$$ We also set $\Lambda := N\cap \Gamma^3.$
\begin{lemma}
We have,
    $$N=\{(g_0,g_1,g_2)\in G^3 : g_0\cdot g_1^{-2}\cdot g_2\in G_2\}.$$
    Moreover, $N$ is a closed, connected, rational subgroup of $G^3$, containing $H:=\{(g,g,g):g\in G\}$ as well as $a:=(1,\alpha,\alpha^2)$. Finally, $g(n)=a^n$ is a polynomial sequence in $N$ with $g(0)=\mathrm{Id}_N$.
\end{lemma}
\begin{proof}
  Let $h_0,h_1\in G$ and $h_2\in G_2$. Then,
  $$h_0(h_0h_1)^{-2}h_0h_1^2h_2 = [h_0,h_1^{-1}]h_2\in G_2.$$ Conversely, let $(g_0,g_1,g_2)\in G^3$ so that $g_0\cdot g_1^{-2}\cdot g_2\in G_2$. Let $h_0:=g_0$, $h_1:=g_0^{-1}\cdot g_1$ and $h_2 = g_1^{-1} g_0 g_1^{-1} g_2$ then $(g_0,g_1,g_2) = (h_0,h_0h_1,h_0h_1^2h_2)$ and 
  $$h_2 = [g_1^{-1},g_0]\cdot (g_0 g_1^{-2}g_2)\in G_2.$$ This proves the first equality. Now, since the map $(g_0,g_1,g_2)\mapsto g_0\cdot g_1^{-2}\cdot g_2$ is continuous and $G_2$ is closed, $N$ is closed. Moreover, since that same map is a homomorphism modulo $G_2$ we see that $N$ is a subgroup and it is connected since it is the continuous image of the connected group $G\times G\times  G_2$ under $(h_0,h_1,h_2)\mapsto (h_0,h_0h_1,h_0h_1^2h_2).$ The claims that $a\in N$ and that $g(n)=a^n$ is a polynomial sequence in $N$ with $g(0)=\mathrm{Id}_N$ are now immediate. It is left to show that $N$ is a rational subgroup of $G^3$, which follows from the direct computation that 
  $$\Lambda = \{(\gamma_0,\gamma_0\gamma_1,\gamma_0\gamma_1^2\gamma_2) : \gamma_0,\gamma_1\in \Gamma,\ \gamma_2\in \Gamma_2\}$$ and in particular $N/\Lambda\subseteq X^3$ is compact.
\end{proof}
Now, let $W=N/\Lambda$.
\begin{lemma}
   The set $\{g(n)y:n\in \mathbb{N},y\in Y\}$ is dense in $W$. 
\end{lemma}
\begin{proof}
We need to prove the equality
   $$\overline{\{g(n)y:n\in \mathbb{N},y\in Y\}} = W.$$
     From the previous lemma, we see that $g(n)=a^n$, $a\in N$ and $H\leq N$. Given all that, the inclusion $\subseteq $ is immediate. The reverse inclusion is less trivial, but it follows from a theorem of Leibman \cite[Theorem 6.3]{LeibmanOrbit} (and also from \cite[Lemma 5.2]{BHK}).
\end{proof}
\subsection{Computing the normal closure}
Next, we need to compute the normal closure $Z$. Equivalently, we need to compute the normal closure of $H$ in $N$.
\begin{proposition}
    The normal closure of $H$ in $N$ is 
    $$\langle\langle H\rangle\rangle^N = \left\{\left(h_1,h_1h_2,h_1h_2^2h_3\right) : \forall_{i=1,2,3}\ h_i\in G_i\right\}.$$
\end{proposition}
\begin{proof}
We begin the proof with the claim that:
$$[H,N] = \{(g_0,g_1,g_2) \in G_2^3 : g_0g_1^{-2}g_2\in G_3\}.$$
Indeed, for $g=(g_0,g_1,g_2)\in N$ and $h=(h_0,h_1,h_2)\in H$ (note that $h_0=h_1=h_2$) we have $[h,g] = ([h_0,g_0],[h_1,g_1],[h_2,g_2])$ is in $G_2^3$ and we have
$$[h_0,g_0][h_1,g_1]^{-2}[h_2,g_2] = [h_0,g_0g_1^{-2}g_2]\pmod {G_3}$$ since $g_0g_1^{-2}g_2\in G_2$, we deduce that $[h_0,g_0][h_1,g_1]^{-2}[h_2,g_2]\in G_3$, proving that $[H,N]\subseteq \{(g_0,g_1,g_2) \in G_2^3 : g_0g_1^{-2}g_2\in G_3\}.$ Next we prove the other inclusion. Observe first that since $G_2^3\subseteq N$, and since $[(h,h,h),(w,\mathrm{Id}_G,\mathrm{Id}_G)]=([h,w],\mathrm{Id}_G,\mathrm{Id}_G)$ whenever $h\in G$ and $w\in G_2$, and the same is true in the second and third coordinate, we conclude that $G_3^3\subseteq [H,N]$. Therefore, it suffices to prove the inclusion modulo $G_3^3$. In that case we have that the commutators $[(h,h,h),(\mathrm{Id}_G,n,n^2)] = (\mathrm{Id}_G,[h,n],[h,n]^2)\pmod{G_3^3}$ and $[(h,h,h),(n,n,n)] = ([h,n],[h,n],[h,n])$ together generate the entire group.

Now, since both $G_2$ and $G_3$ are normal, it then follows that $[H,N]$ is normal. In particular, $H\cdot[H,N]$ is a group. We prove that $$\langle\langle H\rangle\rangle^N = H\cdot [H,N].$$ The inclusion $H\cdot [H,N]\subseteq \langle\langle H\rangle\rangle$ is immediate. Conversely, since $H\cdot [H,N]$ contains $H$ it suffices to show that $H\cdot [H,N]$ is normal. Indeed,
$$g h[h',n]g^{-1} = ghg^{-1} (g[h',n]g^{-1}) = h\cdot ([h^{-1},g])\cdot(g[h',n]g^{-1}) \in H\cdot[H,N]$$ whenever $g\in N$, $h,h'\in H$ and $n\in N$. The rest of the proof is now a direct computation.
\end{proof}
\subsection{Vanishing $Z$-integral}
Let $f_0,f_1,f_2:G/\Gamma\rightarrow \mathbb{C}$ be as in the counterexample we constructed in the previous sections, and let $F=f_0\otimes f_1\otimes f_2:X^3\rightarrow \mathbb{C}$ denote the tensor product (i.e., $F(x_0,x_1,x_2) = f_0(x_0)\cdot f_1(x_1)\cdot f_2(x_2)$). Note that $F$ is not continuous (it is merely piecewise continuous), but as we will see shortly, this will not be a problem since it can be arbitrarily approximated by continuous functions in $L^1$, and all $f_0,f_1,f_2$ are bounded. 
\begin{proposition}
    For every $n\in \Z$, we have
    $$\int_{g(n)Z} F~d\mu_{g(n)Z}=0.$$
\end{proposition}
\begin{proof}
    We proved that $G_3^3\subseteq \langle\langle H\rangle\rangle^N$. Thus, the nilmanifold $Z$ contains all of the last layer of $X^3$. Since $\eta\not = 0$ integrating over any of the coordinates on the last layer (freezing all other coordinates) gives the desired result.
\end{proof}
\subsection{Concluding the argument}
In order to refute Leibman's conjecture, we must replace $F$ with a continuous function. 
\begin{lemma}[Same marginals]
    For every $n\in \Z$, and every $i\in \{1,2,3\}$, the push-forward of $\mu_{g(n)Y}$ and $\mu_{g(n)Z}$ to the $i^{\mathrm{th}}$ coordinate coincide.
\end{lemma}
\begin{proof}
The marginals of $\mu_{g(n)Y}$ are clearly $\mu_X$, since multiplication by $\alpha$ is measure-preserving. For $\mu_{g(n)Z}$, recall that the Homogeneous group associated with $Z$ is the normal closure computed above $$\langle\langle H\rangle\rangle^N = \{(h_1,h_1h_2,h_1h_2^2h_3) : h_1\in G, h_2\in G_2, h_3\in G_3\}.$$ Let $\pi_i : G^3 \to G$ denote the projection onto the $i^{\mathrm{th}}$ coordinate. Since $H\subseteq \langle\langle H\rangle\rangle^N$, it is immediate that $\pi_i(\langle\langle H\rangle\rangle^N) = G$ for all $i \in \{1,2,3\}$. Consequently, the natural projection of the subnilmanifold $Z$ onto any of the three coordinates is the entire space $X=G/\Gamma$. Since $\mu_Z$ is the unique normalized Haar measure on the homogeneous space $Z$, its push-forward under this coordinate projection must be the unique normalized Haar measure on the target space, which is exactly $\mu_X$. Finally, the measure $\mu_{g(n)Z}$ is simply the push-forward of $\mu_Z$ under left-translation by $g(n) = (\mathrm{Id}_G, \alpha^n, \alpha^{2n})$. Since $\mu_X$ is preserved under translations by $\alpha$, we conclude the proof.
\end{proof}

Now, choose $\varepsilon>0$ sufficiently small. By Lusin's theorem, we can choose $f'_0,f'_1,f'_2\in C(X)$ so that $\|f_i-f'_i\|_{L^1(\mu_X)}<\varepsilon$. Let $F'=f'_0\otimes f'_1\otimes f'_2$. Since $\|\prod_{i=1}^3 g_i \|_{L^1}\leq \|g_j\|_{L^1}\cdot \prod_{i\not=j}\|g_i\|_\infty$ for all $g_1,g_2,g_3\in L^1(\mu_X)$ we deduce that 
$$\left|\int_{g(n)Y} (F-F')\,d\mu_{g(n)Y}\right|\leq 3\varepsilon,\qquad  \left|\int_{g(n)Z} (F-F')\,d\mu_{g(n)Z}\right|\leq 3\varepsilon.$$

Letting $n_k$ denote the subsequence from Theorem~\ref{thm:transalign} and $L$ the limit of the correlation sequence associated with $f_0,f_1,f_2$ along this subsequence in absolute value i.e., $L=\lim_{k\rightarrow\infty}|a(n_k)|$. Note that $a(n)=\left|\int_{g(n)Y} F\,d\mu_{g(n)Y}\right|$. We have:

$$\left|\int_{g(n_k)Y} F'\,d\mu_{g(n_k)Y} - \int_{g(n_k)Z} F'\,d\mu_{g(n_k)Z}\right|\geq \left|\int_{g(n_k)Y} F\,d\mu_{g(n_k)Y} - \int_{g(n_k)Z} F\,d\mu_{g(n_k)Z}\right|-6\varepsilon \rightarrow L-6\varepsilon. $$
Thus, taking $0<\varepsilon<\frac{L}{6}$ gives a contradiction. \qed
\section{Symbolic computations}\label{app:cert}
\subsection{The embedding to the tensor algebra}
Let $V=\bigoplus_{i=1}^{r}\R e_i$ and let
$T(V)=\bigoplus_{n\ge 0}V^{\otimes n}$ be its tensor algebra, with product the
concatenation $a\otimes b$ and unit $1\in V^{\otimes 0}$. We write $T(V)_{\mathrm{Lie}}$
for the vector space $T(V)$ regarded as a Lie algebra under the commutator
$[a,b]:=ab-ba$, and we identify $V$ with $V^{\otimes 1}\subseteq T(V)$.

Let $\mathrm{FL}(V)\subseteq T(V)_{\mathrm{Lie}}$ be the Lie subalgebra generated by $V$ (i.e., the smallest Lie algebra containing $V$ and closed under the bracket $[\cdot,\cdot]$).
By the Poincar\'e--Birkhoff--Witt theorem and Friedrichs' criterion
\cite{Reutenauer}, $\mathrm{FL}(V)$ is the \emph{free} Lie algebra on $V$, the inclusion
$\mathrm{FL}(V)\hookrightarrow T(V)$ is injective, and it is homogeneous for the
word-length grading:
\[
\mathrm{FL}(V)=\bigoplus_{n\ge 1}\mathrm{FL}_n,\qquad
\mathrm{FL}_n:=\mathrm{FL}(V)\cap V^{\otimes n},
\]
with $\dim\mathrm{FL}_n=\tfrac1n\sum_{d\mid n}\mu(d)r^{n/d}$ (Witt's formula). Let $\gamma_k(\mathrm{FL}(V))$ denote the $k^{\mathrm{th}}$ term in the lower central series.

\begin{lemma}\label{lem:grading-lcs}
The word-length grading of $\mathrm{FL}(V)$ coincides with its lower central series:
$\gamma_k\bigl(\mathrm{FL}(V)\bigr)=\bigoplus_{n\ge k}\mathrm{FL}_n$ for all $k\ge 1$.
\end{lemma}

\begin{proof}
Brackets are homogeneous, $[\mathrm{FL}_p,\mathrm{FL}_q]\subseteq\mathrm{FL}_{p+q}$, so
$\gamma_{k+1}(\mathrm{FL}(V))=[\mathrm{FL}(V),\gamma_k(\mathrm{FL}(V))]\subseteq\bigoplus_{n\ge k+1}\mathrm{FL}_n$ by
induction. For the reverse inclusion it suffices to show $\mathrm{FL}_n\subseteq\gamma_n(\mathrm{FL}(V))$,
since the $\gamma_k(\mathrm{FL}(V))$ are nested. As $\mathrm{FL}(V)$ is generated in degree~$1$, every
degree-$n$ element is a sum of left-normed brackets, hence
$\mathrm{FL}_n=[\mathrm{FL}_1,\mathrm{FL}_{n-1}]$ \cite{Reutenauer}. By induction
$\mathrm{FL}_{n-1}\subseteq\gamma_{n-1}(\mathrm{FL}(V))$, so
$\mathrm{FL}_n=[\mathrm{FL}_1,\mathrm{FL}_{n-1}]\subseteq[\mathrm{FL}(V),\gamma_{n-1}(\mathrm{FL}(V))]
=\gamma_n(\mathrm{FL}(V))$.
\end{proof}

The \emph{free $c$-step nilpotent Lie algebra} on $V$ is
$\f=\f_{r,c}:=\mathrm{FL}(V)/\gamma_{c+1}(\mathrm{FL}(V))$. On the associative side, set
\[
I:=\bigoplus_{n> c}V^{\otimes n}\ \trianglelefteq\ T(V),\qquad
T_{\le c}(V):=T(V)/I,
\]
a two-sided ideal and the associated quotient algebra. Let $\pi\colon T(V)\to T_{\le c}(V)$ be the projection
and $T_{\le c}(V)_{\mathrm{Lie}}$ the quotient with its commutator bracket.

\begin{definition}[The embedding]\label{def:embedding}
Let $\Phi\colon\f\to T_{\le c}(V)_{\mathrm{Lie}}$ be the map induced by
$\pi|_{\mathrm{FL}(V)}\colon \mathrm{FL}(V)\to T_{\le c}(V)_{\mathrm{Lie}}$, i.e., the
unique Lie homomorphism with $\Phi(e_i)=e_i$ that sends every iterated bracket to the
corresponding iterated commutator. Explicitly,
\[
\Phi(e_i)=e_i,\qquad
\Phi([e_i,e_j])=e_ie_j-e_je_i,\qquad
\Phi([e_k,[e_i,e_j]])=\bigl[e_k,\,e_ie_j-e_je_i\bigr],
\]
all computed in $T_{\le c}(V)$.
\end{definition}

\begin{proposition}
$\Phi$ is a well-defined injective homomorphism of Lie algebras, with image the
degree-$\le c$ Lie part $\bigoplus_{n=1}^{c}\pi(\mathrm{FL}_n)$ of $T_{\le c}(V)$. Thus
$\f$ is realized as the Lie subalgebra of $T_{\le c}(V)_{\mathrm{Lie}}$ generated by
$e_1,\dots,e_r$.
\end{proposition}

\begin{proof}
The projection $\pi$ is an algebra homomorphism, so its restriction
$\pi|_{\mathrm{FL}(V)}\colon\mathrm{FL}(V)\to T_{\le c}(V)_{\mathrm{Lie}}$ is a Lie
homomorphism. Its kernel is
$\mathrm{FL}(V)\cap I=\bigoplus_{n>c}\mathrm{FL}_n=\gamma_{c+1}(\mathrm{FL}(V))$ by
Lemma~\ref{lem:grading-lcs}. The first isomorphism theorem yields a well-defined
injective Lie homomorphism
$\Phi\colon\mathrm{FL}(V)/\gamma_{c+1}(\mathrm{FL}(V))=\f\hookrightarrow T_{\le c}(V)_{\mathrm{Lie}}$,
with image $\pi(\mathrm{FL}(V))=\bigoplus_{n=1}^{c}\pi(\mathrm{FL}_n)$.
\end{proof}
\subsection{The code in simple words}
Here we include the code used in the paper. Specifically the computation of  $\Lambda^{(3)}$, the law $s^{(3)}(\alpha^n)$, the exact coefficients for $P,Q$ from the product law (in second-kind coordinates), and finally the computation of $P(m)$.
The code is written in Python and is explained here in simple words. The first lines of the code are merely a technicality:
\begin{align*}&\textbf{import} \textit{ sympy} \textbf{ as} \text{ sp}\\
&\textbf{from} \textit{ itertools} \textbf{ import} \textit{ product}\\
&\textbf{from} \textit{ collections} \textbf{ import } \textit{defaultdict}\\
\text{ }\\
&\textit{Letters} = (1,2,3,4).
\end{align*}
The first three lines import the tools needed for the computation. The last line defines our letters (namely, the four generators $e_1,e_2,e_3,e_4$ used in the paper).
Our first function $t_{\mathrm{clean}}$ is used to simplify terms, in particular by removing any variable multiplied by a zero coefficient. It works like this:
\begin{align*}
    &\textbf{def } t_{\text{clean}}(d):\\ &\quad \textit{out}=\{\text{ }\}\\
    &\quad \textbf{for } \text{w, c} \text{ in} \textit{ d.items():}\\&\qquad \textit{c = sp.expand(c)}\\&\qquad \textbf{if } \text{c!=0:}\\&\qquad \quad \textit{out[w] = c}\\&\quad \textbf{return} \textit{ out.}
\end{align*}
We defined a new function $t_{\text{clean}}$ and it takes a term $d$ and performs this: First it creates an empty word $\text{out}=\{\}$. Then for a word $w$ and a coefficient $c$ of $w$ in $d$ it simplifies $c$. If $c$ is not zero, then it returns the coefficient $c$, but if it is zero, then it returns the empty word that we have set earlier.\\
Addition is commutative and is computed by the code:
\begin{align*}
    &\textbf{def } t_{\text{add}}(a, b):\\
    &\quad \textit{r = dict(a)}\\
    &\quad \textbf{for } \text{w, c} \textbf{ in } \textit{b.items():}\\
    &\qquad \textit{r[w] = r.get(w, 0) + c}\\
    &\quad \textbf{return } t_{\text{clean}}(r).
\end{align*}
The function $t_{\text{add}}$ takes two terms $a$ and $b$. It first copies $a$ into a working term $r$. Then it runs over every word $w$ of $b$, with coefficient $c$, and adds $c$ to whatever coefficient $w$ already had in $r$ (if $w$ did not appear in $a$, its old coefficient is taken to be $0$). Finally it cleans the result.

\medskip
Scaling by a number and subtraction are equally simple:
\begin{align*}
    &\textbf{def } t_{\text{scale}}(a, s):\\
    &\quad \textbf{return } t_{\text{clean}}\big(\{\, \text{w}:\ s\cdot c \ \textbf{for } \text{w, c} \textbf{ in } \textit{a.items()}\,\}\big).\\[4pt]
    &\textbf{def } t_{\text{sub}}(a, b):\\
    &\quad \textbf{return } t_{\text{add}}\big(a,\ t_{\text{scale}}(b, -1)\big).
\end{align*}
Here $t_{\text{scale}}$ multiplies the coefficient of every word of $a$ by the number $s$, and $t_{\text{sub}}$ subtracts $b$ from $a$ by adding $(-1)\cdot b$.

\medskip
The next function is the heart of the model: multiplication.
\begin{align*}
    &\textbf{def } t_{\text{mul}}(a, b):\\
    &\quad \textit{r = defaultdict(lambda: 0)}\\
    &\quad \textbf{for } \text{wa, ca} \textbf{ in } \textit{a.items():}\\
    &\qquad \textbf{for } \text{wb, cb} \textbf{ in } \textit{b.items():}\\
    &\qquad\quad \textit{w = wa + wb}\\
    &\qquad\quad \textbf{if } \textit{len}(w) \leq 3:\\
    &\qquad\qquad \textit{r[w] += ca}\cdot \textit{cb}\\
    &\quad \textbf{return } t_{\text{clean}}\big(\textit{dict}(r)\big).
\end{align*}
We multiply two terms $a$ and $b$ in the tensor algebra (not in $\f$). For every word $w_a$ ($\text{wa}$) of $a$ with coefficient $c_a$ ($\text{ca}$), and every word $w_b$ of $b$ with coefficient $c_b$, their product is the \emph{concatenation} $w_a+w_b$ of the two words, and its coefficient is the product $c_a\cdot c_b $. These contributions are accumulated in $r$. Fortunately, these computations are simplified by the line ``$\textbf{if } \textit{len}(w)\leq 3$'' which drops any word of length $\geq 4$. 

\medskip
The Lie bracket is now just the commutator $[a,b]=ab-ba$:
\begin{align*}
    &\textbf{def } t_{\text{bracket}}(a, b):\\
    &\quad \textbf{return } t_{\text{sub}}\big(t_{\text{mul}}(a, b),\ t_{\text{mul}}(b, a)\big).
\end{align*}

\medskip
We record the unit $1$ of the algebra (the empty word, with coefficient $1$),
$$\textit{ONE} = \{\, ():\ \textit{sp.Integer}(1)\,\},$$
and then the group operations $\exp,\log$ and the inverse. Since the algebra is $3$-step nilpotent, the exponential and logarithm series \emph{terminate} and are therefore finite and exact:
\begin{align*}
    &\textbf{def } t_{\text{exp}}(z):\\
    &\quad z_2 = t_{\text{mul}}(z, z)\\
    &\quad z_3 = t_{\text{mul}}(z_2, z)\\
    &\quad \textbf{return } t_{\text{add}}\big(t_{\text{add}}(\textit{ONE}, z),\ t_{\text{add}}(t_{\text{scale}}(z_2, \tfrac12),\ t_{\text{scale}}(z_3, \tfrac16))\big).\\[4pt]
    &\textbf{def } t_{\text{log}}(g):\\
    &\quad w = t_{\text{sub}}(g, \textit{ONE})\\
    &\quad w_2 = t_{\text{mul}}(w, w)\\
    &\quad w_3 = t_{\text{mul}}(w_2, w)\\
    &\quad \textbf{return } t_{\text{add}}\big(t_{\text{sub}}(w,\ t_{\text{scale}}(w_2, \tfrac12)),\ t_{\text{scale}}(w_3, \tfrac13)\big).\\[4pt]
    &\textbf{def } t_{\text{inv}}(g):\\
    &\quad \textbf{return } t_{\text{exp}}\big(t_{\text{scale}}(t_{\text{log}}(g), -1)\big).
\end{align*}
Here $t_{\text{exp}}$ computes $\exp(z)=1+z+\tfrac12 z^2+\tfrac16 z^3$ for an element $z$ without constant term, and $t_{\text{log}}$ computes $\log(g)=w-\tfrac12 w^2+\tfrac13 w^3$ where $w=g-1$. (The fractions $\tfrac12,\tfrac16,\tfrac13$ are kept as exact rationals.) The inverse of a group element is $t_{\text{inv}}(g)=\exp(-\log g)$.

\medskip
\noindent\textbf{The generators and the Hall basis.}
We now build the standard basis of $\f$ inside the model. The generators are the single-letter words,
$$E = \{\, i:\ \{(i,): \textit{sp.Integer}(1)\}\ \textbf{for } i \textbf{ in } \textit{Letters}\,\},$$
the second layer consists of the brackets $e_{ij}=[e_i,e_j]$ for $i<j$,
\begin{align*}
&\textit{Pairs} = [\,(i,j)\ \textbf{for } i \textbf{ in } \textit{Letters}\ \textbf{for } j \textbf{ in } \textit{Letters}\ \textbf{if } i<j\,],\\
&e_2 = \{\,(i,j):\ t_{\text{bracket}}(E[i], E[j])\ \textbf{for } (i,j) \textbf{ in } \textit{Pairs}\,\},
\end{align*}
and the third layer is the Hall basis $e_{k,ij}=[e_k,e_{ij}]$ with $i<j$ and $k\geq i$,
\begin{align*}
&\textit{Hall}_3 = [\,(k,i,j)\ \textbf{for } (i,j) \textbf{ in } \textit{Pairs}\ \textbf{for } k \textbf{ in } \textit{Letters}\ \textbf{if } k\geq i\,],\\
&e_3 = \{\,(k,i,j):\ t_{\text{bracket}}(E[k], e_2[(i,j)])\ \textbf{for } (k,i,j) \textbf{ in } \textit{Hall}_3\,\}.
\end{align*}
Together these are $4+6+20=30$ elements. The code assembles them into one ordered list $\textit{Basis}$ and forms an $84\times 30$ matrix $M$ whose columns are these $30$ vectors written out in the basis of all words of length $1,2,3$ (there are $4+16+64=84$ such words). One then checks
$$\operatorname{rank}(M)=30.$$
This single line certifies two things at once: the $30$ Hall elements are linearly independent, and hence $\dim\f=30$ (Witt's count). Because $M$ has full column rank, every Lie element $v$ has unique coordinates in the Hall basis, given by the exact rational formula $c=(M^{\top}M)^{-1}M^{\top}\widehat v$, where $\widehat v$ is $v$ written in the $84$ words. This is what the function $\textit{coords}$ returns:
\begin{align*}
    &\textbf{def } \textit{coords}(v):\\
    &\quad \widehat v = \text{the }84\text{-vector of coefficients of } v\\
    &\quad \textbf{return } (M^{\top}M)^{-1}\, M^{\top}\, \widehat v \quad\text{(as a dictionary $\{\text{tag}:\text{coefficient}\}$)}.
\end{align*}

\medskip
\noindent\textbf{Computation of $\Lambda^{(3)}$.}
The rotation is $\alpha=\exp(\alpha_1 e_1)\exp(\alpha_2 e_2)\exp(\alpha_3 e_3)\exp(\alpha_4 e_4)$, and $\Lambda=\log\alpha$. With symbolic $\alpha_1,\dots,\alpha_4$ this is exactly:
\begin{align*}
    &\textit{acc} = \textit{ONE}\\
    &\textbf{for } i \textbf{ in } \textit{Letters}:\quad \textit{acc} = t_{\text{mul}}\big(\textit{acc},\ t_{\text{exp}}(t_{\text{scale}}(E[i], \alpha_{i}))\big)\\
    &\Lambda = t_{\text{log}}(\textit{acc})\\
    &cL = \textit{coords}(\Lambda).
\end{align*}
The dictionary $cL$ holds the Hall coordinates of $\Lambda$. Its first layer is $\sum_i\alpha_i e_i$, its second layer is $\tfrac12\sum_{i<j}\alpha_i\alpha_j e_{ij}$, and its third layer is compared, coordinate by coordinate, with the $20$ coefficients of $\Lambda^{(3)}$ stated in the paper; all $20$ agree.

\medskip
\noindent\textbf{The maps $\tau$ and $\mathrm{peel}$.}
To pass to second (Mal'cev) coordinates we need the basis vectors $Y_{ij},Y_{k,ij}$, which the paper \emph{defines} as group commutators $(g,h)=g\ast h\ast g^{-1}\ast h^{-1}$. The commutator and the two families of vectors are
\begin{align*}
    &\textbf{def } \textit{comm}(g, h):\quad \textbf{return } t_{\text{mul}}\big(t_{\text{mul}}(t_{\text{mul}}(g, h), t_{\text{inv}}(g)), t_{\text{inv}}(h)\big),\\[4pt]
    &\textit{expY}_{ij} = \textit{comm}\big(t_{\text{exp}}(E[i]),\ t_{\text{exp}}(E[j])\big),\\
    &Y_{ij} = t_{\text{log}}\big(\textit{expY}_{ij}\big),\\
    &Y_{k,ij} = t_{\text{log}}\big(\textit{comm}(t_{\text{exp}}(E[k]),\ \textit{expY}_{ij})\big).
\end{align*}
Computing them this way also certifies the paper's formulas $Y_{ij}=e_{ij}+\tfrac12(e_{i,ij}+e_{j,ij})$ and $Y_{k,ij}=e_{k,ij}$. The map $\tau$ is the ordered product of exponentials, $$\tau(s)=\prod_{i=1}^{4}\exp(s_i Y_i)\ \prod_{1\leq i<j\leq 4}\exp(s_{ij}Y_{ij})\ \prod_{k\geq i, 1\leq i <j\leq 4}\exp(s_{k,ij}Y_{k,ij}),$$ which in code reads
\begin{align*}
    &\textbf{def } \tau(s):\\
    &\quad g = \textit{ONE}\\
    &\quad \textbf{for } i \textbf{ in } \textit{Letters}:\quad g = t_{\text{mul}}\big(g,\ t_{\text{exp}}(t_{\text{scale}}(E[i],\ s_i))\big)\\
    &\quad \textbf{for } (i,j) \textbf{ in } \textit{Pairs}:\quad g = t_{\text{mul}}\big(g,\ t_{\text{exp}}(t_{\text{scale}}(Y_{ij},\ s_{ij}))\big)\\
    &\quad \textbf{for } (k,i,j) \textbf{ in } \textit{Hall}_3:\quad g = t_{\text{mul}}\big(g,\ t_{\text{exp}}(t_{\text{scale}}(Y_{k,ij},\ s_{k,ij}))\big)\\
    &\quad \textbf{return } g.
\end{align*}
Here $s_i,s_{ij},s_{k,ij}$ denote the first, second and third layer coordinates stored in the dictionary $s$ (a missing coordinate is taken to be $0$). 

Its inverse $\mathrm{peel}=\tau^{-1}$ recovers the second-kind coordinates one layer at a time: read off the first layer from $\log g$, divide out $\tau_1$, read off the second layer, divide out $\tau_2$, then read off the third layer:
\begin{align*}
    &\textbf{def } \mathrm{peel}(g):\\
    &\quad s_i = \textit{coords}\big(t_{\text{log}}(g)\big)_i \qquad (i\in\textit{Letters})\\
    &\quad g_1 = \textstyle\prod_{i} t_{\text{exp}}(t_{\text{scale}}(E[i],\ s_i)) \qquad (=\tau_1(s^{(1)}))\\
    &\quad g' = t_{\text{mul}}\big(t_{\text{inv}}(g_1),\ g\big)\\
    &\quad s_{ij} = \textit{coords}\big(t_{\text{log}}(g')\big)_{ij} \qquad (i<j)\\
    &\quad g_2 = \textstyle\prod_{i<j} t_{\text{exp}}(t_{\text{scale}}(Y_{ij},\ s_{ij})) \qquad (=\tau_2(s^{(2)}))\\
    &\quad g'' = t_{\text{mul}}\big(t_{\text{inv}}(g_2),\ g'\big)\\
    &\quad s_{k,ij} = \textit{coords}\big(t_{\text{log}}(g'')\big)_{k,ij}\\
    &\quad \textbf{return } s.
\end{align*}
At each step only the freshly exposed layer is read off, because $\log g_1$ has only a first layer (so $\log g'$ has no first layer), $Y_{ij}\equiv e_{ij}$ modulo the third layer (so $\log g''$ has no second layer), and $Y_{k,ij}=e_{k,ij}$ (so the third layer is read directly).

\medskip
\noindent\textbf{The law $s^{(3)}(\alpha^n)$.}
Since $\alpha^n=\exp(n\Lambda)$, the second-kind coordinates of the orbit are obtained in one line,
$$s(\alpha^n) = \mathrm{peel}\big(t_{\text{exp}}(t_{\text{scale}}(\Lambda, n))\big),$$
with symbolic $n$. The first two layers reproduce $n\alpha$ and $-\binom{n}{2}\sum_{i<j}\alpha_i\alpha_j e_{ij}$, and the third layer is checked against the cubic $(n-n^3)\Lambda^{(3)}-\tfrac{n-n^2}{4}S_1-\tfrac{n^2-n^3}{4}S_2$ of the paper.

\medskip
\noindent\textbf{The product law and the coefficients of $P,Q$.}
The multiplication in second-kind coordinates is $\mu(x,x')=\tau^{-1}\big(\tau(x)\ast\tau(x')\big)$, computed with two sets of $30$ independent symbolic coordinates,
$$\textit{MU} = \mathrm{peel}\big(t_{\text{mul}}(\tau(x), \tau(x'))\big).$$
From $\textit{MU}$ the code reads off the bilinear correction $P=\mu^{(2)}-x^{(2)}-x'^{(2)}$ and the cubic correction $Q=\mu^{(3)}-x^{(3)}-x'^{(3)}$. 

\medskip
\noindent\textbf{Computation of $P(m)$.}
The chosen top-layer character $\eta$ has four non-zero coordinates, so the pairing $\langle\eta,\cdot\rangle$ acting on a third-layer vector $v$ is implemented directly as
$$\langle\eta, v\rangle = 2\,v_{4,12} - 9\,v_{1,13} + v_{2,14} + 6\,v_{2,23},$$
where $v_{k,ij}$ is the $e_{k,ij}$-coordinate of $v$. With $g=\exp(m\Lambda)\ast\tau(x^{(1)},0,0)$, $s(g)=\mathrm{peel}(g)$, and $\phi=\fl{m\alpha+x^{(1)}}$, the single-point phase is assembled as
\begin{align*}
    &g = t_{\text{mul}}\big(t_{\text{exp}}(t_{\text{scale}}(\Lambda, m)),\ \tau(x^{(1)},0,0)\big)\\
    &s(g) = \mathrm{peel}(g)\\
    &A_0 = \langle\eta,\ s^{(3)}(g)\rangle\\
    &\Pi = \langle\eta,\ Q\big(s^{(1)}(g),\ s^{(2)}(g),\ -\phi\big)\rangle\\
    &P(m) = \textit{sp.expand}\big(A_0 + \Pi\big).
\end{align*}
Here $A_0$ is the floor-free part $\langle\eta,s^{(3)}(g)\rangle$, and $\Pi$ evaluates the cubic $Q$ (built above) at the orbit coordinates $s^{(1)}(g),s^{(2)}(g)$ and at $-\phi$. 

\bibliographystyle{abbrv}
\bibliography{bibliography}

@article {LeibmanOrbit,
    AUTHOR = {Leibman, A.},
     TITLE = {Orbit of the diagonal in the power of a nilmanifold},
   JOURNAL = {Trans. Amer. Math. Soc.},
  FJOURNAL = {Transactions of the American Mathematical Society},
    VOLUME = {362},
      YEAR = {2010},
    NUMBER = {3},
     PAGES = {1619--1658},
      ISSN = {0002-9947,1088-6850},
   MRCLASS = {37A17 (22F30 37A05 37A30)},
  MRNUMBER = {2563743},
MRREVIEWER = {Nikos\ Frantzikinakis},
       DOI = {10.1090/S0002-9947-09-04961-7},
       URL = {https://doi.org/10.1090/S0002-9947-09-04961-7},
}

@book{AuslanderGreenHahn,
  AUTHOR    = {Auslander, Louis and Green, Leon and Hahn, Frank},
  TITLE     = {Flows on homogeneous spaces},
  SERIES    = {Annals of Mathematics Studies, No. 53},
  PUBLISHER = {Princeton University Press, Princeton, N.J.},
  YEAR      = {1963},
  PAGES     = {vii+107},
  MRNUMBER  = {0167569},
}

@book {Schmidtsubspace,
    AUTHOR = {Schmidt, Wolfgang M.},
     TITLE = {Diophantine approximation},
    SERIES = {Lecture Notes in Mathematics},
    VOLUME = {785},
 PUBLISHER = {Springer, Berlin},
      YEAR = {1980},
     PAGES = {x+299},
      ISBN = {3-540-09762-7},
   MRCLASS = {10Fxx (10-02)},
  MRNUMBER = {568710},
MRREVIEWER = {A.\ J.\ van der Poorten},
}

@article {BLgeneralized,
    AUTHOR = {Bergelson, Vitaly and Leibman, Alexander},
     TITLE = {Distribution of values of bounded generalized polynomials},
   JOURNAL = {Acta Math.},
  FJOURNAL = {Acta Mathematica},
    VOLUME = {198},
      YEAR = {2007},
    NUMBER = {2},
     PAGES = {155--230},
      ISSN = {0001-5962,1871-2509},
   MRCLASS = {11K31 (11J54)},
  MRNUMBER = {2318563},
MRREVIEWER = {Alexander\ Gorodnik},
       DOI = {10.1007/s11511-007-0015-y},
       URL = {https://doi.org/10.1007/s11511-007-0015-y},
}

@misc{Leng2025,
  author       = {Leng, James},
  title        = {Structured extensions and multi-correlation sequences},
  year         = {2025},
  eprint       = {2504.07038},
  archivePrefix= {arXiv},
  primaryClass = {math.DS},
  note         = {Preprint},
  url          = {https://arxiv.org/abs/2504.07038}
}

@article{bm,
AUTHOR = {Bergelson, V., Moragues, A.F.},
     TITLE = {An ergodic correspondence principle, invariant means and applications.},
   JOURNAL = {Israel Journal of Mathematics},
    VOLUME = {245},
      YEAR = {2021},
     PAGES = {921--962},
DOI={https://doi.org/10.1007/s11856-021-2233-y},
}

@article {Franpoly,
    AUTHOR = {Frantzikinakis, Nikos},
     TITLE = {Multiple ergodic averages for three polynomials and
              applications},
   JOURNAL = {Trans. Amer. Math. Soc.},
  FJOURNAL = {Transactions of the American Mathematical Society},
    VOLUME = {360},
      YEAR = {2008},
    NUMBER = {10},
     PAGES = {5435--5475},
      ISSN = {0002-9947,1088-6850},
   MRCLASS = {37A05 (11B75 28D05 37A45)},
       DOI = {10.1090/S0002-9947-08-04591-1},
       URL = {https://doi.org/10.1090/S0002-9947-08-04591-1},
}

@misc{ShalomGrothendieck,
AUTHOR = {Or Shalom},
TITLE = {An application of Grothendieck theorem to the theory of multicorrelation sequences, multiple recurrence and partition regularity of quadratic equations},
Journal = {Preprint, available at https://arxiv.org/abs/2302.12857},
Year={2023},

}

@article {Stepin,
    AUTHOR = {Stepin, A. M.},
     TITLE = {Flows on solvable manifolds},
   JOURNAL = {Uspehi Mat. Nauk},
  FJOURNAL = {Akademiya Nauk SSSR i Moskovskoe Matematicheskoe Obshchestvo.
              Uspekhi Matematicheskikh Nauk},
    VOLUME = {24},
      YEAR = {1969},
    NUMBER = {5 (149)},
     PAGES = {241--242},
      ISSN = {0042-1316},
   MRCLASS = {28.70 (22.00)},
  MRNUMBER = {0267075},
MRREVIEWER = {J. Merza},
}

@article {ParrySpectral,
    AUTHOR = {Parry, W.},
     TITLE = {Spectral analysis of {$G$}-extensions of dynamical systems},
   JOURNAL = {Topology},
  FJOURNAL = {Topology. An International Journal of Mathematics},
    VOLUME = {9},
      YEAR = {1970},
     PAGES = {217--224},
      ISSN = {0040-9383},
   MRCLASS = {54.82 (28.00)},
  MRNUMBER = {261581},
MRREVIEWER = {Robert Ellis},
       DOI = {10.1016/0040-9383(70)90011-X},
       URL = {https://doi.org/10.1016/0040-9383(70)90011-X},
}

@article {HKM,
    AUTHOR = {Host, Bernard and Kra, Bryna and Maass, Alejandro},
     TITLE = {Complexity of nilsystems and systems lacking nilfactors},
   JOURNAL = {J. Anal. Math.},
  FJOURNAL = {Journal d'Analyse Math\'{e}matique},
    VOLUME = {124},
      YEAR = {2014},
     PAGES = {261--295},
      ISSN = {0021-7670},
   MRCLASS = {37A30 (37B05)},
  MRNUMBER = {3286054},
MRREVIEWER = {Siming Tu},
       DOI = {10.1007/s11854-014-0032-7},
       URL = {https://doi.org/10.1007/s11854-014-0032-7},
}

@book{Reutenauer,
  Author    = {Reutenauer, Christophe},
  Title     = {Free {L}ie algebras},
  Series    = {London Mathematical Society Monographs. New Series},
  Volume    = {7},
  Publisher = {The Clarendon Press, Oxford University Press, New York},
  Year      = {1993},
  Pages     = {xviii+269},
  ISBN      = {0-19-853679-8},
  MRNUMBER  = {1231799},
}

@article{ARS,
  Author        = {Ackelsberg, Ethan and Richter, Florian K. and Shalom, Or},
  Title         = {On the maximal spectral type of nilsystems},
  Journal       = {Proc. Amer. Math. Soc. Ser. B},
  Fjournal      = {Proceedings of the American Mathematical Society, Series B},
  Volume        = {11},
  Year          = {2024},
  Pages         = {469--480},
}

@article {ABS,
    AUTHOR = {Ackelsberg, Ethan and Bergelson, Vitaly and Shalom, Or},
     TITLE = {Khintchine-type recurrence for 3-point configurations},
   JOURNAL = {Forum Math. Sigma},
  FJOURNAL = {Forum of Mathematics. Sigma},
    VOLUME = {10},
      YEAR = {2022},
     PAGES = {Paper No. e107, 57},
      ISSN = {2050-5094},
   MRCLASS = {37A15 (05D10 37A30)},
  MRNUMBER = {4519061},
       DOI = {10.1017/fms.2022.97},
       URL = {https://doi.org/10.1017/fms.2022.97},
}

@article {morejoint,
    AUTHOR = {Donoso, Sebasti\'an and Ferr\'e{} Moragues, Andreu and
              Koutsogiannis, Andreas and Sun, Wenbo},
     TITLE = {Decomposition of multicorrelation sequences and joint
              ergodicity},
   JOURNAL = {Ergodic Theory Dynam. Systems},
  FJOURNAL = {Ergodic Theory and Dynamical Systems},
    VOLUME = {44},
      YEAR = {2024},
    NUMBER = {2},
     PAGES = {432--480},
      ISSN = {0143-3857,1469-4417},
   MRCLASS = {37A05 (28A99 37A30 60F99)},
  MRNUMBER = {4686270},
MRREVIEWER = {Anh\ N.\ Le},
       DOI = {10.1017/etds.2023.30},
       URL = {https://doi.org/10.1017/etds.2023.30},
}

@article {integerpart,
    AUTHOR = {Koutsogiannis, Andreas},
     TITLE = {Integer part polynomial correlation sequences},
   JOURNAL = {Ergodic Theory Dynam. Systems},
  FJOURNAL = {Ergodic Theory and Dynamical Systems},
    VOLUME = {38},
      YEAR = {2018},
    NUMBER = {4},
     PAGES = {1525--1542},
      ISSN = {0143-3857,1469-4417},
   MRCLASS = {37A30 (11B30 37A05 37A45 40A05 81P40 94A55)},
  MRNUMBER = {3789175},
MRREVIEWER = {Bryna\ Kra},
       DOI = {10.1017/etds.2016.67},
       URL = {https://doi.org/10.1017/etds.2016.67},
}

@article {MCS,
    AUTHOR = {Frantzikinakis, Nikos},
     TITLE = {Multiple correlation sequences and nilsequences},
   JOURNAL = {Invent. Math.},
  FJOURNAL = {Inventiones Mathematicae},
    VOLUME = {202},
      YEAR = {2015},
    NUMBER = {2},
     PAGES = {875--892},
      ISSN = {0020-9910,1432-1297},
   MRCLASS = {37A30 (05D10 11B30 37A05)},
  MRNUMBER = {3418246},
MRREVIEWER = {Vladimir\ S.\ Anashin},
       DOI = {10.1007/s00222-015-0579-7},
       URL = {https://doi.org/10.1007/s00222-015-0579-7},
}

@article {decomposition,
    AUTHOR = {Le, Anh N. and Moreira, Joel and Richter, Florian K.},
     TITLE = {A decomposition of multicorrelation sequences for commuting
              transformations along primes},
   JOURNAL = {Discrete Anal.},
  FJOURNAL = {Discrete Analysis},
      YEAR = {2021},
     PAGES = {Paper No. 4, 27},
      ISSN = {2397-3129},
   MRCLASS = {37A44 (11B30 37A25)},
  MRNUMBER = {4274457},
MRREVIEWER = {Bryna\ Kra},
       DOI = {10.19086/da},
       URL = {https://doi.org/10.19086/da},
}

@article {FranKucajoint,
    AUTHOR = {Frantzikinakis, Nikos and Kuca, Borys},
     TITLE = {Joint ergodicity for commuting transformations and
              applications to polynomial sequences},
   JOURNAL = {Invent. Math.},
  FJOURNAL = {Inventiones Mathematicae},
    VOLUME = {239},
      YEAR = {2025},
    NUMBER = {2},
     PAGES = {621--706},
      ISSN = {0020-9910,1432-1297},
   MRCLASS = {37A44 (05D10 11B30 28D05)},
  MRNUMBER = {4850605},
MRREVIEWER = {Ryo\ Moore},
       DOI = {10.1007/s00222-024-01313-w},
       URL = {https://doi.org/10.1007/s00222-024-01313-w},
}

@article {FranKuca,
    AUTHOR = {Frantzikinakis, Nikos and Kuca, Borys},
     TITLE = {Degree lowering for ergodic averages along arithmetic
              progressions},
   JOURNAL = {J. Anal. Math.},
  FJOURNAL = {Journal d'Analyse Math\'ematique},
    VOLUME = {154},
      YEAR = {2024},
    NUMBER = {1},
     PAGES = {199--253},
      ISSN = {0021-7670,1565-8538},
   MRCLASS = {37A30},
  MRNUMBER = {4846312},
MRREVIEWER = {Song\ Shao},
       DOI = {10.1007/s11854-024-0347-y},
       URL = {https://doi.org/10.1007/s11854-024-0347-y},
}

@article {FranHost,
    AUTHOR = {Frantzikinakis, Nikos and Host, Bernard},
     TITLE = {Higher order {F}ourier analysis of multiplicative functions
              and applications},
   JOURNAL = {J. Amer. Math. Soc.},
  FJOURNAL = {Journal of the American Mathematical Society},
    VOLUME = {30},
      YEAR = {2017},
    NUMBER = {1},
     PAGES = {67--157},
      ISSN = {0894-0347},
       DOI = {10.1090/jams/857},
       URL = {https://doi.org/10.1090/jams/857},
}

@article {Fran,
    AUTHOR = {Frantzikinakis, Nikos},
     TITLE = {Some open problems on multiple ergodic averages},
   JOURNAL = {Bull. Hellenic Math. Soc.},
  FJOURNAL = {Bulletin of the Hellenic Mathematical Society},
    VOLUME = {60},
      YEAR = {2016},
     PAGES = {41--90},
   MRCLASS = {37A30 (05D10 11B30 37A45)},
  MRNUMBER = {3613710},
MRREVIEWER = {El Houcein El Abdalaoui},
}

@article {BrietGreen,
    AUTHOR = {Bri\"{e}t, Jop and Green, Ben},
     TITLE = {Multiple correlation sequences not approximable by
              nilsequences},
   JOURNAL = {Ergodic Theory Dynam. Systems},
  FJOURNAL = {Ergodic Theory and Dynamical Systems},
    VOLUME = {42},
      YEAR = {2022},
    NUMBER = {9},
     PAGES = {2711--2722},
      ISSN = {0143-3857},
   MRCLASS = {11B30 (11B25)},
  MRNUMBER = {4461688},
       DOI = {10.1017/etds.2021.66},
       URL = {https://doi.org/10.1017/etds.2021.66},
}

@article {MCFranHost,
    AUTHOR = {Frantzikinakis, Nikos and Host, Bernard},
     TITLE = {Weighted multiple ergodic averages and correlation sequences},
   JOURNAL = {Ergodic Theory Dynam. Systems},
  FJOURNAL = {Ergodic Theory and Dynamical Systems},
    VOLUME = {38},
      YEAR = {2018},
    NUMBER = {1},
     PAGES = {81--142},
      ISSN = {0143-3857},
   MRCLASS = {37A25 (11K31)},
  MRNUMBER = {3742539},
MRREVIEWER = {Song Shao},
       DOI = {10.1017/etds.2016.19},
       URL = {https://doi.org/10.1017/etds.2016.19},
}

@article {BHK,
    AUTHOR = {Bergelson, Vitaly and Host, Bernard and Kra, Bryna},
     TITLE = {Multiple recurrence and nilsequences},
      NOTE = {With an appendix by Imre Ruzsa},
   JOURNAL = {Invent. Math.},
  FJOURNAL = {Inventiones Mathematicae},
    VOLUME = {160},
      YEAR = {2005},
    NUMBER = {2},
     PAGES = {261--303},
      ISSN = {0020-9910},
   MRCLASS = {37A30 (05D10 28D05 37A05)},
  MRNUMBER = {2138068},
MRREVIEWER = {Randall McCutcheon},
       DOI = {10.1007/s00222-004-0428-6},
       URL = {https://doi.org/10.1007/s00222-004-0428-6},
}

@article {Leibman2,
    AUTHOR = {Leibman, A.},
     TITLE = {Nilsequences, null-sequences, and multiple correlation
              sequences},
   JOURNAL = {Ergodic Theory Dynam. Systems},
  FJOURNAL = {Ergodic Theory and Dynamical Systems},
    VOLUME = {35},
      YEAR = {2015},
    NUMBER = {1},
     PAGES = {176--191},
      ISSN = {0143-3857},
   MRCLASS = {37A30 (28D05)},
  MRNUMBER = {3294297},
MRREVIEWER = {El Houcein El Abdalaoui},
       DOI = {10.1017/etds.2013.36},
       URL = {https://doi.org/10.1017/etds.2013.36},
}

@article {Leibman1,
    AUTHOR = {Leibman, A.},
     TITLE = {Multiple polynomial correlation sequences and nilsequences},
   JOURNAL = {Ergodic Theory Dynam. Systems},
  FJOURNAL = {Ergodic Theory and Dynamical Systems},
    VOLUME = {30},
      YEAR = {2010},
    NUMBER = {3},
     PAGES = {841--854},
      ISSN = {0143-3857},
   MRCLASS = {37A30 (28D05 37A05 37A45)},
  MRNUMBER = {2643713},
MRREVIEWER = {Bryna Kra},
       DOI = {10.1017/S0143385709000303},
       URL = {https://doi.org/10.1017/S0143385709000303},
}

@article{shalom2,
AUTHOR = {Shalom, Or}, 
TITLE= {Multiple ergodic averages in abelian groups and Khintchine type recurrence},
JOURNAL = {Trans. Amer. Math. Soc.},
VOLUME = {375},
YEAR = {2022},
PAGES = {2729--2761},
}

@article {cl1,
    AUTHOR = {Conze, Jean-Pierre and Lesigne, Emmanuel},
     TITLE = {Th\'{e}or\`emes ergodiques pour des mesures diagonales},
   JOURNAL = {Bull. Soc. Math. France},
  FJOURNAL = {Bulletin de la Soci\'{e}t\'{e} Math\'{e}matique de France},
    VOLUME = {112},
      YEAR = {1984},
    NUMBER = {2},
     PAGES = {143--175},
      ISSN = {0037-9484},
   MRCLASS = {28D05 (22D40)},
  MRNUMBER = {788966},
MRREVIEWER = {Karl David},
       URL = {http://www.numdam.org/item?id=BSMF_1984__112__143_0},
}

@article {cl2,
    AUTHOR = {Conze, Jean-Pierre and Lesigne, Emmanuel},
     TITLE = {Sur un th\'{e}or\`eme ergodique pour des mesures diagonales},
   JOURNAL = {C. R. Acad. Sci. Paris S\'{e}r. I Math.},
  FJOURNAL = {Comptes Rendus des S\'{e}ances de l'Acad\'{e}mie des Sciences. S\'{e}rie
              I. Math\'{e}matique},
    VOLUME = {306},
      YEAR = {1988},
    NUMBER = {12},
     PAGES = {491--493},
      ISSN = {0249-6291},
   MRCLASS = {22D40 (28D05)},
  MRNUMBER = {939438},
MRREVIEWER = {Pierre Michel},
}

@incollection {cl3,
    AUTHOR = {Conze, Jean-Pierre and Lesigne, Emmanuel},
     TITLE = {Sur un th\'{e}or\`eme ergodique pour des mesures diagonales},
 BOOKTITLE = {Probabilit\'{e}s},
    SERIES = {Publ. Inst. Rech. Math. Rennes},
    VOLUME = {1987},
     PAGES = {1--31},
 PUBLISHER = {Univ. Rennes I, Rennes},
      YEAR = {1988},
   MRCLASS = {28D05},
  MRNUMBER = {989141},
MRREVIEWER = {Nathaniel F. G. Martin},
}

@Article{bergelson1996polynomial,
  Title                    = {{Polynomial extensions of van der Waerden's and Szemer{\'e}di's theorems}},
  Author                   = {V.~Bergelson and A.~Leibman},
  Journal                  = {J.~Amer.~Math.~Soc.},
  Year                     = {1996},
  Number                   = {3},
  Pages                    = {725--753},
  Volume                   = {9}
}

@Article{furstenberg1977ergodic,
  Title                    = {{Ergodic behaviour of diagonal measures and a theorem of Szemer{\'e}di on arithmetic progressions}},
  Author                   = {H.~Furstenberg},
  Journal                  = {J. Anal. Math.},
  Year                     = {1977},
  Pages                    = {204-256},
  Volume                   = {31}
}

@Article{host2005nonconventional,
  Title                    = {{Nonconventional ergodic averages and nilmanifolds}},
  Author                   = {B.~Host and B.~Kra},
  Journal                  = {Ann.~Math.},
  Year                     = {2005},
  Number                   = {1},
  Pages                    = {397-488},
  Volume                   = {161}
}

@Article{szemeredi1975sets,
  Title                    = {{On sets of integers containing no $k$ elements in arithmetic progression}},
  Author                   = {E.~Szemer{\'e}di},
  Journal                  = {Acta.~Arith.},
  Year                     = {1975},
  Pages                    = {199-245},
  Volume                   = {27}
}

@Article{ziegler2007universal,
  Title                    = {{Universal characteristic factors and Furstenberg averages}},
  Author                   = {T.~Ziegler},
  Journal                  = {J.~Amer.~Math.~Soc.},
  Year                     = {2007},
  Pages                    = {53-97},
  Volume                   = {20}
}

@article {ABB,
    AUTHOR = {Ackelsberg, Ethan and Bergelson, Vitaly and Best, Andrew},
     TITLE = {Multiple recurrence and large intersections for abelian group
              actions},
   JOURNAL = {Discrete Anal.},
  FJOURNAL = {Discrete Analysis},
      YEAR = {2021},
     PAGES = {Paper No. 18, 91},
      ISSN = {2397-3129},
       DOI = {10.19086/da},
       URL = {https://doi.org/10.19086/da},
}

@book {EinsiedlerWardbook,
    AUTHOR = {Einsiedler, Manfred and Ward, Thomas},
     TITLE = {Ergodic theory with a view towards number theory},
    SERIES = {Graduate Texts in Mathematics},
    VOLUME = {259},
 PUBLISHER = {Springer-Verlag London, Ltd., London},
      YEAR = {2011},
     PAGES = {xviii+481},
      ISBN = {978-0-85729-020-5},
   MRCLASS = {37A45 (05D10 11J70 11K50 28Dxx 37-01 37D40)},
  MRNUMBER = {2723325},
MRREVIEWER = {Vitaly\ Bergelson},
       DOI = {10.1007/978-0-85729-021-2},
       URL = {https://doi.org/10.1007/978-0-85729-021-2},
}

\end{document}